\documentclass{article}

\usepackage{amsmath,amssymb,amsfonts,amsthm}
\usepackage{mathtools}
\usepackage{graphicx,overpic}
\usepackage{color}
\usepackage{enumerate}
\usepackage{multicol}
\usepackage{mathrsfs}
\usepackage[mathscr]{eucal}
\usepackage[T1]{fontenc}
\usepackage{hyperref,cleveref}
\usepackage[normalem]{ulem}
\usepackage[margin=1.2in]{geometry}

\numberwithin{equation}{section}
\numberwithin{figure}{section}

\theoremstyle{plain}
\newtheorem{thm}{Theorem}[section]
\newtheorem{lemma}[thm]{Lemma}
\newtheorem{prop}[thm]{Proposition}
\newtheorem{corollary}[thm]{Corollary}

\newcounter{theoremalph}

\newtheorem{thmAlph}[theoremalph]{Theorem}

\theoremstyle{definition}
\newtheorem{defn}[thm]{Definition}
\newtheorem{example}[thm]{Example}
\newtheorem{remark}[thm]{Remark}
\newtheorem{conj}[thm]{Conjecture}

\newtheorem{notation}[thm]{Notation}
\newtheorem{question}[thm]{Question}
\newtheorem{construction}[thm]{Construction}

\newcommand{\Z}{\mathbb{Z}}
\newcommand{\R}{\mathbb{R}}
\newcommand{\N}{\mathbb{N}}

\renewcommand{\H}{\mathbb{H}}
\newcommand{\Hy}{\mathbb{H}}

\renewcommand{\P}{\mathbb{P}}

\newcommand{\cB}{\mathcal{B}}
\newcommand{\cC}{\mathcal{C}}

\newcommand{\cF}{\mathcal{F}}

\newcommand{\cH}{\mathcal{H}}

\newcommand{\cL}{\mathcal{L}}

\newcommand{\cN}{\mathcal{N}}

\newcommand{\cP}{\mathcal{P}}

\newcommand{\cS}{\mathcal{S}}
\newcommand{\cT}{\mathcal{T}}

\newcommand{\cW}{\mathcal{W}}
\newcommand{\cX}{\mathcal{X}}
\newcommand{\cY}{\mathcal{Y}}

\newcommand{\G}{\Gamma}

\newcommand{\CAT}{\operatorname{CAT}}

\newcommand{\diam}{\operatorname{diam}}

\newcommand{\Confdim}{\operatorname{Confdim}}

\newcommand{\Hdim}{\operatorname{Hdim}}
\newcommand{\Ima}{\operatorname{Im}}
\newcommand{\Rea}{\operatorname{Re}}
\newcommand{\St}{\operatorname{St}}
\newcommand{\Lk}{\operatorname{Lk}}
\newcommand{\Shapes}{\operatorname{Shapes}}

\newcommand{\p}{\partial}

\newcommand{\ba}{\boldsymbol{a}}

\definecolor{amethyst}{rgb}{0.6, 0.4, 0.8}

\newcommand{\hide}[1]{}

\title{Conformal dimension bounds \\ for certain Coxeter group Bowditch boundaries}
\author{Elizabeth Field \and Radhika Gupta \and Robert Alonzo Lyman \and Emily Stark}

\begin{document}

\maketitle

    \begin{abstract}
        We give upper and lower bounds on the conformal dimension of the Bowditch boundary of a Coxeter group with defining graph a complete graph and  edge labels at least three. The lower bounds are obtained by quasi-isometrically embedding Gromov's round trees in the Davis complex. The upper bounds are given by exhibiting a geometrically finite action on a $\mathrm{CAT}(-1)$ space and bounding the Hausdorff dimension of the visual boundary of this space. Our results imply that there are infinitely many quasi-isometry classes within each infinite family of such Coxeter groups with edge labels bounded from above. As an application, we prove there are infinitely many quasi-isometry classes among the family of hyperbolic groups with Pontryagin sphere boundary. Combining our results with work of Bourdon--Kleiner proves the conformal dimension of the boundaries of hyperbolic groups in this family achieves a dense set in $(1,\infty)$.
    \end{abstract}

    \section{Introduction} \label{sec:intro}

    Conformal dimension is a powerful analytic invariant of a metric space that captures the fractal behavior amongst all deformations of the space. This invariant has made a significant impact on geometric group theory and the study of spaces with coarse negative curvature.
    Pansu~\cite{pansu} introduced conformal dimension to study the geometry and boundaries of the rank-1 symmetric spaces. He computed the conformal dimension of the boundary of each rank-1 symmetric space and consequently proved that non-isometric non-compact rank-1 symmetric spaces are not quasi-isometric. Recently, Carrasco--Mackay~\cite{carrascomackay} characterized the hyperbolic groups without 2-torsion whose boundaries have conformal dimension equal to one, illustrating a remarkable connection between this analytic property of the boundary and the algebraic structure of the group.

    The conformal dimension of a metric space $Z$ is defined to be the infimum of the Hausdorff dimension among all metric spaces quasisymmetric to $Z$; see \Cref{sec:prelims}. Conformal dimension computations are notoriously difficult. For example, the conformal dimension of the standard square Sierpinski carpet is unknown; see~\cite{kwapisz}. The only hyperbolic groups for which the conformal dimension of their boundary has been computed are uniform lattices in isometry groups of the rank-1 symmetric spaces, uniform lattices in isometry groups of certain Fuchsian buildings~\cite{bourdon}, and groups that split as a graph of groups with groups from these two families as vertex groups and with elementary edge groups.
    Nonetheless, there are a plethora of tools to bound the conformal dimension of a metric space. These tools have been utilized to prove there are infinitely many quasi-isometry classes within certain families of hyperbolic groups~\cite{bourdon95, mackay, mackay16}.

    In this paper, we give the first non-trivial bounds on the conformal dimension of the Bowditch boundary of a non-hyperbolic, relatively hyperbolic group pair. Consequently, we prove there are infinitely many quasi-isometry classes within each family of Coxeter groups with defining graph a complete graph with edge labels at least three and a uniform upper bound on the edge labels.

    Let $\cW$ denote the family of Coxeter groups with defining graph a complete graph and with edge labels greater than or equal to three. These groups are familiar and well-studied if the graph has few vertices. Two-generator Coxeter groups are dihedral. Groups in $\cW$ defined by graphs with three vertices are classically studied triangle groups, which act cocompactly by reflections on the Euclidean plane if all edge labels are three and act cocompactly by reflections on the hyperbolic plane otherwise. Groups in $\cW$ defined by graphs with four vertices are Kleinian: they act discretely on 3-dimensional real hyperbolic space by Andreev's Theorem~\cite{andreev, schroeder}. In general, a group in the family $\cW$ contains many of the above well-understood groups as subgroups. Broadly, one hopes to understand the geometry, boundaries, and rigidity properties for the groups in this family $\cW$.

    Some aspects of the geometry of groups in $\cW$ are well-understood. Each group acts geometrically on its Davis--Moussong complex, which is a 2-dimensional $\CAT(0)$ space with each $2$-cell isometric to a regular Euclidean polygon and with each vertex link a complete graph. If the defining graph contains a triangle with all edges labeled three, then the Davis--Moussong complex contains a Euclidean flat tiled by regular Euclidean hexagons. It follows from an observation of Wise~\cite{wise96} that the Davis--Moussong complex has isolated flats in this case (see~\cite{hruska04}). Otherwise, if no such triangle exists, then the group is $\delta$-hyperbolic \cite[Moussong's Theorem; Corollary 12.6.3]{davis}.

    The $\CAT(0)$ boundary of each group in $\cW$ is therefore well-defined by work of Hruska--Kleiner \cite{hruskakleiner}. If the defining graph has three vertices, then the $\CAT(0)$ boundary is a circle. If the defining graph has four vertices, then the $\CAT(0)$ boundary is homeomorphic to the Sierpinski carpet~\cite{ruane-sier}. Haulmark--Hruska--Sathaye~\cite{haulmarkhruskasathaye} proved that when the defining graph has at least five vertices, then the $\CAT(0)$ boundary is homeomorphic to the Menger curve. Their result helped motivate our study of this family $\cW$, as these groups were the first examples of non-hyperbolic $\CAT(0)$ groups with Menger curve $\CAT(0)$ boundary. In particular, these Coxeter groups cannot be distinguished by the topology of their boundaries, and Haulmark--Hruska--Sathaye asked for the quasi-isometry classification in their work~\cite[Question 1.6]{haulmarkhruskasathaye}. A corollary of our main results is the following.

    \begin{corollary}
	   There are infinitely many quasi-isometry classes among each family of groups in $\cW$ with a uniform upper bound on the edge labels in the defining graph.
       In particular, groups in $\cW$ fall into infinitely many quasi-isometry classes.
    \end{corollary}

    The corollary follows from our lower bound on the conformal dimension of the Bowditch boundary of a group in the family $\cW$. Each group $W_\Gamma \in \cW$ is hyperbolic relative to the collection $\cP$ of stabilizers of two-dimensional flats in the Davis--Moussong complex (where $\cP$ is empty if $W_\Gamma$ is hyperbolic). The pair $(\cW_\Gamma,\cP)$ admits a Bowditch boundary $\p (W_\Gamma, \cP)$ that is well-defined up to homeomorphism~\cite{bowditch12}. If $W_\Gamma$ is hyperbolic,  then this boundary is homeomorphic to the Gromov boundary of $W_\Gamma$. The Bowditch boundary admits a family of metrics inducing its topology. Since each group in $\cP$ is not itself nontrivially relatively hyperbolic, the quasisymmetry type of $\p(W_\Gamma, \cP)$ is a quasi-isometry invariant of the group $W_\Gamma$ \cite{behrstockdrutumosher, groff13, mackaysisto24, healyhruska}. Thus, the conformal dimension of $\p(W_\Gamma, \cP)$ is a quasi-isometry invariant of $W_\Gamma$ by definition. We first exhibit lower bounds.

    \begin{thmAlph}\label{thm:intro:lowb}(\Cref{thm:lowerbound}.)
       Let $\Gamma$ be a complete graph with $m \geq  11$ vertices and edge labels $m_{ij} \geq 3$. Let $M = \max{m_{ij}}$. Then,
        \[\Confdim\bigl(\p (W_\G, \cP)\bigr) \geq 1 + \frac{\log \bigl(\lfloor\frac{m-5}{3}\rfloor\bigr)}{\log (2M-1)}.\]
    \end{thmAlph}

   Note that we use the notation $\log x$ to represent the natural logarithm throughout this paper.

    \Cref{thm:intro:lowb} proves that the conformal dimension of the Bowditch boundary tends to infinity within each infinite family of Coxeter groups in $\cW$ with a uniform upper bound on the edge labels.

  We prove \Cref{thm:intro:lowb} by quasi-isometrically embedding Gromov's round trees~\cite[Section 7.C3]{gromov93} in the Davis--Moussong complex for groups in the family $\cW$. A round tree can be viewed as a complex made up of pieces that are each quasi-isometric to an infinite sector in the hyperbolic plane and which are glued together in the pattern of a rooted tree, $T$. When a round tree is quasi-isometrically embedded in a hyperbolic metric space, it produces a system of arcs in the boundary homeomorphic to the product of a Cantor set and an interval. It turns out that the conformal dimension of this product is controlled by the Hausdorff dimension of the Cantor set, which in turn is controlled by the branching of the tree $T$.

    Round trees have been a primary tool for producing lower bounds on the conformal dimension of the boundary of a hyperbolic metric space~\cite{bourdon95, mackay10, mackay, mackay16, frost}.
    The argument given in this paper is notable in that the Davis--Moussong complex is not itself hyperbolic. By carefully avoiding flats in our construction and maintaining convexity, we are able to show that the system of arcs in the boundary of the round tree that we construct will still survive in the Bowditch boundary.

    Hyperbolic groups in the family $\cW$ were previously studied by Bourdon--Kleiner~\cite{bourdonkleiner15}. They used $\ell_p$-cohomology to provide upper bounds on the conformal dimension of the boundary of certain hyperbolic Coxeter groups. Their results imply that if $W_\Gamma \in \cW$ is defined on a complete graph with $m$ vertices and every edge label is at least $M \geq 4$, then
        \[ \Confdim(\p W_\Gamma) \leq 1 + \frac{\log(m-1)}{\log(2M-5)}.\]
    This yields the following corollary, as pointed out by John Mackay.

    \begin{corollary} (\Cref{cor:dense_set}.)
         Among hyperbolic groups $W_\Gamma \in \cW$, $\Confdim(\p W_{\Gamma})$ achieves a dense set of values in $(1, \infty)$.
    \end{corollary}

    We complement the lower bounds on the conformal dimension with the following upper bounds that also (for fixed edge labels) behave like a constant plus a term that is logarithmic in the number of vertices of the defining graph.

    \begin{thmAlph}\label{thm:intro:upb} (\Cref{cor:HausUpperBound}.)
        Let $W_\Gamma \in \cW$ be generated by $m \geq 4$ elements. Let $M = \max_{i \neq j}\{m_{ij}\}$.
           \begin{enumerate}
        \item If $M=3$, then
        \[ \Confdim\bigl(\p (W_\Gamma, \cP)\bigr) \leq 23 + 12\log m.  \]
        \item If $M \geq 4$, then
        \[ \Confdim\bigl(\p (W_\Gamma, \cP)\bigr) \leq  13 + 12 \log m + 19 \log M. \]
        \end{enumerate}
    \end{thmAlph}

    We prove \Cref{thm:intro:upb} by providing an upper bound on the Hausdorff dimension of a metric space quasisymmetric to $\p(W_\Gamma,\cP)$ equipped with a visual metric. The primary difficulty in this setting is that if $X$ is a $\delta$-hyperbolic space, then the permissible visual metric parameters for $\p X$ depend on the explicit hyperbolicity constant. Rather than try to bound $\delta$ in a particular cusped Cayley graph for $(W_\Gamma, \cP)$, we exhibit a $\CAT(-1)$ model geometry for the pair $(W_\Gamma, \cP)$, since a visual parameter is known in this setting.

    \begin{thmAlph}\label{thm:intro:cat-1}(\Cref{thm:gfCAT-1}.)
        The group pair $(W_{\Gamma}, \cP)$ admits a geometrically finite action on a $\CAT(-1)$ space that is quasi-isometric to the cusped Davis--Moussong complex.
    \end{thmAlph}

    The $\CAT(-1)$ model space we construct is inspired by the view that groups in the family $\cW$ are built from Kleinian groups. Indeed, each 4-generator subgroup of $W_\Gamma \in \cW$ is Kleinian. However, the Kleinian structures are not compatible if the edge labels vary. In the simplest case where all edges in the defining graph have label three, each 4-generator subgroup acts on $\Hy^3$ as the group generated by reflections in the faces of an ideal tetrahedron with dihedral angles $\frac{\pi}{3}$. We utilize this geometry as a common building block for our $\CAT(-1)$ model geometry. In particular, we cut this ideal tetrahedron into pieces we call {\it truncated blocks}, and we force compatibility by gluing over the Davis--Moussong complex in a natural way. This glued space is no longer a manifold, but we make modifications to arrive at a model space that is simply connected and locally $\CAT(-1)$.

    We note if $W_\Gamma$ is hyperbolic, then the Davis-Moussong complex already admits a $\CAT(-1)$ metric \cite[Moussong's Theorem; Corollary 12.6.3]{davis}. The $\CAT(-1)$ model space we build is combinatorially a certain thickening of the Davis-Moussong complex, and these two spaces are quasi-isometric.

    With the explicit $\CAT(-1)$ model in hand, it now follows from work of Bourdon~\cite{bourdon96} and Paulin~\cite{paulin} that an upper bound on the Hausdorff dimension can be obtained from an upper bound on the exponential growth rate of an orbit of a point in the $\CAT(-1)$ space. Our bounds come from a careful analysis of the structure of geodesics in the model space, together with geometric computations in 3-dimensional hyperbolic space.

\paragraph{Additional applications.}

    The lower bounds on the conformal dimension of the Bowditch boundary obtained in this paper can be used to exhibit lower bounds on the conformal dimension of the boundary of other hyperbolic Coxeter groups. Indeed, every special subgroup (one defined by the inclusion of an induced subgraph) of a hyperbolic Coxeter group is quasi-convex. Thus, if $W_\Gamma$ is a hyperbolic Coxeter group whose defining graph contains an induced complete graph $\Gamma'$ with all edge labels at least three, then the lower bound on the conformal dimension of $\p W_{\Gamma'}$ yields a lower bound on the conformal dimension of $\p W_{\Gamma}$. In particular, we thank Kevin Schreve for suggesting the following application.

    \begin{thmAlph} \label{thm:pont} (\Cref{thm:pont-text})
        There are infinitely many quasi-isometry classes among hyperbolic Coxeter groups with Pontryagin sphere boundary.
    \end{thmAlph}

    The Pontryagin sphere is a compact, metrizable, nowhere planar, two-dimensional fractal that can be viewed as the inverse limit of an inverse system of tori; see  \cite{jakobsche,swiatkowski-trees, zawislak, fischer, DoubaLeeMarquisRuffoni} for details of its construction. The Pontryagin sphere arises as the visual boundary of hyperbolic pseudo-manifolds, where the link at each point is either a sphere or a higher genus surface. This space is known to arise as the boundary of hyperbolic and $\CAT(0)$ groups, including Coxeter groups; the right-angled case was shown by Fischer~\cite{fischer} and the general case by \'Swi\k{a}tkowski~\cite{swiatkowski-trees}. In particular, the boundary of a hyperbolic Coxeter group whose nerve is a closed surface of genus at least one is homeomorphic to the Pontryagin sphere~\cite[Theorem 2]{swiatkowski-trees}. The quasi-isometry classification and questions of rigidity for hyperbolic groups with Pontryagin sphere boundary are widely open.

    Every hyperbolic group quasi-isometrically embeds in real hyperbolic space $\Hy^n$ for some $n$ by work of Bonk--Schramm~\cite{bonkschramm}. An analogous result in the relatively hyperbolic setting was given by Mackay--Sisto~\cite[Theorem 1.5]{mackaysisto24}. They proved that a group is hyperbolic relative to virtually nilpotent subgroups if and only if it embeds into truncated real hyperbolic space with at most polynomial distortion. In particular, the proof of \cite[Theorem 1.5]{mackaysisto24} shows that there is a quasisymmetric embedding from the Bowditch boundary of the relatively hyperbolic pair to the sphere $S^{n-1}$ with the standard Euclidean metric. We thank John Mackay for pointing out that another consequence of our lower bounds on conformal dimension yields a family of groups, each hyperbolic relative to virtually $\Z^2$ subgroups, which each embed in truncated $\Hy^n$ for some~$n$ and so that the minimal $n$ tends to infinity. We further note that groups in this family act geometrically on a two-dimensional contractible complex and do not split over an elementary subgroup.

    \begin{corollary}
        Let $W_n \in \cW_{\Gamma}$ be the group generated by $n$ elements and with all edge labels equal to three. The embeddings of $W_n$ into truncated real hyperbolic space of dimension $m(n)$ with at most polynomial distortion guaranteed by Mackay--Sisto must satisfy $m(n) \rightarrow \infty$ as $n \rightarrow \infty$.
    \end{corollary}

\paragraph{Further directions.}

    The bounds in \Cref{thm:intro:lowb} and \Cref{thm:intro:upb} together distinguish some pairs of groups in $\cW$ up to quasi-isometry. However, these bounds are not fine enough to give the complete classification, which we conjecture as follows:

    \begin{conj}
        Groups $W_\Gamma, W_{\Gamma'} \in \cW$ are quasi-isometric if and only if they are isomorphic.
    \end{conj}

    We refer the reader to \cite{dani-survey} for a discussion of the quasi-isometric classification for right-angled Coxeter groups. We note that Cashen--Dani--Thomas~\cite{danithomas} give the quasi-isometric classification within the family of right-angled Coxeter groups whose defining graphs are obtained by sufficiently subdividing a complete graph on at least four vertices.

    A long-standing problem in geometric group theory asks whether every hyperbolic group acts geometrically on a $\CAT(0)$ or $\CAT(-1)$ metric space. One can ask the analogous problem for a relatively hyperbolic group pair; see \cite[Section 7.B]{gromov93}, for example. The relative version is open even for Coxeter groups. We note that relatively hyperbolic Coxeter groups are classified in~\cite{capracebuildingsisolated}. We expect the techniques developed in this paper will be useful in addressing the problem for this family. We note that the $\CAT(-1)$ model was crucial for exhibiting an explicit visual metric parameter in the Hausdorff dimension computations.

    \begin{question}
        Does every relatively hyperbolic group pair $(G,\cP)$ admit a geometrically finite action on a $\CAT(-1)$ space that is quasi-isometric to the cusped Cayley graph for $(G, \cP)$? What if $G$ is a Coxeter group?
    \end{question}

    The $\CAT(-1)$ space that we construct is in some sense more singular than appears necessary; see \Cref{rem:upperbounds}. In particular, since each $4$-generator group in $\cW$ is Kleinian, it would be nice to have a $\CAT(-1)$ model geometry which contains the (unique) Kleinian structure for each $4$-generator subgroup as a convex subspace; see \Cref{rem:Kleinian_model}.

    \begin{question} \label{ques:Kleinian}
        If $W_\Gamma \in \cW$, is there a $\CAT(-1)$ space on which $(W_\Gamma, \cP)$ acts geometrically finitely so that each Kleinian subgroup stabilizes a convex subspace isometric to a convex subspace of $\H^3$?
    \end{question}

 Since the difficulty in our construction occurs when trying to glue
 different Kleinian subgroups together, one way to avoid this issue is to insist that all edge labels of $\Gamma$ are equal, say to $M \ge 4$. In this case, Desgroseilliers-Haglund \cite[Lemma 1.2]{DesHaglund} show that $W_\Gamma$ acts as a convex-cocompact reflection group on $\Hy^{v - 1}$ when $\Gamma$ has $v$ vertices.

     We note that the upper bounds given by Bourdon--Kleiner in the hyperbolic setting with all edge labels greater than three are better than ours obtained from Hausdorff dimension bounds. Consequently, we are very interested to know which metric on the boundary of a hyperbolic group in $\cW$ realizes (or nearly realizes) the infimal Hausdorff dimension. We note that in the case of Bourdon's Fuchsian buildings, the optimal metric is not induced from a $\CAT(-1)$ metric~\cite[Theorem 1.2]{bourdon}.

    Finally, we note that \Cref{thm:intro:lowb} omits the cases that the graph has less than eleven vertices. We believe the Davis--Moussong complex in these cases also contains quasi-isometrically embedded round trees, though the construction must be altered slightly from the one given in our paper. We leave these cases to an interested reader to find.

    \paragraph{Outline.}In \Cref{sec:prelims}, we set notation and recall the definitions of conformal dimension, Hausdorff dimension, relatively hyperbolic groups, the Bowditch boundary, and Coxeter groups. We construct a round tree in the Davis-Moussong complex and use it to compute the lower bounds of \Cref{thm:intro:lowb} in \Cref{sec:lowerbound}. We construct the $\CAT(-1)$ model of \Cref{thm:intro:cat-1} in \Cref{subsection:thespaceYm}, and we check link conditions in \Cref{subsection:LC}. Finally, we compute the upper bounds of \Cref{thm:intro:upb} in \Cref{sec:upperbound}.

\paragraph{Acknowledgments.}
The authors are thankful for helpful discussions
with Daniel Groves, Chris Hruska, Kevin Schreve, Genevieve Walsh, and Daniel Woodhouse.
We thank John Mackay, Lorenzo Ruffoni, and Kevin Schreve for comments on a draft of the article.
This collaboration was facilitated by WiGGD 2020,
which was sponsored by the National Science Foundation
under grants DMS-1552234, DMS-1651963 and DMS-1848346.
EF was supported by NSF grants DMS-1840190 and DMS-2103275.
RG was supported by the Department of Atomic Energy, Government of India, under project no.12-R\&D-TIFR-5.01-0500, partially supported by SERB research grant SRG/2023/000123 and an endowment of the Infosys Foundation.
RAL was partially supported by the National Science Foundation
under Award No.\ DMS-2202942. ES was supported by NSF Grant No. DMS-2204339.

    \section{Preliminaries} \label{sec:prelims}

This section sets notation and collects results that will be used in the paper.
These include visual metrics on the boundary of a proper geodesic $\delta$-hyperbolic metric space
in \Cref{subsection:metrics},
relatively hyperbolic groups and the Bowditch boundary in \Cref{subsection:relhyp}, Hausdorff dimension and the critical exponent in \Cref{subsection:hausdim},
and the Davis--Moussong complex for Coxeter groups in \Cref{subsection:Coxeter}.

\subsection{Metrics on the boundary
  of a proper geodesic \texorpdfstring{$\delta$}{delta}-hyperbolic space.}\label{subsection:metrics}

Let $(X,d)$ be a proper geodesic metric space. A \emph{geodesic ray} in $X$ is an isometric embedding $\gamma\colon \left[0, \infty \right) \to X$. We refer the reader to \cite{bridsonhaefliger, buyaloschroeder, mackaytyson} for details on the following notions.

\begin{defn}
  The \emph{visual boundary} (or \emph{boundary}) of $X$, denoted $\p X$, is the collection of geodesic rays in $X$ up to the following notion of equivalence. Geodesic rays $\gamma$ and $\gamma'$ are equivalent if they \emph{fellow travel,} meaning that there exists $D \geq 0$ so that $d\bigl(\gamma(t), \gamma'(t)\bigr) \leq D$ for all $t \in [0,\infty)$.
\end{defn}

  When $X$ is proper, geodesic, and $\delta$-hyperbolic,
  there is a natural topology on $X \cup \partial X$, making it into a compact space with $X$ an open and dense subset. The boundary $\partial X$ is compact in this topology and is metrizable with metrics given as follows.

    Let $p,x,y \in (X,d)$. The {\it Gromov product} of $x$ and $y$ based at $p$ is the quantity
    \[   {(x,y)}_p := \frac{1}{2}\bigl( d(p,x) + d(p,y) - d(x,y)\bigr). \]
    For $\xi,\eta \in \p X$ and $p \in X$, the {\it Gromov product} of $\xi$ and $\eta$ based at $p$ is the quantity
    \[ (\xi,\eta)_p := \sup \liminf_{i,j \to \infty} \, (x_i,y_j)_p, \]
    where the supremum is taken over all sequences $\{x_i\}$ and $\{y_j\}$ converging to $\xi$ and $\eta$, respectively.

\begin{defn}
  A metric $d_a$ on $\partial X$ is called a \emph{visual metric with parameter $a > 1$}
  if for some point $p \in X$,
  there exist positive constants $k_1$ and $k_2$ such that for all $\xi,\eta \in \partial X$,
  \[
    k_1 a^{-{(\xi,\eta)}_p} \le d_a(\xi, \eta) \le k_2 a^{-{(\xi,\eta)}_p}.
  \]
\end{defn}

    Visual metrics exist on the boundary of a proper geodesic $\delta$-hyperbolic metric space, and the possible visual metric parameters depend on $\delta$, which is often impractical to compute. Hausdorff dimension bounds require control on this parameter, and we will use the following result of Bourdon.

\begin{thm}[\cite{bourdon-flot}, Theorem 2.5.1] \label{thm:vis_parameter}
  Suppose $X$ is a proper, complete $\CAT(-1)$ metric space.
  Then for each $p \in X$ and each $\xi, \eta\in \partial X$, the formula
  \[
    d_e(\xi,\eta) = e^{-{(\xi,\eta)}_p}
  \]
  defines a visual metric on $\partial X$ with parameter $e$.
\end{thm}

    The isometry type of the boundary equipped with a visual metric depends on the choices in the construction of the visual metric. Nonetheless, all such metrics sit in the same quasisymmetry class defined as follows.

\begin{defn}
  Let $(Z,d)$ and $(Z',d')$ be metric spaces. A homeomorphism $f:Z \rightarrow Z'$ is a {\it quasisymmetry} if there exists a homeomorphism $\phi \colon \left[0,\infty\right) \to \left[0,\infty\right)$ so that for all $x,y,z \in Z$ with $x \neq z$,
  \[
    \frac{d'\bigl(f(x), f(y)\bigr)}{d'\bigl(f(x), f(z)\bigr)} \le \phi\left( \frac{d(x,y)}{d(x,z)} \right).
  \]
  The spaces $(Z,d)$ and $(Z',d')$ are \emph{quasisymmetric}
  if there exists a quasisymmetry $f\colon Z \to Z'$.
\end{defn}

    We apply the following theorem.

\begin{thm}[\cite{bourdon-flot}, Theorem 1.6.4; \cite{buyaloschroeder}, Theorem 5.2.17] \label{thm_qs}
  Let $X$ and $X'$ be proper geodesic $\delta$-hyperbolic metric spaces, and let $(\partial X,d_a)$ and $(\partial X',d_{a'})$ denote their visual boundaries equipped with visual metrics.
  A quasi-isometry $f\colon X \to X'$ extends to a quasisymmetry
  $(\partial X,d_a) \to (\partial X',d_{a'})$.
\end{thm}

    Thus, the quasisymmetry class of the boundary of a proper geodesic $\delta$-hyperbolic space equipped with a visual metric is a quasi-isometry invariant of the $\delta$-hyperbolic space. We will use the following quasisymmetry invariant in this paper. (Hausdorff dimension is discussed in more detail in \Cref{subsection:hausdim}.)

\begin{defn}
    Let $(X,d)$ be a metric space. The {\it conformal dimension} of $X$ is the infimal Hausdorff dimension of any metric space $(X',d')$ quasisymmetric to $(X,d)$.
\end{defn}

\subsection{Relatively hyperbolic groups
and the Bowditch boundary}\label{subsection:relhyp}

    There are many equivalent definitions of a relatively hyperbolic group pair~\cite{grovesmanning08,hruska10}. We recall the definition of a cusped Cayley graph due to Groves--Manning and note that there are alternative definitions of cusped spaces quasi-isometric to the cusped Cayley graph; see \cite[Appendix A]{grovesmanningsisto}, \cite{healyhruska}.

    \begin{defn}
     Let $\Lambda$ be a graph, metrized so that each edge has length $1$. The {\it combinatorial horoball based on $\Lambda$} is the metric graph $\cH(\Lambda)$ whose vertex set is $\Lambda^{(0)} \times \Z_{\geq 0}$, and with two types of edges:
     \begin{enumerate}
      \item A {\it vertical edge} of length one from $(v,n)$ to $(v,n+1)$ for any $v \in \Lambda^{(0)}$ and any $n \geq 0$; and
      \item For $k \geq 0$, if $v$ and $w$ are vertices of $\Lambda$ so that $0< d_{\Lambda}(v,w) \leq 2^k$, then there is a single {\it horizontal edge} of length $1$ joining $(v,k)$ to $(w,k)$.
     \end{enumerate}
     The {\it depth} of a vertex in $\cH(\Lambda)$ is $D(v,n) = n$, and the inverse image $D^{-1}(n)$ of an integer $n$ is the subgraph of $\cH(\Lambda)$ spanned by all vertices at depth $n$ and is called the {\it horosphere at depth $n$.}
    \end{defn}

    A {\it group pair} $(G,\cP)$ is a group $G$ together with a collection of subgroups $\cP$ of $G$. We will assume throughout that $\cP$ is a finite collection and that $G$ and each group $P \in \cP$ are finitely generated.

    \begin{defn}
        Let $(G, \cP)$ be a group pair. Choose a generating set $S$ for $G$ containing a generating set for each $P \in \cP$. Let $\Lambda_G$ be the Cayley graph for $G$ associated to the generating set $S$, metrized so that each edge has length one. A left coset $gP$ of $P \in \cP$ spans a connected subgraph $\Lambda(gP) \subset \Lambda_G$. The {\it cusped Cayley graph}, $X(G,\cP)$, defined with respect to $S$ is obtained from $\Lambda_G$ by attaching, for each $P \in \cP$ and each coset $gP$, a copy of $\cH\bigl(\Lambda(gP)\bigr)$ by identifying $\Lambda(gP)$ to the horosphere at depth $0$ in $\cH\bigl(\Lambda(gP)\bigr)$.

        A group pair $(G,\cP)$ is a \emph{relatively hyperbolic group pair} if $X(G, \cP)$ is $\delta$-hyperbolic. Every group $G$ is hyperbolic relative to itself or to any finite collection of finite-index subgroups. Such relatively hyperbolic structures are called \emph{trivial}.
    \end{defn}

    \begin{defn}
         Suppose that $(G,\cP)$ is a relatively hyperbolic group pair and each $P \in \cP$ is an infinite group. The \emph{Bowditch boundary}, $\partial (G, \cP)$, is the visual boundary of the cusped Cayley graph $X(G, \cP)$.
    \end{defn}

    Bowditch proved that the Bowditch boundary is independent of the chosen finite generating set~\cite{bowditch12}. We apply the following theorem.

    \begin{thm}[\cite{behrstockdrutumosher}, Theorem 4.8; \cite{groff13}, Theorem 6.3; \cite{mackaysisto24}, Theorem 1.2; \cite{healyhruska}, Theorem 1.2]
     \label{thm:periph_preserved}
        Suppose that $(G, \cP)$ and $(G', \cP')$ are relatively hyperbolic group pairs. If $G$ and $G'$ are quasi-isometric and no group in $\cP \cup \cP'$ is nontrivially relatively hyperbolic, then there exists a quasi-isometry from any cusped Cayley graph for $G$ to any cusped Cayley graph for $G'$.
    \end{thm}

    \Cref{thm:periph_preserved} together with \Cref{thm_qs} yields the following corollary.

    \begin{corollary} \label{cor:quasisymmetry_Bowditch}
        Suppose that $(G, \cP)$ and $(G', \cP')$ are relatively hyperbolic group pairs and their Bowditch boundaries $\p(G, \cP)$ and $\p(G', \cP')$ are equipped with visual metrics. If $G$ and $G'$ are quasi-isometric and no group in $\cP \cup \cP'$ is nontrivially relatively hyperbolic, then there exists a quasisymmetry $\p(G, \cP) \rightarrow \p(G', \cP')$.
    \end{corollary}

    Bowditch~\cite{bowditch12} proved an alternative characterization of relative hyperbolicity that can be expressed in terms of certain convergence group actions. We will refer to the following notions.

    \begin{defn}
        Let $Z$ be a compact metric space with at least three points. An action of a group $G$ on $Z$ is a {\it convergence group action} if the induced action of $G$ on the space of distinct triples of points of $Z$ is proper. A point $z \in Z$ is a {\it conical limit point} if there exists a sequence $\{g_i\}$ in $G$ and a pair of distinct points $z_1 \neq z_2$ in $Z$ so that for every $z' \in Z \setminus \{z\}$, $g_i(z') \to z_1$ and $g_i(z) \rightarrow z_2$. A point $z \in Z$ is a {\it bounded parabolic point} if its stabilizer acts properly and cocompactly on $Z \setminus \{z\}$.

        A convergence group action of $G$ on $Z$ is {\it geometrically finite} if every point of the limit set $\Lambda G \subset Z$ is either a conical limit point or a bounded parabolic point. The stabilizers of the bounded parabolic points are called {\it maximal parabolic subgroups}. A proper action of a group $G$ on a proper $\delta$-hyperbolic space $X$ is {\it geometrically finite} if the induced action of $G$ on $\p X$ is a geometrically finite convergence action.
    \end{defn}

    If $(G,\cP)$ is a relatively hyperbolic group pair, then the action of $G$ on any cusped Cayley graph is a geometrically finite convergence action, and $\cP$ is the collection of maximal parabolic subgroups.

\subsection{Hausdorff dimension and the critical exponent}\label{subsection:hausdim}

We begin by recalling the definition of Hausdorff dimension. We refer the reader to the book of Falconer~\cite{falconer} for additional background.

\begin{defn}
  Let $(Z,d)$ be a metric space, $X \subset Z$ a subset, and
  let $\epsilon > 0$ and $s \ge 0$.
  Define
  \[
    \mathcal{H}^s_\epsilon(X) = \inf\left\{ \sum_{i = 1}^\infty {\diam(U_i)}^s \,\,:\,\,
      U_i \subset Z, \, X \subset \bigcup_{i = 1}^\infty U_i, \, \diam(U_i) < \epsilon \right\}.
  \]
  The \emph{$s$-dimensional Hausdorff content} of $X$, denoted $\mathcal{H}^s(X)$,
  is $\lim_{\epsilon \to 0}\mathcal{H}^s_\epsilon(X)$.
  The \emph{Hausdorff dimension} of $X$
  is the quantity
  \[
    \Hdim(X) = \sup \{ s \ge 0 : \mathcal{H}^s(X) = 0\} = \inf \{ s \ge 0 : \mathcal{H}^s(X) = \infty \}.
  \]
\end{defn}

If $X \subset Z$ is countable, then its $s$-dimensional Hausdorff content is zero for all $s \ge 0$. Moreover, if $X$ is countable and $Y \subset Z$ is any subset, then $\Hdim(X \cup Y) = \Hdim(Y)$.

\begin{defn}
  Suppose that $X$ is a proper geodesic $\delta$-hyperbolic metric space
  and $G$ is a discrete group of isometries of $X$.
  Fix a point $x_0 \in X$.
  The \emph{conical limit set of the action of $G$ on $X$}, denoted $\Lambda_c(G, X)$,
  is the set of all points $\xi \in \partial X$
  for which there exists a sequence of group elements $g_n \in G$
  such that the sequence $g_nx_0$ converges to $\xi$ in $X \cup \partial X$
  and there exists a ray $\gamma$ representing $\xi$
  such that the distances $d(g_nx_0, \gamma)$ remain uniformly bounded.
  This definition is independent of the choice of basepoint $x_0$.
\end{defn}

For example, if a hyperbolic group $G$ acts geometrically (i.e. properly discontinuously and cocompactly by isometries) on a proper geodesic metric space $X$,
then $\Lambda_c(G, X) = \partial X$.

\begin{defn}
  Suppose $(X,d)$ is a proper geodesic Gromov-hyperbolic metric space and $G$ is a discrete group of isometries of $X$. Let $x_0 \in X$. The \emph{critical exponent $\delta(G \curvearrowright X)$} of the action of $G$ on $X$ is the infimal $s > 0$ for which the \emph{Poincar\'e series}
  \[
    \sum_{g \in G} e^{-s d(x_0, g.x_0)}
  \]
  converges.
  This definition is independent of the choice of basepoint $x_0$.
\end{defn}

The following result of Paulin relates Hausdorff dimension to the critical exponent. This extends work of Coornaert~\cite{coornaert93} in the setting of hyperbolic groups.

\begin{thm}[\cite{paulinCrit}]\label{thm:Paulin}
    Let $X$ be a proper geodesic $\CAT(-1)$ metric space, and equip the boundary $\p X$ with a visual metric with parameter $e$. Suppose $G$ is a discrete group of isometries of $X$ such that the action of $G$ on $\partial X$ is without a global fixed point.
      Then $\delta(G \curvearrowright X) = \Hdim\bigl(\Lambda_c(G, X)\bigr)$.
\end{thm}

    \begin{remark}
        Let $(G, \cP)$ be a relatively hyperbolic group pair and $\partial (G, \cP)$ its Bowditch boundary. Since the bounded parabolic points of $\p(G, \cP)$ are in one-to-one correspondence with conjugates of the maximal parabolic subgroups of $G$, the Bowditch boundary $\partial(G, \cP)$ is the union of $\Lambda_c(G, X)$ and a countable set. So, the Hausdorff dimension of $\partial(G, \cP)$ is equal to the Hausdorff dimension of $\Lambda_c(G,X)$.
\end{remark}

\subsection{Coxeter groups
and the Davis--Moussong complex}\label{subsection:Coxeter}

    Let $\Gamma$ be a finite simplicial graph with vertex set $V\G = \{s_1, \ldots, s_m\}$ and edge set $E\G$. Assume that each unoriented edge $\{s_i,s_j\} \in E\Gamma$ is labeled by a positive integer $m_{ij} \geq 2$. The \emph{Coxeter group $W_\Gamma$ with defining graph $\Gamma$} is the group given by the presentation
    \[ W_\Gamma = \langle \, s_1, \ldots, s_m \mid s_i^2,\ (s_is_j)^{m_{ij}} \textrm{ for all } s_i \in V\G \textrm{ and } \{s_i,s_j\} \in E\G \, \rangle. \]
    The generators $s_i$ are called \emph{standard} or \emph{Coxeter} generators of $W_\Gamma$. A {\it standard triangle subgroup} of $W_{\Gamma}$ is a subgroup generated by three of the standard generators of $W_{\Gamma}$ whose induced subgraph in $\Gamma$ is a triangle. A {\it special subgroup} of $W_\Gamma$ is a subgroup generated by a full or induced subgraph of $\Gamma$, that is, a subgraph which contains every edge connecting a pair of its vertices. For more on Coxeter groups and their geometry, the reader is referred to Davis~\cite{davis}.

\begin{defn}[The Davis--Moussong complex $X_\Gamma$] \label{const:DMComplex}
    Consider the Cayley graph for $W_\Gamma$ associated to the presentation above. Each of the Coxeter generators has order two, so this Cayley graph has bigons. Collapse each bigon to a single edge to obtain a new graph $X_\Gamma^{(1)}$ on which $W_\Gamma$ acts geometrically. Modify $X_\Gamma^{(1)}$ by gluing on cells for each subset of $V\Gamma$ which generates a finite subgroup of $W_\Gamma$ as follows. The graph $X_\Gamma^{(1)}$ will be the $1$-skeleton of this complex.

    When a subset $S$ of the Coxeter generators generates a finite Coxeter subgroup $W_S$ of $W_\Gamma$, there is a corresponding finite subgraph $\Delta_S$ of $X_\Gamma^{(1)}$ spanned by the orbit of the identity vertex under $W_S$. When $|S| = 2$ the graph $\Delta_S$ is a cycle of length $2 m_{ij}$ for $S = \{s_i, s_j\}$. For each such 2-element subset $S$, glue a $2$-cell to $X_{\Gamma}^{(1)}$ so its boundary is $\Delta_S$, and repeat for each coset of $W_S$ in $W_\Gamma$. The resulting $2$-complex $X_\Gamma^{(2)}$ is simply connected (it is essentially the Cayley 2-complex for the given presentation of $W_\Gamma$). If $|S| = 3$, then the subcomplex spanned by the orbit of $W_S$ in $X_\Gamma^{(2)}$ is a cellulation of the $2$-sphere, and one glues in a $3$-cell for every coset, and this pattern continues until the subsets of Coxeter generators in the finite graph $\Gamma$ which generate finite subgroups of $W_\Gamma$ have been exhausted. Let $X_\Gamma$ denote the resulting complex, called the \emph{Davis--Moussong complex} (or {\it Davis complex}) for the Coxeter group $W_\Gamma$. The complex $X_{\Gamma}$ admits a $W_\Gamma$-equivariant piecewise-Euclidean $\CAT(0)$ metric~\cite{moussong,davisbuildings}. By construction, the action of $W_\Gamma$ on $X_{\Gamma}$ is geometric.
\end{defn}

    This paper concerns a family of Coxeter groups of the following form.

\begin{example}[Large-type Coxeter groups]
        A Coxeter group $W_{\G}$ has {\it large-type} if $m_{ij} \geq 3$ for all edges $\{s_i,s_j\} \in E\G$. In this case, every $3$-generated Coxeter subgroup $W_S \leq W_\G$ is infinite, so the Davis--Moussong complex $X_\G$ is equal to its $2$-skeleton. In the $\CAT(0)$ metric on $X_\G$ alluded to above, each $2$-cell is isometric to a regular Euclidean polygon.
        If the $2$-cell corresponds to distinct Coxeter generators $s_i$ and $s_j$, the polygon has $2m_{ij}$ sides. Its boundary has a labeling of $s_is_j\ldots s_is_j$,
        which is inherited from the Cayley graph. We say the polygon has {\it label} $\{s_i,s_j\}$.
    \end{example}

    The Coxeter groups $W_\G$ considered in this paper are of large-type and have defining graph $\G$ a complete graph $K_m$. In this setting, the link of each vertex in the Davis--Moussong complex is the complete graph $K_m$.

    Each standard generator of a Coxeter group $W_{\Gamma}$ (and each conjugate of a standard generator) acts on $X_\Gamma$ as a reflection: there is a connected, separating codimension-1 subspace fixed pointwise by the action of the generator; the complement of this subspace has two components that are exchanged by the action. Such a fixed-point set is called a \emph{wall} in $X_\Gamma$, and the complementary components are called \emph{halfspaces.} If $W$ is a wall, its \emph{carrier}, $C(W)$, is the set of all polygons which intersect it nontrivially. Note that since the dimension of $X_\G$ is two, each wall $W \subset X_\G$ is a tree. Its intersection with every edge of $X_\G$ is either empty or the midpoint of that edge, and its intersection with each polygon is either empty or a single geodesic segment.

    A subspace $A$ of a geodesic metric space $X$ is \emph{convex} if for each pair of points $a$ and $a' \in A$, some geodesic in $X$ between $a$ and $a'$ is contained in $A$. If every geodesic between $a$ and $a'$ is contained in $A$, then $A$ is \emph{strongly convex.}

\begin{lemma}[\cite{davis}, Lemma 4.5.7] \label{lemma:convex}
  If $W \subset X_{\Gamma}$ is a wall and $C(W)$ its carrier, then $C(W) \cap X_{\Gamma}^{(1)}$ is strongly convex in the $1$-skeleton $X_{\Gamma}^{(1)}$.
  Moreover, $W$ is strongly convex in the complex $X_{\Gamma}$.
\end{lemma}

    We will use the following geodesic criterion, which follows immediately from the lemma above.

\begin{corollary} \label{cor:geodesic}
    An edge path $\gamma$ in $X_{\Gamma}^{(1)}$ is a geodesic if and only if $\gamma$ meets each wall it intersects nontrivially exactly once.
\end{corollary}

We complete this section by proving \Cref{thm:pont}.

\begin{thm} \label{thm:pont-text}
        There are infinitely many quasi-isometry classes among hyperbolic Coxeter groups with Pontryagin sphere boundary.
    \end{thm}

\begin{proof}
    The Gromov boundary of a hyperbolic Coxeter group
    whose nerve is a closed surface of genus at least one is homeomorphic to the Pontryagin sphere by \cite[Theorem 2]{swiatkowski-trees}. Recall, the \emph{nerve} of a Coxeter group is the simplicial complex obtained from its defining graph by adding a higher-dimensional simplex whenever its $1$-skeleton corresponds to a finite subgroup of $W_\Gamma$. For a surface nerve, we will want certain $3$-generator Coxeter subgroups to be finite. Note that every three-generator Coxeter subgroup whose defining graph is a triangle with two edges labeled $2$ has this property.

    Let $\Gamma$ be a complete labeled graph on $n \geq 5$ vertices and with edge labels equal to $M \geq 4$. Then, $W_\Gamma$ is hyperbolic. The graph $\Gamma$ embeds in an orientable surface $S$, which has positive genus since $\Gamma$ is nonplanar. Without loss of generality, each complementary component of $S \setminus \Gamma$ is a disk. Add 2-simplices to each complementary region by adding a vertex in the interior of the region and coning off the boundary of the region to the new vertex. Label the new edges by 2. The graph $\Lambda$ triangulates the surface, and $\Gamma$ is an induced subgraph. Further, the corresponding Coxeter group is hyperbolic, its nerve is the surface $S$, and $W_\Gamma$ is a quasi-convex subgroup. Therefore, the lower bounds on the conformal dimension of $\p W_\Gamma$ given in \Cref{thm:lowerbound} tend to infinity as $n \rightarrow \infty$, proving the theorem.
\end{proof}

    \section{Lower bounds on the conformal dimension}\label{sec:lowerbound}

In this section we prove \Cref{thm:lowerbound}, providing a lower bound on the conformal dimension of the Bowditch boundary of a large-type Coxeter group $W_\Gamma$, where $\Gamma$ is a complete graph and $|V\Gamma| \geq 11$. We construct combinatorial round trees in the Davis--Moussong complex and use their structure to obtain lower bounds on the conformal dimension.

    \subsection{Combinatorial round trees} \label{subsec:combroundtree}

    A combinatorial round tree, defined below, is a polygonal complex with visual boundary homeomorphic to the product of a Cantor set and an interval. Its conformal dimension can be bounded from below in terms of the branching data in its construction, as shown by Mackay~\cite{mackay16}, using work of Bourdon~\cite{bourdon95} and Pansu~\cite[Proposition 2.9]{pansu}, \cite[Lemma 6.3]{pansu89b}. In what follows, if $A'\subset A$ is a subcomplex of a polygonal complex $A$, then the {\it star} of $A'$ in $A$, denoted $\St(A')$, is the union of all closed cells in $A$ that nontrivially intersect $A'$.  Let $\N = \Z_{\geq 1}$.

    \begin{defn} \cite[Definition 7.1]{mackay16} \label{defn:RT} (Combinatorial round tree.)
        A polygonal $2$-complex $A$ is a {\it combinatorial round tree} with {\it vertical branching} $V \in \N$ and {\it horizontal branching} at most $H \in \N$ if the following holds. Let $T = \{1,2, \ldots, V\}$.
        Then
            \[ A = \bigcup_{\ba \in T^{\N}} A_{\ba}, \]
        where
        \begin{enumerate}
            \item The complex $A$ has a basepoint $x_0$ contained in the boundary of a unique $2$-cell $A_0' \subset A$.
            \item Each complex $A_{\ba}$ is an infinite planar $2$-complex homeomorphic to a half-plane whose boundary is the union of two rays $L_{\ba}$ and $R_{\ba}$ with $L_{\ba} \cap R_{\ba} = \{x_0\}$.

            \item Let $A_0$ be a union of $2$-cells so that $A_0' \subseteq A_0$ and $A_0$ is homeomorphic to a closed disk.
            For $n>0$, let the {\it round tree at step $n$} be defined as $A_n = \St(A_{n-1})$. Given $\ba = (a_1, a_2, \ldots) \in T^\N$, let $\ba_n = (a_1, \ldots, a_n) \in T^n$. If $\ba,\ba' \in T^\N$ satisfy $\ba_n = \ba'_n$ and $\ba_{n+1} \neq \ba'_{n+1}$, then
                \[ A_n \cap A_{\ba} \subseteq A_{\ba'} \cap A_{\ba} \subseteq A_{n+1} \cap A_{\ba}.\]
            \item Each $2$-cell $P \subseteq A_n$ meets at most $VH$ $2$-cells in $A_{n+1} \backslash A_n$.
        \end{enumerate}
    \end{defn}

    As explained by Mackay~\cite{mackay16}, each complex $A_n$ is the union of $V^n$ distinct planar $2$-complexes $\{A_{\ba_n}\}_{\ba_n \in T^n}$. Each complex $A_{\ba_n}$ is homeomorphic to a disk; see \Cref{figure:Aan}. These complexes in $A_n$ are glued together along their initial subcomplexes in a branching fashion. The boundary of $A_{\ba_n}$ decomposes as a union of three paths: $L_{\ba} \cap A_n$, $R_{\ba} \cap A_n$, and a connected path $E_{\ba_n}$. The complex $A_{\ba_{n+1}}$ is built from $A_{\ba_n}$ by attaching $2$-cells in $V$ strips along the path $E_{\ba_n}$. In each strip, each $2$-cell in $A_{\ba_n}$ is adjacent to at most $H$ new $2$-cells. We will apply the following theorem.

    \begin{thm}[\cite{mackay16}, Theorem 7.2] \label{thm:mackayRT}
        Let $X$ be a hyperbolic polygonal $2$-complex. Suppose there is a combinatorial round tree $A$ with vertical branching $V \geq 2$ and horizontal branching $H \geq 2$. Suppose that $A^{(1)}$, with the natural length metric giving each edge length one, admits a quasi-isometric embedding into $X$. Then
            \[ \Confdim(\p X) \geq 1 + \frac{\log V}{\log H}.\]
    \end{thm}

    \subsection{Round tree construction} \label{subsec:roundtrees}

Fix $\G$, a complete graph with all edge labels $m_{ij}\geq 3$, and let $W_\G$ be the Coxeter group with defining graph $\G$.
In this subsection we build combinatorial round trees in the Davis complex $X_\G$ of $W_\Gamma$ provided $|V\Gamma| \ge 11$.

    \begin{construction}[Combinatorial round trees]
        Fix $V \geq 2$ and suppose $|V\G| = m \geq 3V+5$.
        Let $X_\Gamma$ denote the Davis complex for the group $W_\Gamma$ as above. Label the edges of $X_\Gamma$ with elements from $\{1, \ldots, m\}$ according to a bijection with the $m$ generators of $W_\G$. For a basepoint, let $x_0 \in X_\G^{(0)}$ be the identity vertex.
        Construct a combinatorial round tree $A \subset X_\G$ recursively as follows with vertical branching $V$ and horizontal branching at most $H = 2 \max\bigl\{m_{ij} \,|\, \{s_i,s_j\} \in E\G \bigr\} - 1$.

    \begin{figure}
    \begin{centering}
	\begin{overpic}[width=.6\textwidth, tics=5]{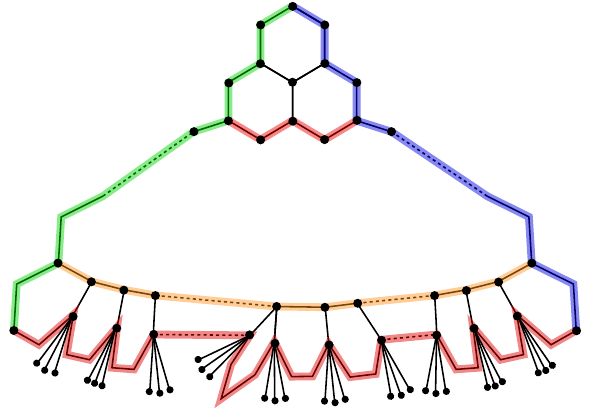}
        \put(51,68.5){\small{$x_0$}}
        \put(38,62){$A_0$}
        \put(15,42){$L_{\ba_n}$}
        \put(78,42){$R_{\ba_n}$}
        \put(-3,9){$E_{\ba_n}$}
        \put(31,21){$E_{\ba_{n-1}}$}
        \put(46,43){$E_0$}
        \put(9,17){\small{$v_1$}}
        \put(56.5,11.5){\small{$v_i$}}
        \put(64.5,13.5){\small{$v_{i+1}$}}
        \put(89,16){\small{$v_k$}}
        \put(15,23.3){\small{$v_1'$}}
        \put(53.5,19.4){\small{$v_i'$}}
        \put(59,20){\small{$v_{i+1}'$}}
        \put(82.5,23.7){\small{$v_k'$}}
    \end{overpic}
	\caption{\small{The $2$-complex $A_{\ba_n}$ in the case that $m_{ij}=3$ and the $2$-cells in the Davis--Moussong complex are hexagons. The initial complex $A_0$ is the union of the three hexagons at the top of the figure, and its outer edge path $E_0$ is drawn in red. The outer edge path of $A_{\ba_n}$ is $E_{\ba_n}$, which is drawn in red, and its adjacent path $E_{\ba_{n-1}}$ is drawn in orange. The internal vertices $v_1, \ldots, v_k$ along $E_{\ba_n}$ are indicated with black circles. The complex is extended by adding $V$ strips of polygons glued to $E_{\ba_n}$ between the new edges emanating from the internal vertices. (As drawn, $V=3$.)}}
	\label{figure:Aan}
    \end{centering}
    \end{figure}
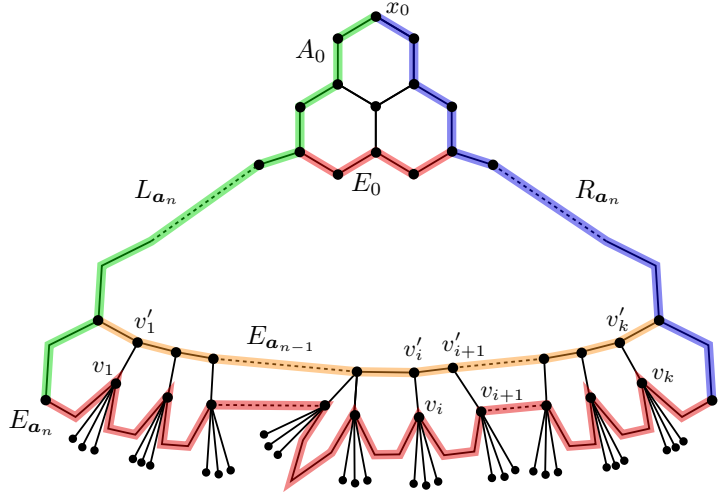

        \noindent {\bf Initial Step.} Let $A_0 \subset X_\G$ be the union of three polygons as shown in \Cref{figure:Aan}. More precisely, take the union of the polygon containing $x_0$ and label $\{1,2\}$, and the polygons labeled by $\{1,3\}$ and $\{2,3\}$ so that all three polygons intersect in a vertex as in the figure. Let $L_0$ and $R_0$ be the left and right edge paths of length four along the boundary of $A_0$ that begin at $x_0$ and are colored green and blue in the figure. Let $E_0$ be the outer edge path of $A_0$ so that the boundary of $A_0$ is the union $L_0 \cup R_0 \cup E_0$.

        \noindent {\bf Inductive Hypothesis.} Let $n \geq 0$ be an integer. Let $T = \{1,2, \ldots, V\}$. Let $T^0 =\{0\}$. Suppose that a $2$-complex $A_n \subset X_\Gamma$ has been constructed satisfying the following induction hypotheses.

        \begin{enumerate}
         \item[(IH1)] (Round Tree, I.) The complex $A_n$ is the union of $V^n$ complexes $\{A_{\ba_n} \, | \, \ba_n \in T^n\}$, where $A_{\ba_n}$ is homeomorphic to a closed disk. Each
         $2$-complex $A_{\ba_n}$ can be viewed as a triangular region based at $x_0$, with left side $L_{\ba_n}$ a geodesic from $x_0$, right side $R_{\ba_n}$ a geodesic from $x_0$, and outer edge path $E_{\ba_n}$. The paths $L_{\ba_n}$ and $R_{\ba_n}$ have length $4+2n$ and are incident to $2+n$ polygons in $A_{\ba_n}$. The {\it left tree} of $A_n$ is the subcomplex $L_n= \bigcup L_{\ba_k}$, and the {\it right tree} of $A_n$ is the subcomplex $R_n = \bigcup R_{\ba_k}$, where the unions are taken over all $k \leq n$ and $\ba_k \in T^k$. The left and right trees at stage~$n$ are homeomorphic to a rooted $(V+1)$-valent tree, truncated at height $n$.

         \item[(IH2)] (Round Tree, II.) At each stage $k \leq n-1$ of the construction and each $\ba_k \in T^k$, there are $V$ new strip subcomplexes $S_{\ba_k}^1, \ldots, S_{\ba_k}^V$ each homeomorphic to $[0,1]^2$ glued along an interval $[0,1]\times \{0\}$ to the edge path $E_{\ba_k}$. Moreover, for each $\ba_k \in T^k$ and $i,j \in \{1, \ldots, V\}$ with $i \neq j$, we have $S_{\ba_k}^i \cap S_{\ba_k}^j = E_{\ba_k}$ and for $\ba_k, \ba_k' \in T^k$ with $\ba_k \neq \ba_k'$ and all $i,j \in \{1, \ldots, V\}$, we have $S_{\ba_k}^i \cap S_{\ba_k'}^j = \emptyset$.  The intersection of each polygon in $S_{\ba_k}^i$ with $E_{\ba_k}$ is either a vertex or an edge.  Each polygon in $A_k$ incident to $E_{\ba_k}$ intersects at most $H$ polygons in each strip $S_{\ba_k}^1, \ldots, S_{\ba_k}^V$ (and therefore at most $VH$ polygons in $A_{k+1} \backslash A_k$). The opposite edge path $[0,1] \times\{1\}$ of $S_{\ba_k}$ forms a new outer edge path $E_{\ba_{k+1}}$ of $A_{\ba_{k+1}}$, where $\ba_k$ coincides with $\ba_{k+1}$ on the first $k$ coordinates, and $E_{\ba_k}$ is called the {\it incident} edge path to $E_{\ba_{k+1}}$.

        \item[(IH3)] (Convexity.) The $1$-skeleton $A_n^{(1)}$ is convex in $X_\Gamma^{(1)}$.

        \item[(IH4)] (To avoid fellow-traveling with a flat.)
        Let $\{t_i\}_{i=1}^R$, where $R = \binom{m}{3}$, be the set of all triples $\{p,q,r\}$ of distinct numbers $1 \leq p,q,r \leq m$. At each stage $k \leq n$, the new polygons added in each strip subcomplex $S^j_{\ba_k}$ with $j$ satisfying $1 \le j \le V$ are not labeled with a pair in the triple~$t_{\ell}$, where $\ell = k \mod R$. (By periodically disallowing every triple of generators to repeat, we ensure the round tree has uniformly bounded intersection with every flat.)
        \end{enumerate}

    \noindent {\bf Inductive Step.} To construct the complex $A_{n+1}$ we will glue $V$ strips to the edge path $E_{\ba_n} \subset A_n$ for each $\ba_n \in T^n$. We then verify the inductive hypotheses hold for the resulting complex.

    We begin by setting notation. Let $\ba_n \in T^n$. Note that $\ba_0 = (0)$ and we let $A_{\ba_0} = A_0$.
    If $n \geq 1$, call a vertex $v$ in the outer edge path $E_{\ba_n}$ {\bf internal}
    if it is adjacent to a vertex in $E_{\ba_{n-1}}$.  If $v \in E_{\ba_n}$ is an internal vertex, let $v' \in E_{\ba_{n-1}}$ denote the unique vertex adjacent to $v$. Let $\{v_i\}_{i=1}^k$ denote the set of internal vertices of $E_{\ba_n}$, labeled in order along $E_{\ba_n}$ from left to right. Let $\{v_i'\}_{i=1}^k \subset E_{\ba_{n-1}}$ be the corresponding adjacent vertices. See \Cref{figure:Aan}. Note that it may be the case that $v_i' = v_{i+1}'$. If $n=0$, the {\bf internal} vertex $v_1$ is the vertex of $E_0$ at distance two from $L_0$ and $R_0$, and $v_1'\subset A_0$ is the vertex at distance one from $L_0, R_0$, and $E_0$.

    We choose $V$ edges in $X_\Gamma \setminus A_{\ba_n}$ with distinct labels that are incident to each internal vertex in $\{v_i\}_{i=1}^k$, while omitting edges with some labels. For instance, since a vertex in $X_{\Gamma}$ does not have two incident edges with same labels, we have to omit the labels of edges in $A_{\ba_n}$ incident to $v_i$.
Before we explain which other labels to omit, we describe how to glue polygons along the path $E_{\ba_n}$ assuming we have chosen $V$ edges $e^i_1, \ldots, e^i_V$ with distinct labels as above, incident at each vertex $v_i$ in $\{v_i\}_{i=1}^k$.

    \begin{figure}
    \begin{centering}
	\begin{overpic}[width=.6\textwidth,tics=5]{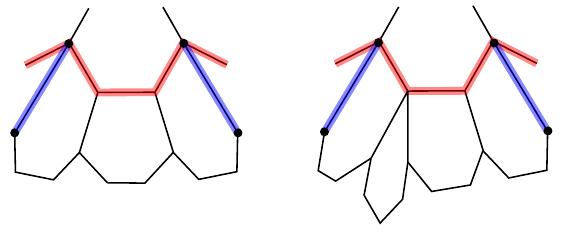}
        \put(1,25){\small{ $e_j^i$}}
        \put(39,25){\small{ $e_j^{i+1}$}}
        \put(56,25){\small{ $e_j^i$}}
        \put(93,25){\small{ $e_j^{i+1}$}}
        \put(10,36){\small{$v_i$ }}
        \put(32,36){\small{$v_{i+1}$ }}
        \put(64.5,36){\small{$v_i$ }}
        \put(87,36){\small{$v_{i+1}$ }}
        \put(7,22){\small{$x$}}
        \put(16,18){\small{$x$}}
        \put(26.5,18){\small{$x$}}
        \put(36,22){\small{$x$}}
        \put(61.5,22){\small{$x$}}
        \put(66,19){\small{$x$}}
        \put(73,18){\small{$y$}}
        \put(81.5,19){\small{$y$}}
        \put(91,22){\small{$y$}}
    \end{overpic}
	\caption{\small{Gluing in polygons during the inductive step. The red line depicts a segment of the outer edge path $E_{a_n}$.}}
	\label{figure:hex_strips}
    \end{centering}
    \end{figure}
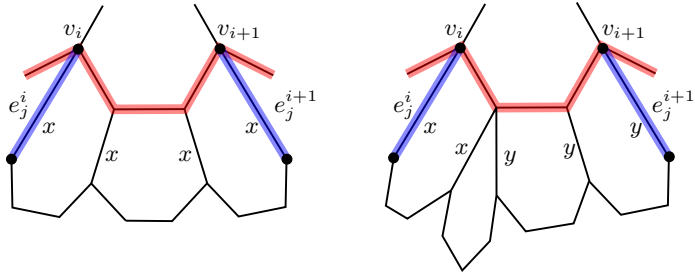

    We glue in polygons along the path $E_{\ba_n}$ in two steps: we glue $V$ strips between the new edges $\{e_r^i\}_{r=1}^V$ and $\{e_r^{i+1}\}_{r=1}^V$ for $1 \leq i \leq k-1$, and then we glue $V$ strips adjacent to each of $\{e^1_r\}_{r=1}^V$ and $\{e^k_r\}_{r=1}^V$ to extend the left and right trees. The  construction illustrated in \Cref{figure:hex_strips} is as follows.
    First, let $i \in \{1, \ldots, k-1\}$, and let $\{e_r^i\}_{r=1}^V$ and $\{e_r^{i+1}\}_{r=1}^V$ be the new edges added to the internal vertices $v_i$ and $v_{i+1}$. Suppose the edges in  $\{e_r^i\}_{r=1}^V$ and $\{e_r^{i+1}\}_{r=1}^V$ have $M$ labels in common for some $M \leq V$. Re-index the edges so that the label of $e_r^i$ is the same as the label of $e_r^{i+1}$ for $r \leq M$. Let $\gamma = \{d_1, \ldots, d_D\}$ be the subpath of $E_{\ba_n}$ from $v_i$ to $v_{i+1}$. By the Induction Hypothesis~(IH2), $D \geq 2$. (Note that $D=2$ occurs only at the first step of the construction.) Let $\ell_s$ be the label of $d_s$. Let $j \in \{1, \ldots, V\}$.

    Suppose first that the labels of $e^i_j$ and $e^{i+1}_j$ agree, and call the label $x$; see the left side of \Cref{figure:hex_strips}. Since we omitted edge labels incident at $v_i$, we have that
    $x \neq \ell_s$ for $1 \leq s \leq D$.
    Glue a polygon labeled by $\{x, \ell_s\}$
    to the edge $d_s$,
    and glue adjacent polygons together along edges labeled by $x$. Call the union of these polygons the {\bf strip} connecting $e^i_j$ and $e^{i+1}_j$.
    Suppose now that the label of $e^i_j$ is $x$, the label of $e^{i+1}_j$ is $y$, and $x \neq y$; see the right side of \Cref{figure:hex_strips}. As before, we have that
    $x \neq \ell_s$ and $y \neq \ell_s$ for $s$ satisfying $1 \leq s \leq D$.
    Glue a polygon labeled by $\{x,\ell_1\}$ to the edge $d_1$, and glue a polygon labeled by
    $\{y, \ell_s\}$ to the edge $d_s$ for $s$ satisfying $2 \leq s \leq D$.
    Glue a polygon labeled $\{x,y\}$ along a vertex adjacent to $d_1$ and $d_2$. Glue adjacent polygons together along edges labeled $x$ or $y$ as indicated. Call the union of these polygons the {\bf strip} connecting $e^i_j$ and $e^{i+1}_j$.

    We now glue $V$ strips along $E_{\ba_n}$ adjacent to $\{e^1_r\}_{r=1}^V$ to extend the left tree. The construction to extend the right tree is analogous. Let $j \in \{1,\ldots, V\}$, and let $x$ be the label of the edge $e^1_j$.
    The vertex $v_1$ is contained in one polygon $P$ intersecting the line $L_{\ba_n}$ as shown in \Cref{figure:Aan}. Let $d_1$ and $d_2$ be the edges of $E_{\ba_n}$ between $v_1$ and $L_{\ba_n}$. Let $\ell_s$ be the label of $d_s$.
    By our choice of new edges, we have that $x \neq \ell_s$ for $s=1,2$.
    Glue a polygon labeled by $\{x,\ell_s\}$ to the edge $d_s$,
    and glue the two polygons together along edges labeled by $x$. Call the union of these polygons the {\it strip} at $e^1_j$. Call the two adjacent edges of the polygon labeled by $\{x,\ell_1\}$ that emanate from $L_{\ba_n}$ an {\bf extension} of $L_{\ba_n}$.

    Next, we explain the labeling of the $V$ new subcomplexes that are the subsets of $A_{n+1}$ attached to
    $A_{\ba_n} \subset A_n$. Suppose $\ba_n = (a_1, \ldots, a_n) \in T^n$. Let $\ba_{n+1}^i = (a_1, \ldots, a_n, i) \in T^{n+1}$ for $i \in \{1, \ldots, V \}$. The strips built at edges incident to the internal vertices naturally glue to form $V$ strips along $E_{\ba_n}$. Arbitrarily label these strips $S^1_{\ba_n}, \ldots, S^V_{\ba_n}$. Let $A_{\ba_{n+1}^i} = A_{\ba_n} \cup S^i_{\ba_n}$. The line $L_{\ba_{n+1}^i}$ is the union of $L_{\ba_n}$ with the extension of $L_{\ba_n}$ that is contained in $S^i_{\ba_n}$. The new right line $R_{\ba_{n+1}^i}$ is similar. The new external line $E_{\ba_{n+1}^i}$ is the subsegment of the boundary of $A_{\ba_{n+1}^i}$ so that the boundary of $A_{\ba_{n+1}^i}$ is $L_{\ba_{n+1}^i} \cup R_{\ba_{n+1}^i} \cup E_{\ba_{n+1}^i}$, and the elements in the union pairwise intersect in a single vertex. The construction is analogous for each $\ba_n \in T^n$.

    The desired vertical branching $V$ and horizontal branching $H = 2 \max\{m_{ij} \,|\, \{s_i,s_j\} \in E\G \} - 1$ are obtained from this construction.  Indeed, let $P$ be a polygon in $A_n$ incident to the path $E_{\ba_n}$. Let $M = \max\{m_{ij} \,|\, \{s_i,s_j\}\in E\Gamma\}$. The polygon $P$ has at most $2M$ sides, and at most $2M-2$ sides are not contained in $A_n$; see \Cref{figure:hex_strips}. There are at most $2M - 2 + 1$ polygons glued to $P$, which gives the value $H$ above. There are $V$ strips built, so each polygon in $A_n$ incident to $E_{\ba_n}$ intersects at most $VH$ polygons in $A_{n+1} \backslash A_n$ as in Inductive Hypothesis (IH2).

    Finally, we are only left to explain which edge labels to omit to conclude the construction of $A_{n+1}$. Let $v_i$ be an internal vertex. Let $v_i'$ be a vertex of $E_{\ba_{n-1}}$ connected to $v_i$ by an edge $e$ and let $S$ be the strip attached to $E_{\ba_{n-1}}$ containing $e$. As mentioned before, we want to avoid the labels of edges in $S$ incident to $v_i$ since a vertex in the Davis complex does not have two edges with same labels incident at a vertex. Note that all such edge labels are contained in the edge labels of edges in $S$ incident at $v_i'$. Depending on whether $v_i'$ is itself an internal vertex or not, there are at most 4 such labels at $v_i'$. In fact, we want to avoid any label of an edge in $A_n$ incident to $v_i'$. This is to avoid the new edges that are being added to come back to the complex along the boundary of a polygon. Excluding the strip $S$, there are $V-1$ other strips and each one contributes at most 2 different labels to edges incident at $v_i'$. Therefore, we need to avoid at most $2V+2$ such labels. Finally we need to avoid three more labels as per Induction Hypothesis~(IH4) to avoid fellow-traveling with a flat, yielding $2V+5$ labels to avoid when making a choice of labels at $v_i$. Since $|V\Gamma| \geq 3V+5$, we can make the required choices.
    \end{construction}

    It follows immediately from construction that if (IH1), (IH2), and (IH4) hold for $A_n$, then they hold for $A_{n+1}$. Therefore, the conditions of \Cref{defn:RT} are satisfied, which gives the following.

    \begin{prop} \label{prop:RT}
        The complex $A$ is a combinatorial round tree with vertical branching $V$ and horizontal branching $H = 2 \max\{m_{ij} \,|\, \{s_i,s_j\} \in E\G \} - 1$. \qed
    \end{prop}

    \begin{prop} \label{prop_convex}
        For each $n \in \Z_{\geq 0}$, if $A_n^{(1)}$ is convex in $X_{\Gamma}^{(1)}$, then $A_{n+1}^{(1)}$ is convex in $X_{\Gamma}^{(1)}$.
    \end{prop}

    \begin{proof}
    Let $p,q \in A_{n+1}^{(1)}$ be two vertices. We will show the existence of a geodesic between $p$ and $q$ in $X_{\Gamma}^{(1)}$ that is contained in $A_{n+1}^{(1)}$, which implies that $A_{n+1}^{(1)}$ is convex.

    Suppose both $p,q \in A_{n+1}^{(1)}\setminus A_n^{(1)}$. The case when one of $p$ or $q$ is in $A_n^{(1)}$ follows by the same argument, and the case when both are follows by hypothesis. Let $p$ and $q$ be contained in polygons $P_p$ and $P_q$ in $A_{n+1}$, respectively. If $P_p$ and $P_q$ intersect non-trivially, then one can verify that a geodesic in $X_{\Gamma}^{(1)}$ is contained in the 1-skeleton of these polygons, and is hence contained in $A_{n+1}^{(1)}$. Henceforth, we suppose $P_p$ and $P_q$ are disjoint. Let $\bar{p}$ and $\bar{q}$ denote the vertices of $A_n^{(1)}$ nearest to $p$ and $q$ respectively. Since $P_p$ and $P_q$ intersect $A_n^{(1)}$, we have $\bar{p} \in P_p$ and $\bar{q} \in P_q$.
    By assumption, there is a geodesic edge path $\gamma \subset A_n^{(1)}$ between $\bar{p}$ and $\bar{q}$. Let $[p,\bar{p}], [\bar{q},q] \subset A_{n+1}^{(1)}$ denote geodesic edge paths from $p$ to $\bar{p}$ and from $\bar{q}$ to $q$, respectively. We will show that the concatenation $[p, \bar{p}] \circ \gamma \circ [\bar{q}, q]$ is a geodesic edge path. To prove this, we will show it crosses every wall at most once and apply \Cref{cor:geodesic}.

    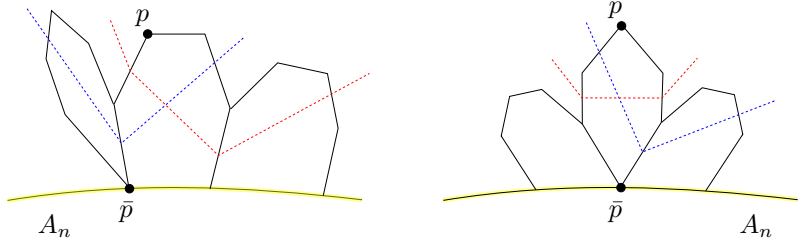
\begin{figure}
    \begin{centering}
	\begin{overpic}[width=.7\textwidth,  tics=5]{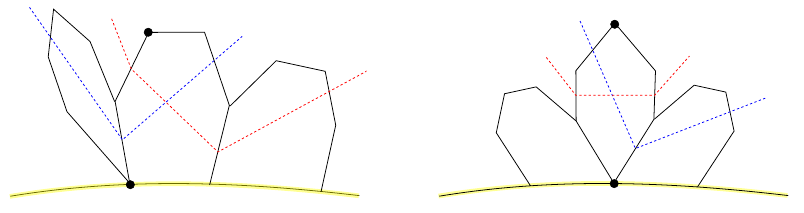}
        \put(17,24){$p$}
        \put(15,0){$\bar{p}$}
        \put(75,25){$p$}
        \put(75,0){$\bar{p}$}
        \put(5,-2){$A_n$}
        \put(91,-2){$A_n$}
    \end{overpic}
	\caption{\small{Configurations of walls locally.}}
	\label{figure:walls}
    \end{centering}
    \end{figure}

    Let $W$ be a wall with $W \cap [p,\bar{p}] \neq \emptyset$. Let $C(W)$ denote the carrier of $W$. By construction, $C(W)$ intersects at most one polygon in $A_{n+1}$ that is adjacent to $P_p$. If such a polygon exists, then it is contained in $A_{n+1} \setminus A_n$. See~\Cref{figure:walls}.

     Suppose towards a contradiction that $W \cap [q,\bar{q}] \neq \emptyset$. Then $\bar{p}, \bar{q} \in C(W)$. Hence, the geodesic~$\gamma$ is contained in $C(W)$ since the carrier is strongly convex in the $1$-skeleton by \Cref{lemma:convex}.
     Thus, the path $[p,\bar{p}] \circ \gamma \circ [\bar{q},q]$ is contained in the connected component of $C(W) \cap A_{n+1}$ containing $p$.
     Since $W$ intersects at most one polygon in $A_{n+1}$ adjacent to $P_p$, this connected component of $C(W) \cap A_{n+1}$ containing $p$ is at most two polygons. So the polygons $P_p$ and $P_q$ must be adjacent, a contradiction. Thus, $W \cap [q,\bar{q}] = \emptyset$.

      Now suppose towards a contradiction that $W \cap \gamma \neq \emptyset$. In principle, $W$ could exit the round tree~$A$ and come back to intersect the path $\gamma$; we will show that this is not possible. Let $P \subset A_i$ be a polygon in $C(W)$ which $\gamma$ intersects, chosen so that $i \in \N$ is minimal. Then by construction, for all $k \in \N$ with $i \leq k \leq n$ there exists $\ba_k \in T^k$ so that $W$ intersects the outer edge path $E_{\ba_k}$. Suppose $W$ intersects $E_{\ba_n}$ in the edge $\{r,r'\}$, which is contained in a polygon $P_r \subset A_{n+1}$. We next argue that $P_p$ and $P_r$ must be adjacent.  Indeed, since $A_n^{(1)}$ is convex, there is a geodesic $\rho$ in $A_n^{(1)}$ joining $\bar{p}$ and $r$. Then, $\bar{p}, r \in C(W)$, and hence
      $\rho \subset C(W)$ by \Cref{lemma:convex}. Hence $[p,\bar{p}] \circ \rho \circ [r,r'] \subset C(W) \cap A_{n+1}$, so $P_p$ and $P_r$ must be adjacent polygons by the same reasoning as in the previous case.
      This contradicts the structure of walls, as $W$ is assumed to cross $P_r$ via an edge in $E_{\ba_n}$.

      Therefore, the path $[p, \bar{p}] \circ \gamma \circ [\bar{q}, q]$ crosses each wall at most once, and is therefore a geodesic. Hence, $A_{n+1}^{(1)}$ is convex in $X_\Gamma^{(1)}$.
     \end{proof}

    \subsection{Lower bound computation} \label{subsec:lowerboundcomp}

    Given the construction of the combinatorial round tree above, we now prove that the boundary of the round tree embeds in the Bowditch boundary. Let $(A,d_A)$ denote the convex subspace $A \subset X_\G$ equipped with the induced metric.

    \begin{lemma} \label{lem:Ahyperbolic}
        The metric space $(A,d_A)$ is $\delta$-hyperbolic.
    \end{lemma}
    We will use the {\it strictly systolic angled complex} condition of Blufstein--Minian~\cite{blufsteinminian} to prove the lemma. Recall, a simplicial complex $X$ is called \emph{angled}, if there is a nonnegative weight function on the corners of 2-simplices of $X$. A weight of a corner $v$ of a 2-simplex can also be thought of as a weight of the corresponding edge in the link of $v$. The complex $X$ is {\it locally $2\pi$-large if} every $2$-full cycle in every vertex link has angular length $\geq 2\pi$ (where a simple cycle $\sigma$ of edge length greater than three in the link of a vertex $v \in X$ is {\it $2$-full} if there is no edge in the link of $v$ that connects two vertices having a common neighbor in $\sigma$).
    A simplicial complex $X$ is {\it $3$-flag} if whenever $X$ has three faces of a tetrahedron, then the whole tetrahedron is in $X$.
    An angled complex $X$ is called \emph{strictly systolic} if $X$ is simply connected, locally $2\pi$-large, $3$-flag and the sum of weights associated to each 2-simplex in $X$ is strictly less than $\pi$. A strictly systolic angled complex is $\delta$-hyperbolic \cite[Corollary 2.10]{blufsteinminian}.

    \begin{proof}[Proof of \Cref{lem:Ahyperbolic}]
    Let $A' \subset A$ be the union of the polygons in $A$ that are disjoint from the left tree and the right tree of $A$ as defined in Inductive Hypothesis (IH1).
    We show that the complex $A'$ can be subdivided to a simplicial complex which is a strictly systolic angled complex.

    Let $P$ be a polygon in $A'$, and suppose $P$ is contained in the planar subspace $A_{\ba} \subset A$ for some $\ba \in T^{\N}$. By construction, $P$ has an internal vertex $u$ that connects back to the previous strip and $u$ has valence 4 in $A_{\ba}$. See the Inductive Step in \Cref{subsec:roundtrees} for more about internal vertices. (In fact, $P$ has two internal vertices.) Divide $P$ into triangles by adding diagonals based at $u$.

    We now assign angles as follows. Recall that the Davis complex is $\CAT(0)$.
    Therefore, if we assign to each corner of each triangle the angle that that triangle has in the Euclidean metric,
    the resulting angled complex will have the property that links of vertices are $2\pi$-large.
    However, this angled complex fails to have the property that each $2$-simplex
    has angle sum strictly smaller than $\pi$; because each triangle is Euclidean,
    we have equality instead. Nonetheless, because $u$ has valence 4 in $A_{\ba}$,
    the angle sum of the cycles in its link is in fact at least $\frac{8\pi}{3}$, since
    the smallest an internal angle sum in any polygon $P$ could be is $\frac{2\pi}{3}$
    in the case that $P$ is a hexagon.
    Therefore in our angled complex structure,
    we will scale the angles of corners at $u$ by $\frac{3}{4}$
    and leave all other angles what they would be in the Euclidean structure.
    Note that this scaling preserves the property that cycles in the link of $u$
    are $2\pi$-large, but that after scaling, we now have that every $2$-simplex in $P$,
    since it contains a corner at $u$,
    has angle sum strictly smaller than $\pi$.

    Since $A'$ is simply connected and a flag (hence $3$-flag) complex, it is a strictly systolic angled complex. Therefore the complex $A'$, and hence the complex $A$, is $\delta$-hyperbolic.
    \end{proof}

    \begin{lemma} \label{lem:Aboundaryembeds}
        The space $(A,d_A)$ quasi-isometrically embeds in the cusped Cayley graph $X(W_\Gamma, \cP)$. Consequently, $\p A$ embeds in the Bowditch boundary $\p(W_\G, \cP)$.
    \end{lemma}
    \begin{proof}
        The space $A$ is $\delta$-hyperbolic and its 1-skeleton $A^{(1)}$ is a convex subset of the 1-skeleton of the Davis complex $X_\Gamma$ by \Cref{lem:Ahyperbolic} and \Cref{prop_convex}. The 1-skeleton of $A$ is a subset of the Cayley graph for $W_\Gamma$ and is thus a subset of the cusped Cayley graph $X(W_\Gamma, \cP)$. The only shortcuts between points in $A^{(1)}$ occur in a combinatorial horoball of $X(W_\Gamma, \cP)$. The diameter of the intersection of any horoball with $A$ is uniformly bounded by Induction Hypothesis~(IH4). Hence, $A$ quasi-isometrically embeds in the cusped Cayley graph, and $\p A$ embeds in the Bowditch boundary $\p(W_\G, \cP)$.
    \end{proof}

    The desired lower bound on the conformal dimension of the Bowditch boundary follows.

    \begin{thm} \label{thm:lowerbound}
    Let $\Gamma$ be a complete graph with $m \geq  11$ vertices and edge labels $m_{ij} \geq 3$. Let $M = \max{m_{ij}}$. Then
        \[\Confdim(\p (W_\G, \cP)) \geq 1 + \frac{\log (\lfloor\frac{m-5}{3}\rfloor)}{\log (2M-1)}.\]
    \end{thm}
    \begin{proof}
    By \Cref{lem:Aboundaryembeds}, $\p A$ quasisymmetrically embeds in
    $\p(W_\G, \cP)$.
    Therefore, $\Confdim(\p (W_\G, \cP)) \geq \Confdim(\p A)$. Now applying \Cref{thm:mackayRT} in the setting when the hyperbolic polygonal 2-complex is the round tree $A$ itself, we get the desired lower bound on $\Confdim(\partial A)$, since $m=|V\Gamma| \geq 3V+5$ and $H = 2M-1$.
    \end{proof}

    We now use the upper bounds given by Bourdon--Kleiner~\cite{bourdonkleiner15} on the conformal dimension of certain hyperbolic groups in the family $\cW$ to prove conformal dimension achieves a dense set of values for each $M>4$.

    \begin{corollary} \label{cor:dense_set}
        Let $W_{M,m} \in \cW$ denote the family of Coxeter groups defined on a complete graph with $m$ vertices and with edge labels equal to $M$. Then $\Confdim(\p W_{M,n})$ achieves a dense set of values in $(1, \infty)$.
    \end{corollary}
    \begin{proof}
        Fix $M \geq 4$ and $Q \in (0,\infty)$. Let $m = M^Q$.
        Then
            \[1 + \frac{\log \bigl(\lfloor\frac{M^Q-5}{3}\rfloor\bigr)}{\log (2M-1)} \leq \Confdim(\p W_{M,n}) \leq 1 + \frac{\log M^Q - 1}{2M-5}.\]
         The lower bound is given by \Cref{thm:lowerbound}, and the upper bound is due to Bourdon--Kleiner~\cite[Corollary 8.1]{bourdonkleiner15}.
        The limit on both sides as $M \rightarrow \infty$ equals $1 + Q$.
    \end{proof}

    \section{A \texorpdfstring{$\CAT(-1)$}{CAT(-1)} model space}

The aim of this section is to prove the following theorem.

    \begin{thm} \label{thm:gfCAT-1}
        The group pair $(W_{\Gamma}, \cP)$ admits a geometrically finite action on a $\CAT(-1)$ space that is quasi-isometric to the cusped Davis complex.
    \end{thm}

To prove the theorem we construct a $\CAT(-1)$ space, denoted $\cY_{\Gamma}$, as a polyhedral complex that contains both compact and noncompact cells that are each subsets of $\Hy^3$. The complex contains a subspace combinatorially isomorphic to the Davis complex for $W_{\Gamma}$, and the group $W_{\Gamma}$ acts on $\cY_{\Gamma}$ by isometries. The construction is described in \Cref{subsection:thespaceYm}. We verify Gromov's link condition is satisfied and conclude the space $\cY_{\Gamma}$ is $\CAT(-1)$ in \Cref{subsection:LC}.

\subsection{The space \texorpdfstring{$\cY_{\Gamma}$.}{Ym}}\label{subsection:thespaceYm}

We begin by outlining the construction. We first build a space $\hat{\cY}_{\Gamma}$ that the group $W_{\Gamma}$ acts on geometrically. The space $\hat{\cY}_{\Gamma}$ is $\CAT(-1)$ away from its boundary, which is a countable collection of disjoint planes, each tiled by Euclidean triangles. These planes are in bijection with the cosets of the standard triangle subgroups of $W_{\Gamma}$. Depending on the standard triangle subgroups of $W_{\Gamma}$, these planes are either isometric to the Euclidean plane or are quasi-isometric to the hyperbolic plane. In order to satisfy $\CAT(-1)$ condition at points in the boundary of $\hat{\cY}_{\Gamma}$,  we attach cells isometric to a subset of $\Hy^3$ to these planes, and this construction depends on the type of triangle group.

\subsubsection{The space \texorpdfstring{$\hat{\cY}_{\Gamma}$}{Ymhat}}\label{subsection:thespaceYmhat}

    The metric space $\hat{\cY}_{\Gamma}$ is constructed from the following building blocks. See \Cref{figure:blockAndDavisHex}. We will denote the Poincar\'{e} disk model of the hyperbolic 3-space by $\Hy^3$ and the upper half space model by $\mathbb{U}^3$.

  \begin{figure}
    \begin{centering}
	\begin{overpic}[width=.8\textwidth, tics=5]{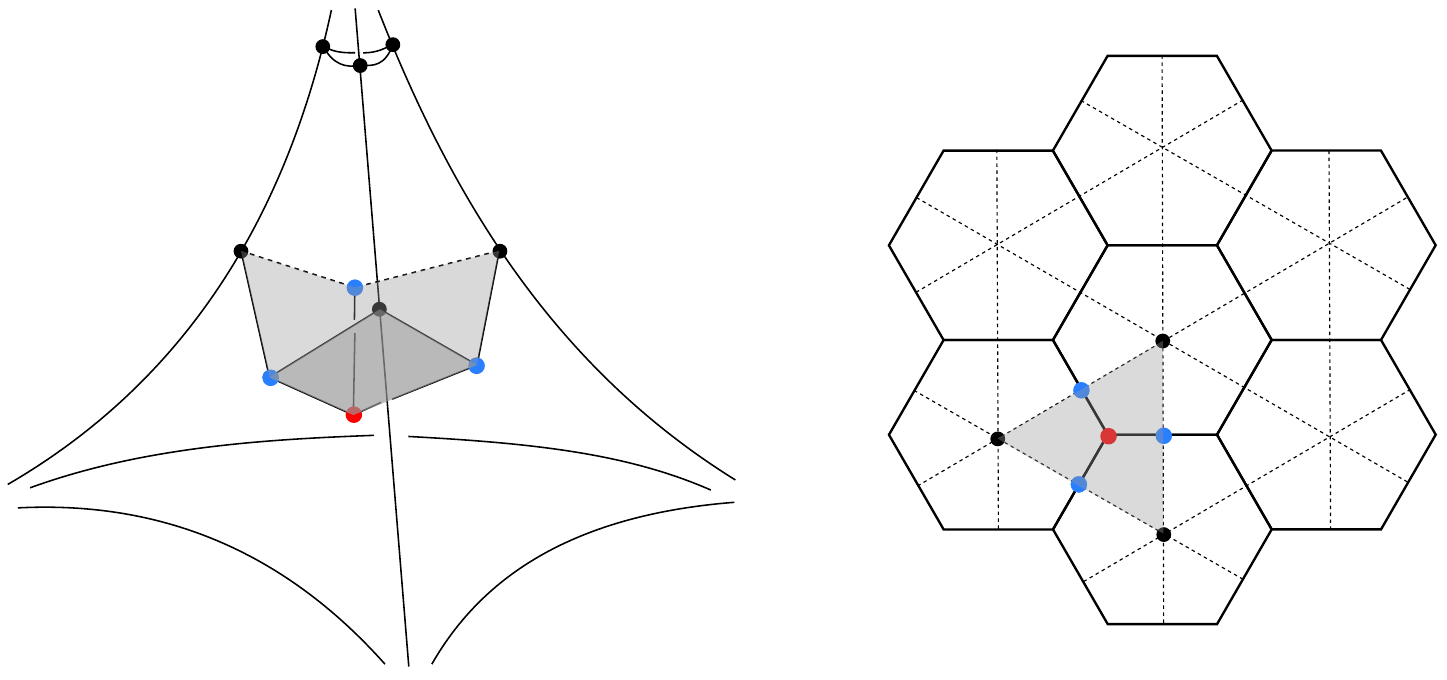}
 	    \put(20.5,17.5){\small{$x_0$}}
	    \put(15,20){\small{$x_i$}}
	    \put(34,21){\small{$x_j$}}
	    \put(22,28.5){\small{$x_k$}}
	    \put(27,26){\small{$x_{ij}$}}
	    \put(12,30){\small{$x_{ik}$}}
	    \put(36,30){\small{$x_{jk}$}}
	    \put(4,40){$\cT$}
	    \put(32,38){\small$\hat{B}(ijk)$}
    \end{overpic}
	\caption{\small{On the left a truncated block is the indicated subspace of the ideal tetrahedron $\cT$ with boundary a triangle, three pentagonal faces, and three shaded kite faces. These kites correspond to the three shaded kites in the portion of the Davis complex illustrated on the right. The union of these kites in the Davis complex form a fundamental domain for the action of the triangle subgroup that stabilizes the plane tiled by hexagons. The length of the geodesic segment connecting $x_0$ to $x_i$ equals $\frac{\log 2}{2}$, as computed in \Cref{lemma:hexside}. }}
	\label{figure:blockAndDavisHex}
    \end{centering}
    \end{figure}
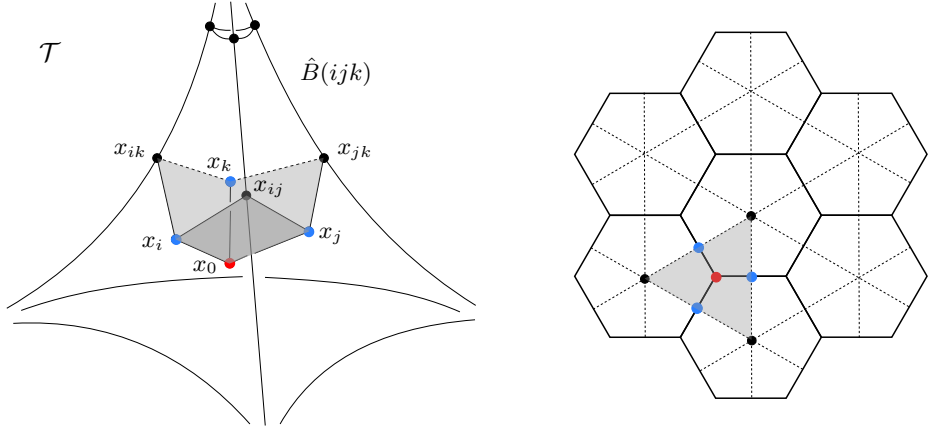

    \begin{notation}[Blocks] \label{nota:blocks}
            Let $\cT$ be an ideal tetrahedron in $\Hy^3$ with all dihedral angles equal to $\frac{\pi}{3}$. Fix distinct $i,j,k,\ell \in \N_+$. Denote the four faces of $\cT$ by $F(i), F(j), F(k), F(\ell)$. Let $E(ij)$ be the edge incident to faces $F(i)$ and $F(j)$, and let $V(ijk)$ be the ideal vertex incident to faces $F(i), F(j)$, and $F(k)$.
            Let $x_0 \in \cT$ be the unique point equidistant from the four faces of $\cT$, and let $x_i \in F(i)$ be the point equidistant from the three sides of $F(i)$. Let $x_{ij} \in E(ij)$ be the intersection of the geodesic plane through $x_0, x_i, x_j$ and the line $E(ij)$. By symmetry, the intersection of the inscribed circle around $x_i$ in $F(i)$ with the side $E(ij)$ is exactly the point $x_{ij}$.

            The {\bf kite} $K(ij)$ is defined to be the convex hull of  $x_0, x_i, x_j, x_{ij}$, which is $2$-dimensional by the choice of~$x_{ij}$. The union of three distinct kites $K(ij) \cup K(jk) \cup K(ik)$ separates the tetrahedron $\cT$ into two components.

            The {\bf cusped block}  $B(ijk)$ is defined to be the closure of the complementary component of the union of kites $K(ij) \cup K(jk) \cup K(ik)$ that contains the ideal vertex $V(ijk)$. The cusped block $B(ijk)$ is thus an ideal polyhedron with three compact faces: the kites, and three non-compact faces, each of which inherits a label from the face, $F(i)$, $F(j)$, or $F(k)$, that contains it.

            Let $\mathbb{U}^3 = \{(w_1,w_2,w_3) \in \R^3 \,|\, w_3>0\}$ be the upper half space model of hyperbolic 3-space with $\p \mathbb{U}^3 = \R^2 \cup \{\infty\}$. Fix an ideal tetrahedron $\cT_0$ in $\mathbb{U}^3$ with vertex set $\{(0,0), (1,0), (\frac{1}{2}, \frac{\sqrt{3}}{2}), \infty\}$, and let \[\cH_0 = \{(w_1,w_2,w_3)\,|\, w_3 \geq y_0\}\] be a horoball based at $\infty$ for some $y_0>0$ large enough so $\cH_0$ does not intersect the face of $\cT_0$ with vertices in $\R^2$. Also denote the horosphere corresponding to $\cH_0$ by $\partial \cH_0$. The intersection of $\cT_0$ with $\partial \cH_0$ is an equilateral triangle of side length one. Let $\Phi \colon \Hy^3 \to \mathbb{U}^3$ be an isometry that maps $\cT$ to $\cT_0$ such that the vertex $V(ijk)$ of $\cT$ maps to the vertex $\infty$ of $\cT_0$. Indeed, $\Phi$ can be constructed by sending two ideal vertices and $V(ijk)$ of $\cT$ to $(0,0)$, $(1,0)$, and $\infty$, respectively, in $\mathbb{U}^3$. The last ideal vertex of $\cT$ has to map to $(\frac{1}{2}, \frac{\sqrt{3}}{2})$ to ensure all the dihedral angles equal $\frac{\pi}{3}$.

            Let $\cH(ijk)$ be the horoball based at $V(ijk)$ given by the $\Phi$-preimage of the horoball $\cH_0$. The height $y_0$ of the horoball $\cH_0$ will be carefully chosen in \Cref{prop:y_0}. The value $y_0$ will play a role in the proof that the constructed space $\cY_\Gamma$ is $\CAT(-1)$ and in the explicit upper bounds on the Hausdorff dimension of its boundary. We note that the value $y_0 = 1.5$ suffices as a special case.

            The {\bf truncated block} $\hat{B}(ijk) \subset B(ijk)$ is defined to be the closure of $B(ijk) \setminus \bigl( B(ijk) \cap \cH(ijk) \bigr)$. It is thus a polyhedron with seven compact faces: the three kites, the triangle boundary from truncation, and three pentagonal faces, each of which inherits a label from the face, $F(i)$, $F(j)$, or $F(k)$, that contains~it.
        \end{notation}

        The isometry types of the kite $K(ij)$, the cusped block $B(ijk)$, and the truncated block $\hat{B}(ijk)$ are independent of the choice of $i,j$ and $i,j,k$, respectively.
        Note that the tetrahedron $\cT$ is tiled by four cusped blocks, where two distinct blocks intersect in a kite, three distinct blocks intersect in a geodesic segment, and the intersection of four distinct blocks is $x_0$. We record a fact regarding the geometry of kites that we will need in our upper bound calculation in \Cref{sec:upperbound}.

    \begin{lemma} \label{lemma:hexside}
        Let $K(ij)$ be the kite specified in \Cref{nota:blocks}. Then $d_{\Hy^3}(x_0, x_i) = \frac{\log 2}{2}$.
    \end{lemma}
    \begin{proof}
        The computation follows from computing the interior angles of the hyperbolic triangle with vertex set $\{x_0, x_i, x_{ij}\}$ and applying the Hyperbolic Law of Cosines. First, the angle $\theta = \angle_{x_0}(x_i,x_j) = \arccos(\frac{-1}{3})$. Indeed, four equally spaced points on a sphere of radius~$1$ span a spherical tetrahedron where each face is a spherical triangle with angles $\frac{2\pi}{3}$; see the left of \Cref{figure:sphere}. Using the Spherical Law of Cosines, one can compute the side length of such a triangle, which is equal to the angle $\theta$. Then, $\angle_{x_0}(x_i, x_{ij}) = \frac{\theta}{2}$. Second, the dihedral angles of $\cT$ are $\frac{\pi}{3}$, hence $\angle_{x_{ij}}(x_i,x_0) = \frac{\pi}{6}$. Indeed, $x_i$ (resp. $x_j$) is the center of the inscribed circle in the ideal triangle containing it and the geodesic joining $x_i$ (resp. $x_j$) to $x_{ij}$ is normal to the side of the ideal triangle containing $x_{ij}$; see the right of \Cref{figure:sphere}. Therefore, the dihedral angle between the two ideal triangles containing $x_j$ and $x_i$ is realized by $\angle_{x_{ij}}(x_i,x_j)$.   Finally, $\angle_{x_i}(x_0, x_{ij}) =\frac{\pi}{2}$ because the nearest point projection of $x_0$ to the ideal triangle in \Cref{figure:blockAndDavisHex} containing $x_i$ is precisely $x_i$. The Hyperbolic Law of Cosines then gives the desired result.
    \end{proof}

\begin{figure}
    \begin{centering}
	\begin{overpic}[width=.8\textwidth, tics=5]{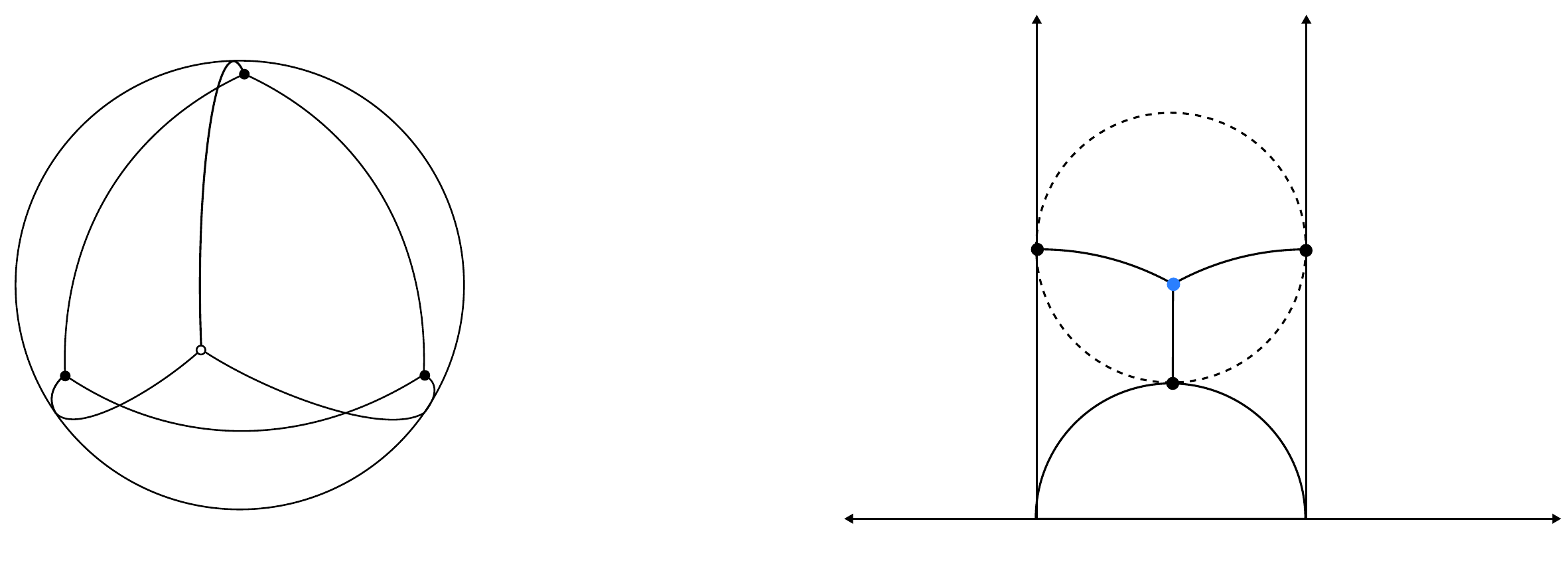}
 	    \put(73.7,19.7){\small{$x_i$}}
        \put(84.5,20){\small{$x_{ij}$}}
        \put(62,20){\small{$x_{ik}$}}
        \put(73,9.5){\small{$x_{i\ell}$}}
        \put(65.5,.5){$0$}
        \put(82.5,.5){$1$}
    \end{overpic}
	\caption{\small{ On the left is a regular tiling of the 2-sphere by triangles with angles $\frac{2\pi}{3}$. On the right is a labeled side of the ideal tetrahedron $\cT$ viewed in the upper half space model of the hyperbolic plane. }}
	\label{figure:sphere}
    \end{centering}
    \end{figure}
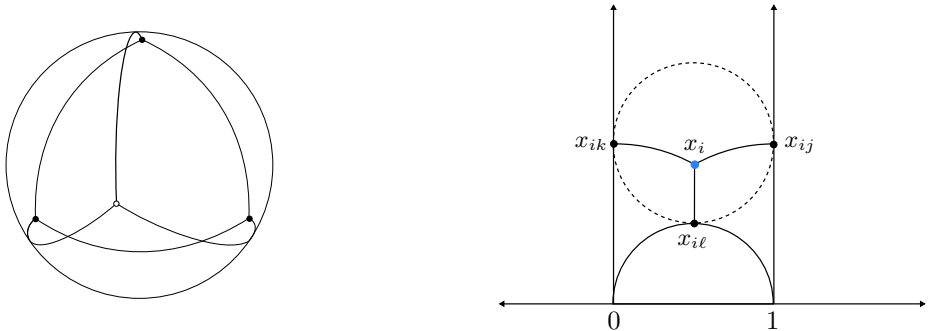

        We now glue together blocks by isometries of their faces to construct a model space $\hat{\cY}_{\Gamma}$ for $W_{\Gamma}$ ``over'' the Davis complex for $W_{\Gamma}$. The construction is modeled on the Davis complex and its action by $W_{\Gamma}$. The group $W_{\Gamma}$ acts on the Davis complex by reflections, with each of the standard generators fixing a wall and exchanging two halfspaces. A fundamental domain for this action is the intersection of all such halfspaces that contain a fixed vertex in the Davis complex. An example of this fundamental domain in the case $\Gamma$ is a triangle with edges labeled three is highlighted in \Cref{figure:blockAndDavisHex}. In general, the fundamental domain for this action is a union of Euclidean kites
        that share a fixed vertex. In the construction below, we alter the geometry of the kites in the Davis complex, and then glue on truncated blocks along their kites. The pentagonal faces of these truncated blocks are then glued if their boundaries intersect.

        \begin{construction}[The space $\hat{\cY}_\Gamma$] \label{const:Ymhat}
            We begin with notation. A $2$-cell in the Davis complex has $2m_{ij}$ sides corresponding to Coxeter generators $s_i$ and $s_j$. There are $m_{ij}$ walls in the Davis complex that intersect this polygon and divide it into $2m_{ij}$ closed regions, each of which we call a {\it Euclidean kite}.  \Cref{figure:blockAndDavisHex} highlights three Euclidean kites in the Davis complex. Note that the isometry type of each Euclidean kite depends on $m_{ij} \in \N$. We call the cellulation of the Davis complex by Euclidean kites the {\it subdivided Davis complex} and denote it by $S X_\Gamma$.

            First, build a 2-dimensional polygonal complex $\cX_{\Gamma}$ by replacing each Euclidean kite in $S X_{\Gamma}$ with a hyperbolic polygon isometric to the kite $K(ij)$ so that the vertex of the kite corresponding to a vertex of the Davis complex corresponds to the point $x_0$ in \Cref{figure:blockAndDavisHex}. Then, the complex $\cX_{\Gamma}$ is combinatorially isomorphic to $S X_\Gamma$ and with all cells hyperbolic and pairwise isometric. Note that this complex is not $\CAT(-1)$; in \Cref{figure:blockAndDavisHex} Gromov's link condition (see  \Cref{subsec:link_cond} for definition) fails for the new geometry at the red vertex. Indeed, the link at the red vertex has a isometrically embedded circle of length less than $2\pi$ and hence it is not $\CAT(1)$ by \cite[Theorem II.5.4]{bridsonhaefliger}. Nonetheless, the group $W_{\Gamma}$ acts geometrically on $\cX_{\Gamma}$ in the same way, combinatorially, as it acts on $X_{\Gamma}$.

            We now glue truncated blocks along their union of kites to collection of kites in the complex $\cX_{\Gamma}$ as follows. Let $v$ be a vertex in $\cX_{\Gamma}$ corresponding to a vertex in the Davis complex, as illustrated in red in \Cref{figure:blockAndDavisHex}. There are $|V\Gamma|$ edges incident to $v$, each of which inherits a label $i \in \{1, \ldots, |V\Gamma|\}$ from the Davis complex. Each pair of distinct $i,j$ yields a kite $k(ij)$ in $\cX_{\Gamma}$ containing~$v$. Each triple of distinct $i,j,k$ with $1 \leq i,j,k \leq |V\Gamma|$ yields a triangle in $\Gamma$ and hence a union of three kites in $\cX_{\Gamma}$
            that contain $v$ and pairwise intersect in an edge. The union of these three kites $k(ij) \cup k(ik) \cup k(jk)$ is isometric to the union of kites $K(ij) \cup K(ik) \cup K(jk) \subset \hat{B}(ijk)$.
            Glue a copy of $\hat{B}(ijk)$ to $\cX_{\Gamma}$ by gluing each kite in $\hat{B}(ijk)$ to the one with the same label in $\cX_{\Gamma}$ by an isometry. Repeat this gluing for each $v \in \cX_{\Gamma}$ corresponding to a vertex in the Davis complex and for each distinct $i,j,k \in \{1, \ldots, |V\Gamma|\}$. The action of $W_{\Gamma}$ on $\cX_{\Gamma}$ extends naturally to a geometric action on this intermediary space.

            Finally, we glue truncated blocks together along subsets of their pentagonal faces when their boundaries intersect as follows. Two truncated blocks $B$ and $B'$ glued to $\cX_{\Gamma}$ intersect if and only if they are glued at vertices $v$ and $w$ in $\cX_{\Gamma}$ whose corresponding vertices in the Davis complex are in the same polygon. Suppose this polygon has edge labels $i$ and $j$. Then, $B$ and $B'$ are the image after gluing of truncated blocks $\hat{B}(ijk)$ and $\hat{B}(ij\ell)$, where $k, \ell \in \{1, \ldots, |V\Gamma|\} \setminus \{i,j\}$.

            If $v$ and $w$ correspond to adjacent vertices in the Davis complex and the edge between them has label $i$, then glue together the pentagonal faces of $B$ and $B'$ labeled $F(i)$ by an isometry. Now suppose $v$ and $w$ correspond to nonadjacent vertices in the same polygon of the Davis complex. In this case, $B$ and $B'$ intersect in a vertex labeled $x_{ij}$ from each block. Glue $B$ and $B'$ together along the edges $E(ij) \cap \hat{B}(ijk)$ and $E(ij) \cap \hat{B}(ij\ell)$ by an isometry.
            Let $\hat{\cY}_{\Gamma}$ denote the resultant space.
        \end{construction}

        \begin{remark} \label{rem:hatYActionMap}
            There is a natural deformation retraction from $\hat{\cY}_{\Gamma}$ to $\cX_{\Gamma}$ with fibers closed intervals. In particular, each pentagonal face in each block in $\hat{\cY}_{\Gamma}$ maps to the union of two edges in the subdivided Davis complex (which are contained in a wall in the Davis complex).
            Furthermore, the action of $W_{\Gamma}$ on $\cX_{\Gamma}$ extends naturally to a geometric action on $\hat{\cY}_{\Gamma}$.
        \end{remark}

\subsubsection{\texorpdfstring{The space $\cY_{\Gamma}$}{The CAT(-1) space}.}

    We now construct a space $\cY_{\Gamma}$ that contains the space $\hat{\cY}_{\Gamma}$ as a subset. All choices in the construction of $\cY_{\Gamma}$ will be $W_{\Gamma}$-equivariant, and $W_{\Gamma}$ will act by isometries on the resultant space.  Recall, the space $\hat{\cY}_{\Gamma}$ is a polyhedral complex with each cell isometric to a truncated block. Each truncated block has a single triangular face isometric to an equilateral Euclidean triangle contained in a horosphere in $\mathbb{U}^3$ and with edge lengths equal to one. We call this face of a truncated block a {\bf boundary face}; the interior of these cells are not glued to other blocks in the construction of $\hat{\cY}_{\Gamma}$. The union of all boundary faces in the complex $\hat{\cY}_{\Gamma}$ is called the {\bf boundary} of $\hat{\cY}_{\Gamma}$. The next lemma follows from the construction of $\hat{\cY}_{\Gamma}$.

    \begin{lemma} \label{lemma:boundarystabilizers}
        The components of the boundary of $\hat{\cY}_{\Gamma}$ are in bijective correspondence with cosets of the standard triangle subgroups of $W_{\Gamma}$. \qed
    \end{lemma}

     We will attach a space to each boundary component of $\hat{\cY}_{\Gamma}$ as follows. If the boundary component corresponds to a Euclidean triangle group, we will attach a horoball in $\Hy^3$ along its boundary horosphere. If the boundary component corresponds to a hyperbolic triangle group, we will attach compact cells defined below that are each isometric to a subset of $\Hy^3$.

    \begin{defn}[Caps] \label{defn:cap}
        Let $\cH$ be a horoball in $\mathbb{U}^3 = \{(w_1,w_2,w_3)\,|\, w_3>0\}$ centered at $\infty$, and let $\partial \cH$ denote the corresponding horosphere. Let $P$ be a regular Euclidean polygon contained in $\partial \cH$. Then there exists a hyperbolic polygon $P'$ in $\mathbb{U}^3$ with vertex set the vertices of $P$. The {\bf cap} $C_P$ over $P$ is the convex hull in $\mathbb{U}^3$ of $P \cup P'$. The {\bf height} of a cap is the maximal value $h$ so that the cap nontrivially intersects the horosphere $w_3 = h$.
        See \Cref{figure:link334} for an example.
    \end{defn}

    \begin{construction}[The space $\cY_\Gamma$] \label{const:Ym}
        Let $\cF$ be a boundary component of $\hat{\cY}_{\Gamma}$. The stabilizer of $\cF$ under the action of $W_{\Gamma}$ is a conjugate of a standard triangle subgroup of $W_{\Gamma}$ by \Cref{lemma:boundarystabilizers}.

        First suppose the triangle group is Euclidean, corresponding to a triangle in $\Gamma$ with all edge labels equal to three. Fix a horoball in $\Hy^3$ and glue its boundary horosphere to $\cF$ by an isometry.

    \begin{figure}
    \begin{centering}
	\begin{overpic}[width=.8\textwidth,tics=5]{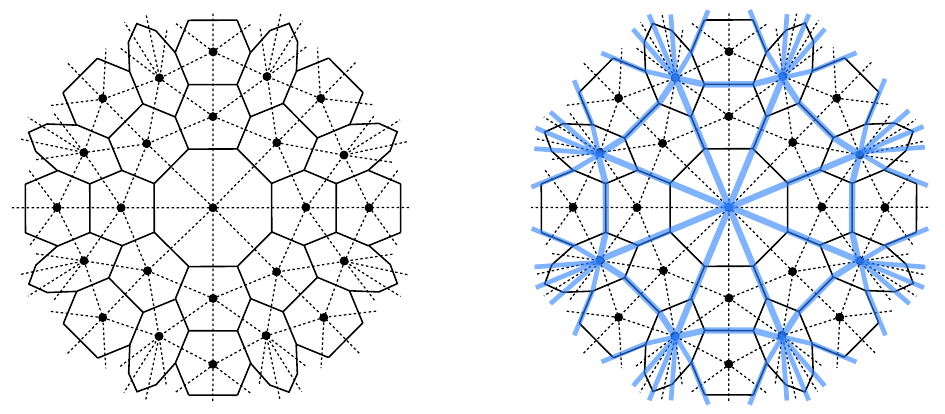}
    \end{overpic}
	\caption{\small{ On the left is the Davis complex for $\Delta(3,3,4)$ and its dual tiling by triangles. On the right in blue is the alternative cellulation by equilateral Euclidean triangles with the key property that the link of every vertex has cone angle strictly greater than $2\pi$. In the construction of the model space, one can visualize the truncated blocks as lying below the dashed triangles, and the added hyperbolic caps as lying above the blue Euclidean polygons. }}
	\label{figure:new_cellulation}
    \end{centering}
    \end{figure}

        Next suppose the triangle group stabilizing $\cF$ is hyperbolic, corresponding to a triangle in $\Gamma$ with edge labels $p,q,r \geq 3$ and with at least one edge label $\geq 4$. The boundary component $\cF$ is homeomorphic to a plane and is tiled by Euclidean triangles. This tiling, call it $T_1(\cF)$,  is the tiling dual to the Davis complex for the Coxeter group $\Delta(p,q,r)$ with defining graph a triangle and edge labels $p,q,r$. See \Cref{figure:new_cellulation}.

        We will now glue caps equivariantly to tiles in a specific tiling of the boundary component $\cF$. The reason we need to consider a tiling of $\cF$ other than $T_1(\cF)$ is as follows. If at least one of $p,r,q$ is equal to three, then there is a vertex in $T_1(\cF)$ with vertex link a $6$-cycle and cone angle $2\pi$. In this case, gluing caps to these triangles would not result in a $\CAT(-1)$ space; indeed, otherwise, one could apply this construction in the case that all edge labels equal three, a contradiction, since such a Coxeter group is not hyperbolic. Thus, if at least one of $p,r,q$ is equal to three we specify an alternative tiling of $\cF$ by Euclidean polygons so that the tiling is preserved by the triangle subgroup and so that each tile is isometric to a regular Euclidean polygon. The key property of this alternative tiling is that the cone angle at each vertex is strictly greater than $2\pi$. See \Cref{figure:new_cellulation}.

        Let  $T_2(\cF)$ denote the following tiling of $\cF$.  If $p,q,r>3$, then $T_2(\cF) = T_1(\cF)$. Otherwise, the vertex set of $T_2(\cF)$ is the vertices of $\cF$ with degree greater than six. Two such vertices are connected by a geodesic path in $\cF$ if the vertices lie in adjacent triangles in $T_1(\cF)$. See \Cref{figure:new_cellulation}.  If $p=q=3$ and $r>3$, then the $2$-cells are triangles with Euclidean edge length $\sqrt{3}$. If $p=3$ and $q,r>3$, then the $2$-cells are hexagons with Euclidean edge length $1$.

        Attach a cap to each $2$-cell $\sigma$ in the tiling $T_2(\cF)$ as follows. If $T$ is a triangle in $T_1(\cF)$, then $T$ is isometric to the intersection of the ideal tetrahedron $\cT_0$ and the horosphere $ \partial \cH_0$ as in \Cref{nota:blocks}.
        Each $2$-cell of $T_2(\cF)$ is contained in a union of such triangles and isometrically embeds in the horosphere $\partial \cH_0$. Thus, we may attach a cap to $\sigma$ as described in \Cref{defn:cap}.

         The triangle subgroup stabilizing $\cF$ preserves the isomorphism type of polygons in the Davis complex, so the stabilizer of $\cF$ acts by isometries on the space obtained after capping all $2$-cells of $T_2(\cF)$. Apply the construction to each boundary component of $\hat{\cY}_{\Gamma}$ to obtain a space $\cY_{\Gamma}$ on which the group $W_{\Gamma}$ acts by isometries.
    \end{construction}

    \begin{prop} \label{prop:QIYAndCuspX}
    The space $\cY_{\Gamma}$ is equivariantly quasi-isometric to the cusped Cayley graph $X(W_{\Gamma}, \cP)$. Therefore, the group $W_{\Gamma}$ admits a geometrically finite action on $\cY_{\Gamma}$. \qed
    \end{prop}
    \begin{proof}
        Cannon--Cooper~\cite[Section 4.2]{cannoncooper92} construct an equivariant quasi-isometry from $\cY_{\Gamma}$ to the {\it augmented Cayley graph} in the case that $\Gamma$ has four vertices and all edge labels equal to $3$. Their argument directly extends to the setting of any graph $\Gamma$ considered here. We refer to their paper for the definition of the augmented Cayley graph. The augmented Cayley graph is equivariantly quasi-isometric to the cusped Cayley graph for any graph $\Gamma$ by \cite[Corollary A.7]{grovesmanningsisto}. The proof of the proposition follows. We leave the details to the reader.
    \end{proof}

    We also record an elementary fact needed to prove that $\cY_{\Gamma}$ is $\CAT(-1)$.

    \begin{lemma} \label{lemma:cYsimplyconn}
       There is a deformation retraction from $\cY_{\Gamma}$ to $\cX_{\Gamma}$. In particular, the space $\cY_{\Gamma}$ is simply connected.
    \end{lemma}
    \begin{proof}
        The deformation retraction from $\hat{\cY}_{\Gamma}$ to $X_{\Gamma}$ described in \Cref{rem:hatYActionMap} with fibers closed intervals extends naturally to the cusps and caps glued to $\hat{\cY}_{\Gamma}$ to yield a deformation retraction $\cY_{\Gamma} \rightarrow \cX_{\Gamma}$. Since $\cX_{\Gamma}$ is combinatorially isomorphic to the Davis complex, it is contractible.
    \end{proof}

    Finally, following the construction of $\cY_\Gamma$, we comment on \Cref{ques:Kleinian} from the Introduction.

    \begin{remark} \label{rem:Kleinian_model}
         Recall \Cref{ques:Kleinian}. If $W_\Gamma \in \cW$, is there a $\CAT(-1)$ space on which $(W_\Gamma, \cP)$ acts geometrically finitely so that each Kleinian subgroup stabilizes a convex subspace isometric to a convex subspace of $\H^3$?
     We suspect the answer to the question is negative. Each 4-generator subgroup has a unique Kleinian structure, coming from reflections across the faces of a hyperbolic polytope. In general, one cannot naturally cut these polytopes into block shapes that can be glued by isometries along their kite faces as in our construction.
    \end{remark}

\subsection{The link condition}\label{subsection:LC}

    The space $\cY_{\Gamma}$ is simply connected by \Cref{lemma:cYsimplyconn}, so it remains to prove that $\cY_{\Gamma}$ is locally $\CAT(-1)$. We will apply Gromov's Link Condition. In \Cref{subsec:link_cond} we describe the variant of the link condition applied. We verify that certain spherical complexes that appear as links are $\CAT(1)$ in \Cref{subsection:otherlinks} and \Cref{subsection:2skeletonnsimplex}.
    \Cref{thm:gfCAT-1} is proved in \Cref{subsec:remaining_links}.

\subsubsection{The link condition for ideal polyhedral complexes} \label{subsec:link_cond}

In this subsection, we state Gromov's Link Condition as in \cite[Theorem II.5.2]{bridsonhaefliger} for a slight generalization of $M^n_{\kappa}$-polyhedral complexes, called `ideal polyhedral complexes' whose cells are not necessarily compact.

Let $M^n_{\kappa}$ denote the $n$-dimensional model space of curvature~$\kappa$. Recall, a {\it convex $M^n_{\kappa}$-polyhedral cell} $C \subset M^n_{\kappa}$ is the convex hull of finitely many points in $M^n_{\kappa}$, and a {\it $M^n_{\kappa}$-polyhedral complex} $K$ is a cell complex whose cells are convex $M^n_{\kappa}$-polyhedral cells, glued by isometries along faces.
We say $K$ satisfies the {\it link condition} if for every vertex $v \in K$, the link complex $\Lk(v,K)$ is $\CAT(1)$. See \cite[Chapter I.7]{bridsonhaefliger} for additional background.

\begin{thm}[\cite{bridsonhaefliger}, Theorem II.5.2]
An $M^n_{\kappa}$-polyhedral complex $K$ with $\Shapes(K)$ finite has curvature $\leq \kappa$, if and only if it satisfies the link condition.
\end{thm}

In this section, we will restrict to $\kappa = -1$. We define a convex \textbf{$M^n_{-1}$-ideal polyhedral cell} to be the convex hull of finitely many points which are allowed to be ideal points (in other words, boundary points) in $M^n_{-1}$. We call a complex $K$ built out of these cells a \textbf{$M^n_{-1}$-ideal polyhedral complex}. Note that the ideal points are not part of the complex $K$. The \emph{vertices of $K$} are all the vertices of the cells that are not ideal points. Such a complex is said to be \emph{locally finite} if for every point $x \in K$, there exists a neighborhood that intersects only finitely many cells of $K$.

Let $x \in K$ and let $\epsilon(x) := \inf\{\epsilon(x,S) \, | \, S\subset K \text{ is an ideal polyhedral cell containing } x\}$, where $\epsilon(x,S) := \inf\{d_S(x,T) \, | \, T \text{ is a face of } S \text{ and } x \notin T\}$. Set $\epsilon(x,S) = \infty$ if $\{x\} = S$. Note that $\epsilon(x)>0$ for every $x \in K$ if $\Shapes(K)$ is finite, irrespective of the cells having ideal vertices or not. The arguments in the next two propositions basically rely on $\epsilon(x)>0$ for all $x \in K$.

\begin{prop} \label{prop:comlenspace}
    An $M^n_{-1}$-ideal polyhedral complex with $\Shapes(K)$ finite is a complete length space. In addition, if $K$ is locally finite, then it is a geodesic metric space.
\end{prop}
\begin{proof}
Since $\epsilon(x)>0$ for every $x \in K$, by \cite[Corollary I.7.10]{bridsonhaefliger} $K$ supports a metric and is a length space with respect to this metric. The rest of the proof for completeness follows as in \cite[Theorem I.7.13]{bridsonhaefliger}. When $K$ is also locally finite, then the Hopf Rinow theorem implies that it is a geodesic metric space.
\end{proof}

\begin{prop} \label{prop:linkcond}
An $M^n_{-1}$-ideal polyhedral complex $K$ with $\Shapes(K)$ finite has curvature $\leq -1$, if and only if for every vertex $v \in K$ the link complex $\Lk(v,K)$ is $\CAT(1)$.
\end{prop}
\begin{proof}
We go through the steps in the proof of \cite[Theorem II.5.2]{bridsonhaefliger} and ensure they hold true when $K$ has ideal polyhedral cells. We need to show that for every point $x \in K$, there exists a neighborhood of $x$ that is $\CAT(-1)$. We will first argue that this is true when $x$ is a vertex of $K$.

\cite[Theorem I.7.39, Theorem I.7.16]{bridsonhaefliger} imply that for every vertex $v$ in $K$ and $\epsilon(v)>0$ defined above, the ball $B(v,\epsilon(v)/2)$ is convex and isometric to the $\epsilon(v)/2$ neighborhood of the cone point in the metric cone $C_{-1}(\Lk(v,K))$ (see \cite[Definition I.5.6]{bridsonhaefliger}). Now Berestovskii's theorem \cite[Theorem 3.14]{bridsonhaefliger} implies that $C_{-1}(\Lk(v,K))$ is $\CAT(-1)$ if and only if $\Lk(v,K)$ is $\CAT(1)$.

Next suppose $x \in K$ is not a vertex. Let $C$ be the unique open cell containing $x$ and let $v$ be a vertex of $C$. Let $x' \in B(v, \frac{\epsilon(v)}{2}) \cap C$. By \cite[Lemma I.7.56]{bridsonhaefliger}, for $\epsilon < \frac{1}{2}\min\{\epsilon(x'), \epsilon(x), \frac{\epsilon(v)}{2}\}$, we have $B(x,\epsilon)\cap B$ is isometric to $B(x', \epsilon) \cap B$ for any closed ideal polyhedron $B$ containing both $x$ and $x'$. Since $B(x',\epsilon)$ is $\CAT(-1)$, we are done.
\end{proof}

\subsubsection{\texorpdfstring{Links at the cap vertices.}{Links at the cap vertices}}
\label{subsection:otherlinks}
    \begin{figure}
    \begin{centering}
	\begin{overpic}[width=.8\textwidth, tics=5]{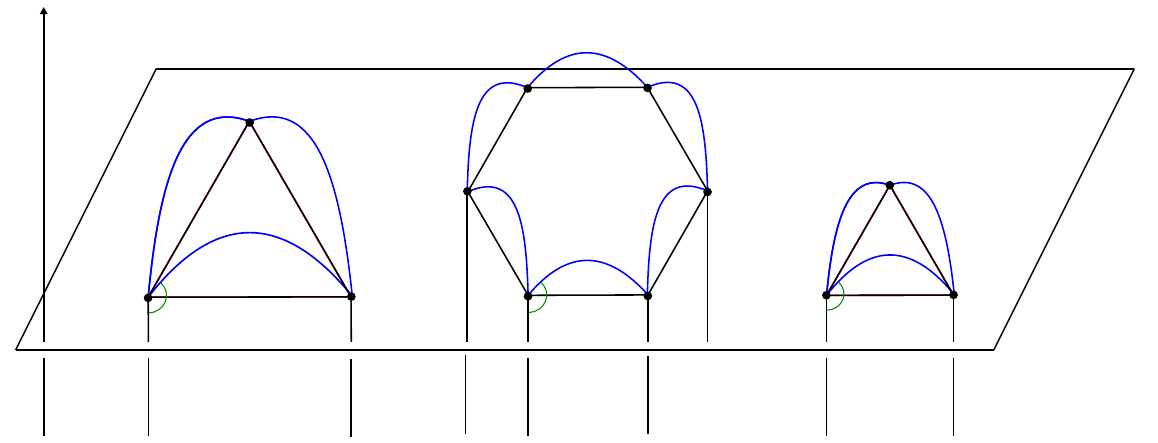}
        \put(-2,30){$\Hy^3$}
        \put(12,27){\small{$C_1$}}
        \put(10,13){\small{$v_1$}}
        \put(15,10){\small{$\theta_1$}}
        \put(19.5,13.5){\small{$\sqrt{3}$}}
        \put(37,27){\small{$C_2$}}
        \put(42,13){\small{$v_2$}}
        \put(47,10){\small{$\theta_2$}}
        \put(50,13.5){\small{$1$}}
        \put(69,20){\small{$C_3$}}
        \put(68,13){\small{$v_3$}}
        \put(73,10){\small{$\theta_3$}}
        \put(76,13.5){\small{$1$}}
        \put(16,36){\small{$\Delta(3,3,r)$}}
        \put(45.5,36){\small{$\Delta(3,q,r)$}}
        \put(71,36){\small{$\Delta(p,q,r)$}}
    \end{overpic}
	\caption{\small{The three isometry types of the caps used to build the space $\cY_\Gamma$. }}
	\label{figure:caps}
    \end{centering}
    \end{figure}
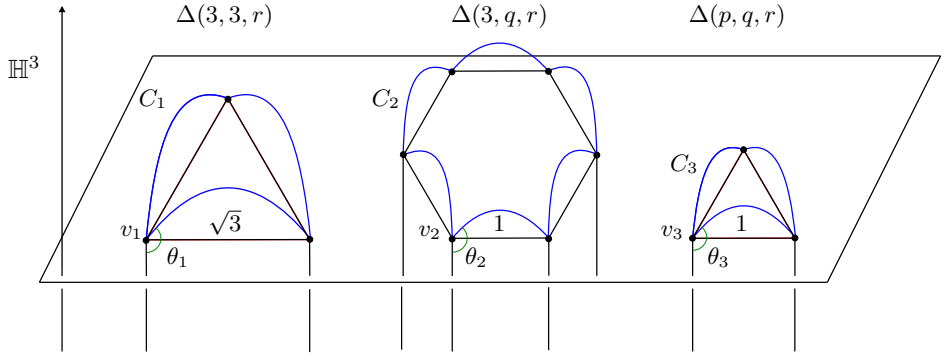

In this section we describe the links in $\cY_\Gamma$ of vertices that lie in a cap and prove these are $\CAT(1)$. To do this, we first describe these links as an abstract spherical complex and provide conditions to ensure this abstract complex is $\CAT(1)$. We then specify the geometry of $\cY_\Gamma$ by giving a suitable constant $y_0$ in the horoball $\cH_0$ as introduced in \Cref{nota:blocks} so that the links of the cap vertices satisfy these conditions.

\begin{notation}[Isometry types of caps and their vertex links]  \label{nota:cap_links}
        Up to isometry, there are three types of caps used in the construction of $\cY_\Gamma$, and they are given as follows. See \Cref{figure:caps}. Let $C_1$, $C_2$, and $C_3$ denote the cap corresponding to the triangle group $\Delta(3,3,r)$, $\Delta(3,q,r)$, and $\Delta(p,q,r)$ with $p,q,r \geq 4$, respectively. Each cap is built over a Euclidean polygon contained in the horosphere $\partial \cH_0$ specified by $w_3 = y_0$ in $\mathbb{U}^3$ as in \Cref{nota:blocks}.
        As explained in \Cref{const:Ym}, the cap $C_1$ is built over an equilateral triangle with edge lengths $\sqrt{3}$. The cap $C_2$ is built over a regular hexagon with edge lengths $1$. Lastly, the cap $C_3$ is built over an equilateral triangle with edge lengths $1$.

        Let $v_i$ be a vertex of the cap $C_i$. Let $\gamma_i$ be a vertical geodesic ray emanating from $v_i$ that is contained in complement of the horoball $\cH_0$. Let $\gamma_i'$ be a geodesic side of the hyperbolic polygon bounding $C_i$ emanating from $v_i$. Let $\theta_i = \angle_{v_i}(\gamma_i, \gamma_i')$. Let $\Sigma_i$ be an isosceles spherical triangle with two sides of length $\theta_i$ and a third side of length denoted $\sigma_i$. Let the angle between the sides of equal length equal $\frac{\pi}{3}$ for $i=1,3$ and equal to $\frac{2\pi}{3}$ for $i=2$.

      The number of caps incident to $v_i$ in the space $\cY_\Gamma$ depends on the   values of $p$, $q$, and $r$ in the hyperbolic triangle group $\Delta(p,q,r)$. Since $C_1$ is built from $\Delta(3,3,r)$, there are $2r$ caps incident to $v_1$. Similarly, there are $q$ or $r$ caps incident to $v_2$, and either $2p$, $2q$, or $2r$ caps incident to $v_3$.  Then, the structure of $\Lk(v_i, \cY_\Gamma)$ has the form $\cL_{\theta,\frac{\pi}{3}, 2r}$, for $i=1,3$ and  $\cL_{\theta, \frac{2\pi}{3}, r}$ for $i=2$, where $\cL_{\theta, \alpha, n}$ is defined as follows.
    \end{notation}

\begin{defn}[Spherical complex $\cL_{\theta, \alpha, n}$]
        For $\frac{\pi}{2} < \theta < \pi$, $0 < \alpha \leq \frac{2\pi}{3}$ and $n$ a positive integer, let $\Delta_{\theta, \alpha}$ denote the isosceles triangle in $S^2$ with two sides of length $\theta$ and with angle $\alpha$ between these sides. The third side of the triangle has length $\sigma = \arccos\left(\cos^2\left(\theta \right)+ \cos(\alpha)\sin^2\left(\theta\right) \right)$ by the Spherical Law of Cosines. Let $\cL_{\theta, \alpha, n}$ denote the spherical complex obtained by taking $n$ copies of $\Delta_{\theta, \alpha}$ and gluing sides of length $\theta$ together cyclically to produce a connected space with $n+1$ vertices as in \Cref{figure:flowerCAT1}.
    \end{defn}

    We will find sufficient condition for $\cL_{\theta,\frac{\pi}{3}, 2r}$ and $\cL_{\theta,\frac{2\pi}{3}, r}$ to be $\CAT(1)$. A similar condition can be obtained for the complex $\cL_{\theta,\alpha, n}$ for appropriate relation between $\alpha$ and $n$.

\begin{figure}
    \begin{centering}
	\begin{overpic}[width=.25\textwidth, tics=5]{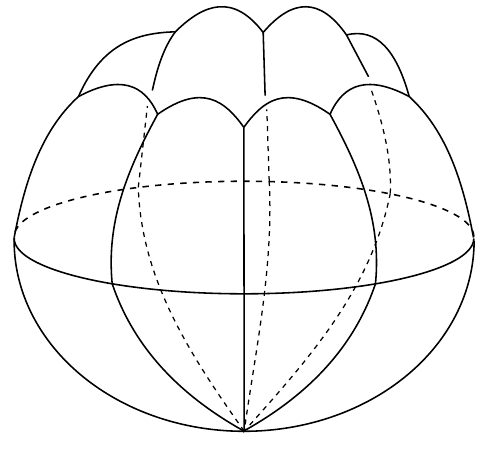}
    \end{overpic}
	\caption{\small{The spherical complex $\cL_{\theta, 2r}$ built from $2r=8$ spherical isosceles triangles with two sides of length $\theta$ for $\frac{\pi}{2}<\theta<\pi$ and with the angle between these sides equal to $\frac{\pi}{3}$. We show the complex is $\CAT(1)$ provided $\theta$ is sufficiently close to $\frac{\pi}{2}$ so that the boundary loop has length greater than $2\pi$. }}
	\label{figure:flowerCAT1}
    \end{centering}
    \end{figure}

\begin{lemma} \label{lemma_LaCAT1}
        If $2r\sigma>2\pi$, then $\cL_{\theta, \frac{\pi}{3}, 2r}$ is $\CAT(1)$. If $r\sigma>2\pi$, then $\cL_{\theta, \frac{2\pi}{3}, r}$ is $\CAT(1)$.
    \end{lemma}
    \begin{proof} Let $\cL_{\theta, \frac{\pi}{3}, 2r}$ be constructed from $2r$ copies of $\Delta_{\theta, \frac{\pi}{3}}$ denoted $\Delta_1, \ldots, \Delta_{2r}$.  Let $s_i^1$ and $s_i^2$ denote the two sides of the triangle $\Delta_i$ with length $\theta$ and let $\Delta_i$ be glued to $\Delta_{i+1}$ by identifying $s_i^2$ with $s_{i+1}^1$ for all $i$, taken modulo $2r$.

        We will use the convex gluing theorem \cite[Theorem II.11.1]{bridsonhaefliger} which in our setting states that if two $\CAT(1)$ spaces are glued by isometry along $\pi$-convex subspaces then the resulting space is also $\CAT(1)$. Here a subspace $A$ of a metric space $(X,d)$ is $\pi$-convex if any two points $x, y$ in $A$ with $d(x,y) < \pi$, can be joined by a geodesic and any such geodesic is contained in $A$.

        We will realize $\cL_{\theta, \frac{\pi}{3}, 2r}$ as $(\cN \sqcup \cS)/\sim$, where $\cN$ and $\cS$ (labeled to indicate ``northern hemisphere and southern hemisphere''), are as follows. Viewing the vertex of $\Delta_i$ with angle $\alpha$ as the south pole of $S^2$, let $\Delta_i^S \subset \Delta_i$ and $\Delta_i^N \subset \Delta_i$ be the intersection of $\Delta_i$ with the southern hemisphere and northern hemisphere of $S^2$, respectively. Let $\underline{s}_i^j = s_i^j \cap \Delta_i^S$, and let $\overline{s}_i^j = s_i^j \cap \Delta_i^N$ for $j =1,2$.
        Let $\cS = \left(\bigsqcup \Delta_i^S\right)/\sim$ and $\cN = \left(\bigsqcup \Delta_i^N \right)/\sim$ be the union of these spaces after the gluing that defines $\cL_{\theta, \frac{\pi}{3}, 2r}$. The {\it equator} of $\Delta_i^S$ and $\Delta_i^N$ is the intersection of the equator in $S^2$ with $\Delta_i^S$ and $\Delta_i^N$, respectively. The {\it equator} in $\cS$ and $\cN$ is the union of the equators in the cells after gluing.
        Then, $\cL_{\theta, \frac{\pi}{3}, 2r}$ is the union of $\cS$ and $\cN$ glued along their equators by an isometry. The equators in $\cS$ and $\cN$ are $\pi$-convex, since if points in $\cS$ or $\cN$ have distance less than $\pi$, then there is an isometric embedding of the subsegment of the equator containing these points to $S^2$. It remains to prove $\cS$ and $\cN$ are $\CAT(1)$.

        We first show that $\cS$ is $\CAT(1)$. Since each side of $\Delta_i^S$ is convex, the space $(\Delta_i^S \cup \Delta_{i+1}^S)/\sim$ obtained by gluing these triangles is $\CAT(1)$ by the convex gluing theorem. Inductively, the spaces $\cS_1:=(\Delta_1^S \cup \Delta_2^S \cup \ldots \cup \Delta_{r}^S)/\sim$ and $\cS_2:=(\Delta_{r+1}^S \cup \Delta_{r+2}^S \cup \ldots \cup \Delta_{2r}^S)/\sim$ obtained by gluing on one triangle at a time along its boundary are each $\CAT(1)$. The space $\cS$ is isometric to $(\cS_1 \cup \cS_2)/\sim$, where the edges $\underline{s}_1^1 \cup \underline{s}_r^2  \subset \cS_1$ are glued to the edges $\underline{s}_{r+1}^1 \cup \underline{s}_{n}^2 \subset \cS_2$. It remains to prove each of these subsets is $\pi$-convex in $\cS_1$ and $\cS_2$, respectively. Let $\rho = \underline{s}_1^1\cup \underline{s}_r^2$. Let $P,Q \in \rho$ with $d_{\cS_1}(P,Q)<\pi$, and let $\gamma$ be the geodesic in $\cS_1$ from $P$ to $Q$. Let $Q' = \gamma \cap \underline{s}_3^2 = \underline{s}_4^1$. Since $\Delta_i$ is a copy of $\Delta_{\theta, \pi/3}$ , $\Delta_1^S \cup \Delta_2^S \cup \Delta_3^S$ isometrically embeds in $S^2$ and $\underline{s}_1^1 \cup \underline{s}_3^2$ is contained in a great circle of $S^2$. Since $d_{\cS_1}(P, Q')<\pi$, the geodesic $[P,Q']$ is the concatenation of the geodesics $[P,s]$ and $[s,Q']$, where $s$ denotes the image of the south pole of each triangle in $\cS$. By the triangle inequality, $d_{\cS_1}(Q,s) \leq d_{\cS_1}(Q,Q') + d_{\cS_1}(Q',s)$, so $\gamma$ is contained in $\rho$  as desired. The argument for $\underline{s}_{r+1}^1 \cup \underline{s}_{2r}^2$ is analogous, so $\cS$ is $\CAT(1)$ by the convex gluing theorem.

        We now show that $\cN$ is $\CAT(1)$. The space $\cN$ is a $M^2_1$-polyhedral complex with finite shapes, so it suffices to prove $\cN$ satisfies the link condition and contains no isometrically embedded circles of length less than $2\pi$ by \cite[Theorem II.5.4]{bridsonhaefliger}. Vertex links of $\cN$ are paths, so they satisfy the link condition. Each edge $\overline{s}_i^j$ is convex in $\cN$. Thus, an isometrically embedded circle crosses each such edge at most once, and hence crosses each such edge exactly once. The distance between $\overline{s}_i^1$ and $\overline{s}_i^2 \geq \sigma$, so the length of an isometrically embedded circle is greater than $2\pi$ by assumption. Thus, $\cN$ is $\CAT(1)$. Therefore, $\cL_{\theta, \frac{\pi}{3}, 2r}$ is $\CAT(1)$ by the convex gluing theorem.

        To prove that $\cL_{\theta, \frac{2\pi}{3}, r}$ is also $\CAT(1)$ we need a very slight variation in the above proof. Note that each $\Delta_i$ is a copy of $\Delta_{\theta, 2\pi/3}$, therefore, $\underline{s}_1^1 \cup \underline{s}_3^2$ is no longer isometric to a subsegment of a great circle of $S^2$. In this case, we look at $\Delta_1^S \cup \Delta_2^S$ which is isometrically embedded in $S^2$ and contains a segment $\gamma'$ containing $\underline{s}_1^1$ such that $\gamma'$ is isometric to half a great circle of $S^2$. We now chose $Q'$ to be the intersection of $[P,Q]$ with $\gamma'$. Rest of the proof follows as above.
    \end{proof}

In addition, we have the following lemma that relates $\theta_1$ and $y_0$, the height of the horosphere $\partial \cH_0$ on which the caps are glued.

\begin{figure}
    \begin{centering}
	\begin{overpic}[width=.4\textwidth, tics=5]{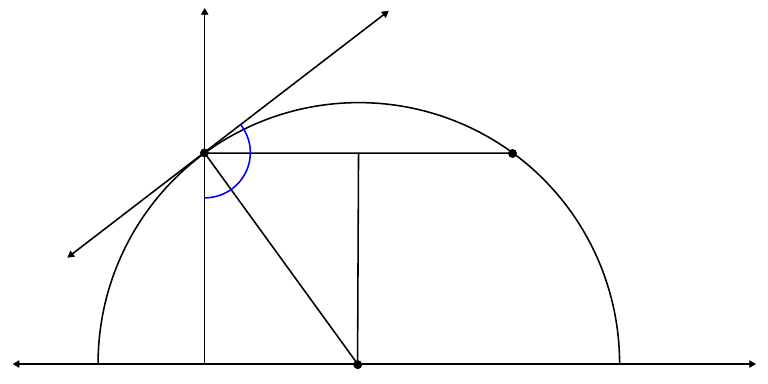}
        \put(20.5,30){\small{$ih$}}
        \put(68,30){\small{$D+ih$}}
        \put(33,25.5){\small{\textcolor{blue}{$\theta$}}}
    \end{overpic}
	\caption{\small{The blue angle between hyperbolic geodesics indicated in the figure is relevant to building a $\CAT(-1)$ model geometry. Its relationship between the values $h$ and $D$ is computed in \Cref{lemma:angles}.  }}
	\label{figure:angles}
    \end{centering}
    \end{figure}

\begin{lemma} \label{lemma:angles}
        View $\Hy^2$ in the upper halfspace model. Let $\gamma$ denote the vertical geodesic ray $(0,ih]$ for some $h >0$. Let $\gamma'$ be a geodesic segment from $ih$ to $D+ih$ for $D>0$.  Let $\theta = \angle_{ih}(\gamma, \gamma')$. See \Cref{figure:angles}. Then $\theta = \pi - \arctan\bigl(\frac{2h}{D}\bigr)$.
    \end{lemma}
\begin{proof}
The conclusion follows from elementary Euclidean geometry as follows; see
\Cref{figure:angles}. The segment $\gamma'$ is contained in a semicircle $C$ that meets the real axis at right angles. Let $\ell$ be the tangent line to $C$ based at $ih$, and let $\ell'$ be the subray of $\ell$ with nonnegative real coordinates. Then $\theta = \angle_{ih}(\gamma, \ell')$. Let $\sigma$ be the Euclidean geodesic between $ih$ and $D+ih$, and let $\sigma'$ be the Euclidean geodesic between $ih$ and $\frac{D}{2}$. Let $\theta_1 = \angle_{ih}(\ell',\sigma)$, and let $\theta_2 = \angle_{ih}(\sigma,\sigma')$. Then, $\theta_1+\theta_2 = \frac{\pi}{2}$, $\tan \theta_2 = \frac{2h}{D}$, and $\theta = \frac{\pi}{2} +\theta_1 = \pi - \theta_2$, as desired.
\end{proof}

 \begin{prop} \label{prop:y_0}
    There exists $y_0>0$ so that if $\cY_\Gamma$ is the complex built using truncated blocks and gluing caps relative to a horoball $\cH_0$ given by $w_3 \geq y_0$ as in \Cref{subsection:thespaceYm}, then the link in $\cY_\Gamma$ of each vertex of a cap is $\CAT(1)$.
 \end{prop}

\begin{proof}
    Let $v$ be a vertex of a cap. As described in \Cref{nota:cap_links}, there are three cases up to isometry, corresponding to caps $C_1$, $C_2$ and $C_3$.
Consider the case $i=1$. Then $\Lk(v_1, \cY_\Gamma)$ is isometric to $\cL_{\theta_1, \frac{\pi}{3}, 2r}$, where the spherical triangle $\Delta_{\theta_1}$ used to build $\cL_{\theta_1, \frac{\pi}{3}, 2r}$ is equal to $\Sigma_1$. As shown in \Cref{lemma_LaCAT1}, this complex is $\CAT(1)$ if $2 r \sigma_1 > 2 \pi$. Thus, we want $\sigma_1 > \frac{\pi}{r}$. Since $r \geq 4$, it suffices to ensure
\begin{equation*}\label{eqn:sigma_1}\sigma_1 > \frac{\pi}{4}. \end{equation*}

The values $\sigma_1$ and $\theta_1$ are related by the Spherical Law of Cosines:
\begin{equation*}\label{eqn:sigma_1 cosine}
\sigma_1 = \arccos\left(\cos^2 \theta_1+ \frac{1}{2}\sin^2 \theta_1\right).
\end{equation*}
Finally, by \Cref{lemma:angles},
\begin{equation*}\label{eqn:theta y_0}
\theta_1 = \pi - \arctan\left(\frac{2y_0}{D}\right)
\end{equation*}
for $D=\sqrt{3}$.
We get a similar set of equations for $i=2$ and $i=3$ and appropriate value of $D$. Any $y_0$ that satisfies all three sets of equations thus satisfies the conclusion of the proposition. For instance, one may take $y_0=1.5$. An explicit calculation to solve these system of equations is provided in the Appendix (\Cref{sec:appendix}).
\end{proof}

We end by recording one other feature of the geometry of the caps required later in the proof of \Cref{lemma:disjoint_prisms}.

    \begin{lemma}\label{lem:height_of_cap}
        The height of a cap in $\cY_\Gamma$ is at most $\sqrt{y_0^2+1}$, where $y_0$ is the height of the horosphere $\partial \cH_0$ in $\mathbb{U}^3$ used to define a truncated block.
    \end{lemma}
    \begin{proof} Let $C$ be a cap built over a Euclidean polygon $C'$ contained in the horosphere $\partial \cH_0$. Let $\gamma$ be the Euclidean segment joining the center of $C'$ to a vertex of $C'$.  Let $S$ be the circle circumscribing $C'$ in the horosphere. Then the hyperbolic polygon bounding the cap $C$ is contained in a Euclidean hemisphere that intersects the horosphere in the circle $S$ and is orthogonal to the boundary of $\mathbb{U}^3$. The height of the cap $C$ is the radius of this hemisphere, which equals the hypotenuse of the right-angled triangle with side lengths $y_0$ and the length of $\gamma$.

        Let the height of the cap $C$ be $h$. There are three cases to check, with the geometry described in \Cref{nota:cap_links}. If $C$ is constructed from a $\Delta(3,3,r)$ plane, the side length of $C'$ is $\sqrt{3}$ and the length of $\gamma$ equals $1$, so $h = \sqrt{y_0^2+1}$. If $C$ is constructed from a $\Delta(3,q,r)$ plane, the side length of $C'$ is $1$ and the length of $\gamma$ equals $\frac{1}{\sqrt{3}}$, so $h = \sqrt{y_0^2+\frac{1}{3}}$. Lastly, if $C$ is constructed from a $\Delta(p,q,r)$ plane, the side length of $C'$ is $1$ and the length of $\gamma$ equals $1$, so $h = \sqrt{y_0^2+1}$. The lemma now follows.
    \end{proof}

 \subsubsection{\texorpdfstring{A $\CAT(1)$ metric on the $2$-skeleton of the $n$-simplex.}{A CAT(-1) metric on the 2-skeleton of the m-simplex.}}\label{subsection:2skeletonnsimplex}

    We show in this subsection that a certain ``maximally symmetric'' piecewise-spherical metric on the $2$-skeleton of the $m$-simplex is a $\CAT(1)$. This space appears as a link in the complex $\cY_{\Gamma}$ at the vertex $x_0$ as in Figure~\ref{figure:blockAndDavisHex}.

    Let $\cL_m$ denote the $2$-skeleton of the $(m-1)$-simplex equipped with the piecewise-spherical metric in which each $2$-cell is isometric to the equilateral triangle in the metric unit $2$-sphere $S^2$ with angles $\frac{2\pi}{3}$. For example, the space $\cL_4$ is isometric to the unit Euclidean sphere.

    Recall relevant definitions from \cite[Chapter 7]{bridsonhaefliger}. A \emph{geodesic $k$-simplex} in $M_{\kappa}^k$, is the convex hull of $(k+1)$ points in general position.  A spherical simplicial complex $L$ is a \emph{metric flag complex} if it satisfies the following condition: if a set of vertices $\{v_0, \ldots, v_k\}$ are pairwise joined by edges $e_{ij}$ in $L$ and there exists a spherical $k$-simplex in $\mathbb{S}^k$ whose edge lengths are $\ell(e_{ij})$, then $\{v_0, \ldots, v_k\}$ span a $k$-simplex in~$L$.

    \begin{lemma}[Moussong's lemma \cite{moussong}, see \cite{bridsonhaefliger}]
        If all the edges of a spherical simplicial complex $L$ have lengths at least $\frac{\pi}{2}$, then $L$ is CAT(1) if and only if it is a metric flag complex.
    \end{lemma}

We now use Moussong's Lemma to prove the following theorem.

\begin{thm} \label{thm:cat1}
    The metric space $\cL_m$ is $\CAT(1)$ for all $m \geq 3$.
\end{thm}
\begin{proof}
The length of each edge in $\mathcal{L}_m$ is equal to $\ell_0 =\arccos(\frac{-1}{3})$, which is greater than $\frac{\pi}{2}$. Therefore, by Moussong's Lemma, it suffices to show that $\mathcal{L}_m$ is a metric flag complex. For a collection of three vertices $\{v_0, v_1, v_2\}$ in $\mathcal{L}_m$, there exists a unique (up to isometry) spherical 2-complex in $\mathbb{S}^2$ whose edge lengths are $\ell_0$, and any three vertices in $\mathcal{L}_m$ span a 2-simplex. Note that the angle at each vertex in such a 2-simplex is $\frac{2\pi}{3}$.

Next consider a collection of four vertices in $\mathbb{S}^3$, each pair joined by a geodesic of length $\ell_0$. Each subcollection of three vertices spans a unique spherical 2-simplex, and the sum of angles of the three 2-simplices incident at a vertex is $2\pi$. Therefore, the four vertices lie on a copy of $\mathbb{S}^2$ and are not in general position. Thus, there is no spherical 3-simplex in $\mathbb{S}^3$ with edge lengths $\ell_0$. This implies there is no spherical $k$-simplex in $\mathbb{S}^k$ with edge lengths $\ell_0$ for $k \geq 3$, completing the proof.
\end{proof}

\subsubsection{\texorpdfstring{Proof of $\CAT(-1)$ model geometry.}{Proof of CAT(-1) model geometry.}} \label{subsec:remaining_links}

    \begin{thm} \label{thm:YGammaCAT-1}
        The space $\cY_{\Gamma}$ is $\CAT(-1)$.
    \end{thm}
    \begin{proof}
        The space $\cY_{\Gamma}$ is a geodesic metric space by Proposition~\ref{prop:comlenspace} and is simply connected by Lemma~\ref{lemma:cYsimplyconn}. Thus, by the Cartan--Hadamard Theorem (see \cite[Theorem II.4.1]{bridsonhaefliger}), it remains to prove that $\cY_{\Gamma}$ has curvature $\leq -1$.
        We will show that the space $\cY_{\Gamma}$ can be realized as an $M^3_{-1}$-ideal polyhedral complex with $\Shapes(\cY_{\Gamma})$ finite. So, by Proposition~\ref{prop:linkcond} it will suffice to prove that for every vertex $v \in \cY_{\Gamma}$ the link complex $\Lk(v,\cY_{\Gamma})$ is $\CAT(1)$.

        We begin by noting that the space $\cY_{\Gamma}$ given in Construction~\ref{const:Ym} is not described as a $M^3_{-1}$-ideal polyhedral complex. Indeed, $\cY_{\Gamma}$ is built by first assembling truncated blocks (which are not convex in $\Hy^3$) and then attaching either caps or cusps to the boundary components. Before describing the alternative cellular structure, we will verify that links of vertices contained in kites in $\cY_{\Gamma}$ are $\CAT(1)$. We will then subdivide cells and glue on subdivided caps to get a $M^3_{-1}$-ideal polyhedral complex. In the new cellulation, the links of vertices contained in kites are simply subdivided, and will remain $\CAT(1)$. We chose this order for exposition because the links before the subdivision are more symmetric and are easier to work with.

        \noindent {\sc (I.) Vertices in kites.}

        Let $\hat{B}(ijk)$ be a truncated block in the space $\cY_{\Gamma}$.
        We will refer to the labeling shown in Figure~\ref{figure:blockAndDavisHex}.  Up to isometry, it suffices to check the link complexes $\Lk(x_0, \cY_{\Gamma})$, $\Lk(x_i, \cY_{\Gamma})$, and $\Lk(x_{ij}, \cY_{\Gamma})$.

        The point $x_0$ corresponds to a vertex in the Davis complex, as shown in red in Figure~\ref{figure:blockAndDavisHex}. The link of this vertex in the Davis complex is a complete graph on $m$ vertices. There is a vertex in the link $\Lk(x_0, \cY_{\Gamma})$ corresponding to each edge $\{x_0, x_i\}$ for $1 \leq i \leq m = |V\Gamma|$. Each disjoint triple $x_i, x_j, x_k$ yields a truncated block $\hat{B}(ijk)$ in the space $\cY_{\Gamma}$. Thus, the link of $x_0$ restricted to $\hat{B}(ijk)$ is a spherical triangle with angles $\frac{2\pi}{3}$. Therefore, $\Lk(x_0, \cY_{\Gamma})$ is the metric space $\cL_m$, which is $\CAT(1)$ by \Cref{thm:cat1}.

        We next show $\Lk(x_i, \cY_{\Gamma})$ is isometric to the spherical join of $S^0$ with the complete graph on $m-1$ vertices, where each edge of the graph has length $\frac{2\pi}{3}$.
        The vertex $x_i$ is incident to $(m-1) + 2$ edges in the polyhedral complex $\cY_\Gamma$. Indeed, in the subdivided Davis complex, the vertex $x_i$ corresponds to the midpoint of an edge; see \Cref{figure:blockAndDavisHex}. Suppose this edge is $\{x_0, x_0'\}$. In addition, for each $j \in \{1, \ldots, m\} \setminus \{i\}$ there is an edge in $\cY_\Gamma$  from $x_i$ to $x_{ij}$. We may thus label the vertices in $\Lk(x_i, \cY_\Gamma)$ by $x_0, x_0'$ and $\{j \, |\, j \in \{1, \ldots, m\} \setminus \{i\}\}$. For each pair of distinct labels $j,k \in \{1, \ldots, m\} \setminus \{i\}$ there is a truncated block $\hat{B}$ (respectively $\hat{B}'$) containing $x_i, x_{ij}, x_{ik}$ and $x_0$ (respectively $x_0'$). The link of $x_i$ in the block $\hat{B}$ is a spherical triangle with vertices $j, k, x_0$ given by the spherical join of a point (corresponding to $x_0$) with an edge (corresponding to $x_{ij}, x_{ik}$) of length $\frac{2\pi}{3}$. The link of $x_i$ in the block $\hat{B}'$ yields an isometric spherical triangle labeled by $j,k$ and $x_0'$. The collection of such labeled spherical triangles are glued along edges with the same vertex labels, yielding the join as desired.
        Since both $S^0$ and the metric graph are $\CAT(1)$, so is the spherical join \cite[II.3.15]{bridsonhaefliger}.

        Finally, we show that $\Lk(x_{ij}, \cY_\Gamma)$ is isometric to the spherical join of $m-2$ singletons with a metric graph isomorphic to a cycle with $2m_{ij}$ edges, each of length $\frac{\pi}{3}$. We first label the vertices of the link. The point $x_{ij}$ corresponds to a barycenter of a polygon with $2m_{ij}$ sides in the Davis complex as in Figure~\ref{figure:blockAndDavisHex}. This polygon is divided into $2m_{ij}$ kites incident to $x_{ij}$. Label these in cyclic order by $K_1, \ldots, K_{2m_{ij}}$ so that $K_n$ shares an edge with $K_{n+1}$. Hence, $x_{ij}$ is incident to $2m_{ij}$ edges in $\cY_{\Gamma}$ coming from these kites, and we label the corresponding vertices in the link of $x_{ij}$ by $v_1, \ldots, v_{2m_{ij}}$ so that $v_n$ corresponds to the edge in the kites $K_{n-1}$ and $K_n$. Let $I:=\{1, \ldots, m\} \setminus \{i,j\}$. For each $k \in I$, each kite $K_n$ is contained in a block labeled by $ijk$, and these blocks are glued along their edges labeled $E(ij)$ that contain $x_{ij}$. Hence, there are $m-2$ additional vertices in the link of $x_{ij}$ that we will label by $\{w_k \,|\, k \in I\}$. For each $n \in \{1, \ldots, 2m_{ij}\}$ and each $k \in I$, the link of $x_{ij}$ in the block containing $K_n$ and labeled by $ijk$ is a spherical join of a point (corresponding to $w_k$) with an edge (corresponding to $v_n$ and $v_{n+1}$) of length $\frac{\pi}{3}$. The collection of such labeled spherical triangles are glued along edges with the same vertex labels, yielding the join as desired. Thus, $\Lk(x_{ij}, \cY_{\Gamma})$ is $\CAT(1)$ by \cite[II.3.15]{bridsonhaefliger}.

        \vskip.2in

    \begin{figure}
    \begin{centering}
	\begin{overpic}[width=.6\textwidth,tics=5]{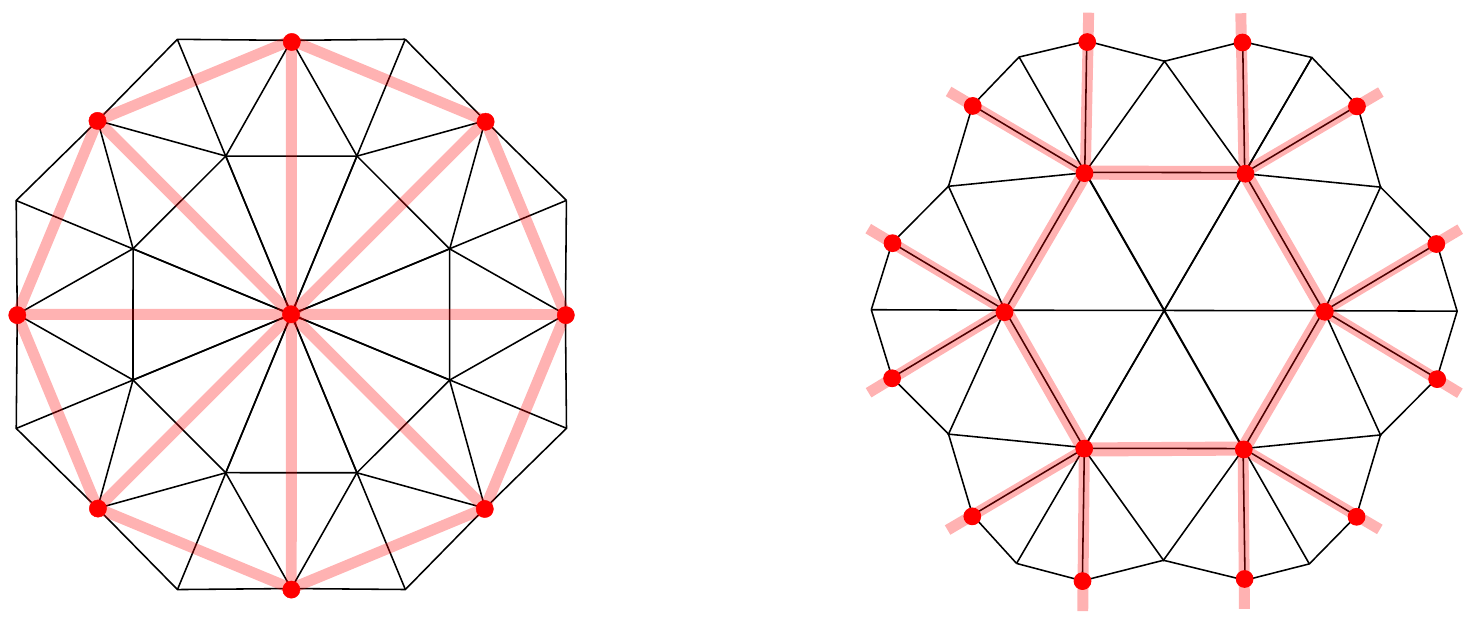}
    \end{overpic}
	\caption{\small{Triangle tilings $\cT_1(\cF)$ of the boundary components of $\hat{\cY}_{\Gamma}$ are shown in black. In red is the alternative tiling $\cT_2(\cF)$, by triangles on the left and hexagons on the right, so that each vertex in this tiling has link of cone angle strictly greater than $2\pi$. The left picture corresponds to a boundary component $\cF$ determined by $\Delta(3,3,4)$ and the right picture to $\Delta(3,4,4)$.}}
	\label{figure:tri_tilings}
    \end{centering}
    \end{figure}

        \noindent {\sc (II.) An $M^3_{-1}$-ideal polyhedral structure on $\cY_{\Gamma}$ and the remaining links.}

            The space $\cY_{\Gamma}$ is constructed by first assembling truncated blocks via the Davis complex and then gluing either caps or cusps to the boundary components of the resultant space. We will subdivide these objects and glue them together to realize $\cY_{\Gamma}$ as a space built from ideal polyhedra. Along the way, we will verify the links of the remaining vertices are $\CAT(1)$.

            Let $\hat{\cB}$ be a truncated block in $\cY_{\Gamma}$. Suppose $\hat{\cB}$ is labeled by three letters $i,j,k \in \{1, \ldots, |V\Gamma|\}$ corresponding to three generators of the Coxeter group that span a triangle in the defining graph with edge labels $p,q,r \geq 3$. First, subdivide the kite faces of $\hat{\cB}$ by adding an edge from the vertex $x_0$ to each of the vertices $x_{ij}$, $x_{ik}$, and $x_{jk}$. Suppose the boundary face of $\hat{\cB}$ is contained in a boundary component $\cF$ of $\hat{\cY}_{\Gamma}$.  Recall, there are tilings $\cT_1(\cF)$ and $\cT_2(\cF)$ of the boundary component $\cF$ that were defined in \Cref{const:Ym}.

            Suppose $p=q=r=3$. Then, there is a horoball from $\Hy^3$ glued to $\cF$ along its boundary horosphere, which inherits a tiling by equilateral Euclidean triangles from the tiling $\cT_1(\cF) = \cT_2(\cF)$ of $\cF$. One such triangle corresponds to the boundary of $\hat{\cB}$. Subdivide the horoball into cusps over these triangles. Then the union of a cusp and the corresponding truncated block is an ideal hyperbolic polyhedron (isometric to the block $\cB$). There are no new vertices constructed from this gluing.

            Suppose $p,q,r>3$. Then $\cT_1(\cF) = \cT_2(\cF)$, and hence there is a cap glued directly to the boundary component of $\hat{\cB}$. Take the union of $\hat{\cB}$ and this cap to obtain an $M^3_{-1}$-polyhedron. There are three vertex links to verify: the vertices that are vertices of the cap. These links are $\CAT(1)$ by \Cref{prop:y_0}.

            Now suppose $p=q=3$ and $r >3$. Each triangle in $\cT_1(\cF)$ has two vertices with cone angle $2\pi$ and one vertex with cone angle $\frac{2\pi r}{3}$. Each vertex in $\cT_2(\cF)$ has cone angle $> 2\pi$.
            As in the left of Figure~\ref{figure:tri_tilings}, locally one may view the space as having a truncated block ``below'' each black equilateral triangle and a hyperbolic cap ``above'' each red equilateral triangle. Since these do not agree along their boundaries, we perform the natural subdivision and gluing. More specifically, suppose that $(x_i,x_j)^r=1$ in $W_{\Gamma}$. First subdivide $\hat{\cB}$ into two $3$-cells by adding a $2$-cell as follows. Let $v$ denote the vertex of the boundary face of $\hat{\cB}$ incident to the edge $E(ij)$,
            and let $v'$ denote the midpoint of the edge of the boundary face of $\hat{\cB}$ incident to the face labeled by $k$. Add a geodesic $2$-cell to $\hat{\cB}$ with vertex set $v$, $v'$, $x_k$, $x_0$, and $x_{ij}$. Then, subdivide the cap (partially) glued to $\hat{\cB}$ into six isometric pieces by adding geodesic $2$-cells that, in the upper half-space model of $\Hy^3$, lie above the black edges in the barycentric subdivision of the red equilateral triangle. The boundary face of each half of the truncated block can then be glued to the boundary of one sixth of the subdivided cap to form a hyperbolic polyhedron.

            This subdivision and gluing results in, up to isometry, three vertex links to verify; see Figure~\ref{figure:link334}. The vertex added to the center of the cap glued to $\hat{\cB}$ lies on the boundary of a halfspace in $\Hy^3$, so its link is a hemisphere in $S^2$, which is $\CAT(1)$. Next consider the vertex $w$, the midpoint of an edge of the cap. The vertex $w$ lies in two caps and has a neighborhood disjoint from the truncated block. In each cap, $w$ lies on the intersection of two geodesic planes, so its link in each cap is a lune in $S^2$.
            Thus, the link of $w$ in $\cY_{\Gamma}$ is the union of two lunes along one boundary arc of each. Each lune is a $\CAT(1)$ space and these spaces are glued along a convex subset, so $\Lk(w, \cY_{\Gamma})$ is $\CAT(1)$ by \cite[Theorem II.11.1]{bridsonhaefliger}. If $v$ is a vertex of a cap, then $\Lk(v,\cY_\Gamma)$ is $\CAT(1)$ by \Cref{prop:y_0}.

    \begin{figure}
    \begin{centering}
	\begin{overpic}[width=.5\textwidth, tics=5]{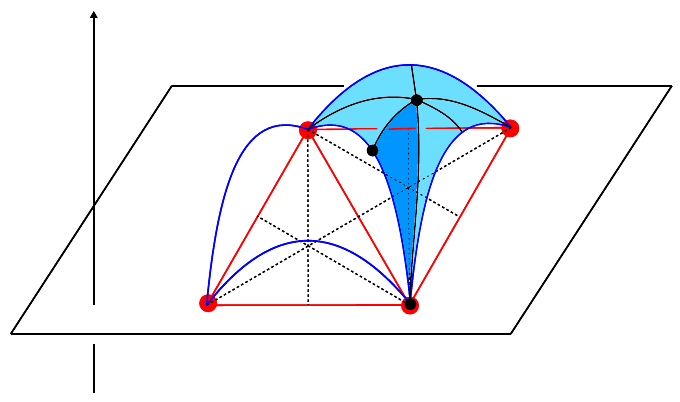}
        \put(0,45){$\Hy^3$}
        \put(56.3,36){\small{$w$}}
        \put(62.5,13.5){\small{$v$}}
    \end{overpic}
	\caption{\small{The red triangles are equilateral triangles on a horosphere in $\Hy^3$. The blue lines depict hyperbolic geodesics between their endpoints, and the blue shaded triangles lie on a halfspace containing the three red points.  }}
	\label{figure:link334}
    \end{centering}
    \end{figure}

            Suppose $p=3$ and $q,r>3$. Each triangle in $\cT_1(\cF)$ has one vertex with link of cone angle $2\pi$ and with two vertices with links of cone angle $> 2\pi$. Therefore, each vertex of $\cT_2(\cF)$ has cone angle $\frac{2\pi r}{3} > 2\pi$. As in the right side picture in Figure~\ref{figure:tri_tilings}, one may view the space as having a truncated block ``below'' each black equilateral triangle and with a hyperbolic cap ``above'' each red hexagon with Euclidean side lengths $1$. As before, these do not agree along their boundaries, so we subdivide the caps and glue them to the truncated blocks in the natural way.
            Each cap is subdivided into six isometric $3$-cells by adding geodesic $2$-cells that, in the upper half-space model of $\Hy^3$ lie above the black edges in the subdivision of the hexagon. The boundary face of the truncated block can be glued to one sixth of the subdivided cap to form a hyperbolic polyhedron. This subdivision and gluing results in, up to isometry, two vertex links to verify. The vertex added to the center of the cap glued to $\hat{\cB}$ lies on the boundary of a halfspace in $\Hy^3$, so its link is a hemisphere in $S^2$, which is $\CAT(1)$. If $v$ is the vertex of the cap then $\Lk(v,\cY_\Gamma)$ is $\CAT(1)$ by \Cref{prop:y_0}.

            Since the subdivisions and gluings are compatible throughout the space $\cY_\Gamma$, the resultant cell complex is a $\CAT(-1)$ $M^3_{-1}$-ideal polyhedral complex. \end{proof}

    \section{Upper bounds on the conformal dimension}\label{sec:upperbound}

    In this section we give an upper bound for the conformal dimension of the Bowditch boundary $\partial(W_\Gamma, \cP)$ equipped with a visual metric. The cusped Cayley graph $X(W_{\Gamma}, \cP)$ is quasi-isometric to the $\CAT(-1)$ space $\cY_{\Gamma}$ by Proposition~\ref{prop:QIYAndCuspX}. The boundary $\p \cY_\Gamma$ admits a visual metric $d_e$ with parameter~$e$ by Theorem~\ref{thm:vis_parameter}. Thus, an upper bound on the conformal dimension can be given by providing an upper bound on the Hausdorff dimension of the metric space $(\p \cY_\Gamma, d_e)$. This Hausdorff dimension is equal to the critical exponent of the Poincar\'{e} series for the action of $W_\Gamma$ on $\cY_\Gamma$ by Theorem~\ref{thm:Paulin}. Therefore, it suffices to count the number of points of the orbit $W_\Gamma \cdot x_0$ of a fixed basepoint $x_0 \in \cY_\Gamma$ that lie in the ball of radius $r$ around $x_0$.

    We will use the following notation throughout this section.

    \begin{notation} \label{notation:curlyXm}
        The subdivided Davis complex for $W_\Gamma$ is denoted $SX_\Gamma$ and is given in Construction~\ref{const:Ymhat}. The union of kites in the $\CAT(-1)$ complex $\cY_\Gamma$ is a connected subcomplex denoted~$\cX_\Gamma$ also described in Construction~\ref{const:Ymhat} and is combinatorially isomorphic to $S X_\Gamma$. We note that by construction, each component of $\cY_\Gamma\setminus \cX_\Gamma$ corresponds to a conjugate of a triangle subgroup of $W_\Gamma$.

        An {\bf $\cX_\Gamma$-polygon} is a union of kites in $\cX_\Gamma$ that is combinatorially isomorphic to a subdivided polygon in $S X_\Gamma$. An {\bf $\cX_\Gamma$-flat} is a union of $\cX_\Gamma$-polygons that is combinatorially isomorphic to a flat in $S X_\Gamma$, which arises from a triangle subgroup of $\Gamma$ with edges labeled $3$. Note that the induced metric on an  $\cX_\Gamma$-flat is not the Euclidean metric. An {\bf $\cX_{\Gamma}$-hyp plane} is the union of $\cX_{\Gamma}$-polygons that is combinatorially isomorphic to a subspace of $S X_\Gamma$ that is invariant under a hyperbolic triangle subgroup of $W_\Gamma$.

        Let $x_0$ be a vertex of a fixed $\cX_\Gamma$-polygon $P_0$.
    \end{notation}

    The strategy for counting orbit points in the ball $B_{\cY_\Gamma}(r,x_0) \subset \cY_\Gamma$ has three steps:

    \begin{enumerate}
        \item Define the {\it itinerary type} of an orbit point $gx_0$ in terms of the distance between disjoint $\cX_\Gamma$-polygons that the geodesic from $x_0$ to $gx_0$ intersects. The itinerary type will be a vector with entries in $\N \cup \{\frac{1}{2}\}$. (This is done in \Cref{subsec:itinerary_type}.)
        \item Count the number of orbit points with a given itinerary type. (\Cref{subsec:fixed_itinerary})
        \item Sum the counts over all the itinerary types of orbit points in the $r$-ball. (\Cref{subsec:summing_over_types})
    \end{enumerate}

    Finally, we put these steps together to obtain the bound on the Hausdorff dimension of $(\p \cY_\Gamma, d_e)$, completed in Subsection~\ref{subsec:Hausdim_bound}. We postpone to \Cref{subsec:GeoArg} the geometric arguments used in the calculation of the constants involved in the final theorem.

\subsection{Itinerary type} \label{subsec:itinerary_type}

   We use disjoint $\cX_{\Gamma}$-polygons to define the itinerary type of an orbit point. This point of view is useful for two reasons: there is a uniform lower bound on the distance between disjoint $\cX_\Gamma$-polygons (Lemma~\ref{lemma:disjoint_polygons}), and we will be able to transfer some counting to the combinatorially well-understood Davis complex.

    \begin{defn}[Polygon sequence]
        Let $g \in W_\Gamma$ and let $\gamma$ be the unique geodesic from $x_0$ to $gx_0$ in~$\cY_\Gamma$. A {\it polygon sequence} for $gx_0$ is a list of $\cX_\Gamma$-polygons $P_0, P_1, \ldots, P_j$, so that $P_0$ is the $\cX_\Gamma$-polygon containing $x_0$ specified above, and for $i\geq 1$, $P_i$ is a first $\cX_\Gamma$-polygon disjoint from $P_{i-1}$ that $\gamma$ intersects after leaving $P_{i-1}$, and where $gx_0 \in P_j$.
        For consecutive $\cX_\Gamma$-polygons $P_i$, $P_{i+1}$ in a polygon sequence there are three possibilities. We say the pair has
        \begin{itemize}
        \item {\it Type~A}, if there exists an $\cX_\Gamma$-polygon $P$ so that $P \cap P_i \neq \emptyset$ and $P \cap P_{i+1} \neq \emptyset$.
        \end{itemize}
        Recall that the subspace $\cX_\G$ of $\cY_\G$ is connected and removing it separates $\cY_\G$ into components,
        each of which corresponds to (and is preserved by) a conjugate of a $3$-generator subgroup of $W_\G$.
        Thus, if $P_i, P_{i+1}$ is not a Type A pair, then the subpath $\gamma_i$ of $\gamma$ connecting $\gamma \cap P_i$ to $\gamma \cap P_{i+1}$ has the following form. Either the interior of $\gamma_i$ is contained in one component of $\cY_\Gamma \setminus \cX_\Gamma$, or there exists an $\cX_\Gamma$-polygon $\bar{P}$ that intersects $P_i$, and $\gamma_i$ is the concatenation of a segment from $P_i$ to $\bar{P}$ and a segment from $\bar{P}$ to $P_{i+1}$ whose interior is contained in a component of $\cY_\Gamma \setminus \cX_\Gamma$. We say the pair $P_i, P_{i+1}$ is
        \begin{itemize}
        \item {\it Type~B}, if $P_i$ (or $\bar{P}$) and $P_{i+1}$ are in the same $\cX_{\Gamma}$-flat.
        \item {\it Type~C}, if $P_i$ (or $\bar{P}$) and $P_{i+1}$ are in the same $\cX_{\Gamma}$-hyp plane.
        \end{itemize}
        We define the \emph{length} of a polygon sequence $P_0, \ldots, P_{j}$ to be equal to $j$.
    \end{defn}

    \begin{lemma}
        If $g \in W_\Gamma$, then there exists a polygon sequence for $gx_0$.
    \end{lemma}
    \begin{proof}
        Let $P_0$ be the $\cX_\Gamma$-polygon containing $x_0$ specified above, and let $\gamma$ be the unique geodesic from $x_0$ to $gx_0$ in $\cY_\Gamma$. If $gx_0\in P_0$, then the polygon sequence for $gx_0$ is simply $P_0$ and has length 0. Otherwise, let $x_1$ denote the first point along $\gamma$ for which $x_1$ is contained in some $\cX_\Gamma$-polygon which is disjoint from $P_0$. If $x_1$ is in the interior of some $\cX_\Gamma$-polygon, define $P_1$ to be the (unique)
        $\cX_\Gamma$-polygon containing $x_1$. If the point $x_1$ is on the boundary of some $\cX_\Gamma$-polygon, then choose $P_1$ to be a polygon which contains $x_1$ and is disjoint from $P_0$. If $gx_0\in P_1$ then the polygon sequence for $gx_0$ is complete. Otherwise, we continue defining polygons recursively in this fashion so that for each $i \geq 1$, if the polygon $P_{i}$ contains the point $gx_0$, then the procedure terminates, and $P_0, P_1, \ldots P_i$ is a polygon sequence for $gx_0$. Otherwise, a $\cX_\Gamma$-polygon $P_{i+1}$ is chosen so that $P_{i+1}$ is both disjoint from $P_{i}$ and contains the first point along the segment of $\gamma$ from $P_{i}$ to $gx_0$ which intersects any $\cX_\Gamma$-polygon that is disjoint from $P_{i}$. We note that if $gx_0$ is not contained in $P_i$, then such a choice for the polygon $P_{i+1}$ exists. Indeed, if every $\cX_\Gamma$-polygon that $\gamma$ intersects after $P_i$ contains the point $gx_0$, then the fact that $gx_0$ is the vertex of a polygon means that there is always a choice of $\cX_\Gamma$-polygon $P_{i+1}$ containing $gx_0$ which is disjoint from $P_i$.
    \end{proof}

    \begin{defn}[Itinerary type]
        Let $g \in W_\Gamma$, and let $P_0, P_1, \ldots, P_j$ be a polygon sequence for $gx_0$. We say the point $gx_0$ has {\it itinerary type} $(\ell_1, \ldots, \ell_j) \in (\N \cup \{\frac{1}{2}\})^j$, where $\ell_i = \frac{1}{2}$ if $P_{i-1}$ and $P_i$ are of Type~A and $\ell_i = \lceil d_{\cY_\Gamma}(P_{i-1}, P_i) \rceil \geq 1$ otherwise.
    \end{defn}

    Note that the point $gx_0$ need not have a unique itinerary type, since an orbit point may have more than one polygon sequence.
    We use the constant $\frac{1}{2}$ in the above definition because of \Cref{lemma:disjoint_polygons} which proves the distance between disjoint $\cX_\Gamma$ polygons is at least~$\frac{1}{2}$. As a consequence of this lemma, we also get the following result.

    \begin{corollary} \label{cor:hexLength}
        If $d_{\cY_\Gamma}(x_0, gx_0) \leq r$, then the length of a polygon sequence for $gx_0$ is at most $2r$.
    \end{corollary}

\subsection{Counting orbit points with a fixed itinerary type} \label{subsec:fixed_itinerary}

    Fix an itinerary type $\ell = (\ell_1, \ldots, \ell_j) \in (\N \cup \{\frac{1}{2}\})^j$. If $gx_0$ has itinerary type $\ell$, then $gx_0$ has a polygon sequence $P_0, P_1, \ldots, P_j$ that yields the vector $\ell$. The main result of this subsection is \Cref{thm:fixed_itinerary}, which gives an upper bound on the number of possible ultimate polygons $P_j$ that could occur in such a polygon sequence yielding $\ell$. Since each orbit point with itinerary $\ell$ is a vertex of such a polygon, this yields an upper bound on the number of orbit points with a given itinerary type.

\begin{defn} \label{nota:comb_sphere}
        Let $\cC$ be a collection of $\cX_\Gamma$-polygons. The {\it combinatorial sphere of radius 1} about $\cC$, denoted $\cS_1(\cC)$, is the collection of polygons in $\cX_\Gamma \setminus \cC$ that intersect $\cC$ nontrivially. Inductively, the {\it combinatorial sphere of radius $k$} about $\cC$, denoted $\cS_k(\cC)$, is given by $\cS_k(\cC) = \cS_1\bigl( \bigcup_{i \leq k-1} \cS_i(\cC) \cup \cC \bigr)$. The {\it ball of radius $k$} about $\cC$ is denoted $\cB_k(\cC)$ and is given by $\cB_k(\cC) = \bigcup_{i \leq k} \cS_k(\cC)$.
    \end{defn}

    Given an itinerary type $\ell$, we define collections of $\cX_\Gamma$-polygons $\cP_0, \cP_1, \ldots, \cP_j$ recursively as follows.

    \begin{defn} \label{def:curlyP}
        Let $\ell = (\ell_1, \ldots, \ell_j)$ be an itinerary type. Let $\cP_0 = \{P_0\}$, where $P_0$ is the $\cX_\Gamma$-polygon specified in Notation~\ref{notation:curlyXm}. Suppose $\cP_0, \cP_1, \ldots, \cP_{i-1}$ have been defined, then define $\cP_{i}$ as follows. If $\ell_i = \frac{1}{2}$, then $\cP_{i} = \cS_2(\cP_{i-1})$. If $\ell_i \geq 1$, then $P \in \cP_{i}$ if there exists an $\cX_\Gamma$-polygon $P' \in \cP_{i-1} \cup \cS_1(\cP_{i-1})$ and an $\cX_\Gamma$-flat or an $\cX_{\Gamma}$-hyp plane containing $P$ and $P'$ with $d_{\cY_\Gamma}(P,P') \leq \ell_i$.
        \end{defn}

The proof of \Cref{thm:fixed_itinerary} contains the heart of the geometric and combinatorial arguments in this section, and we first give an example to illustrate the main idea.

 \begin{example}
        Suppose $\ell = (\frac{1}{2}, \frac{1}{2}, 20)$ is an itinerary type.
        Suppose $P_0, P_1, P_2, P_3$ is a polygon sequence yielding $\ell$. Then, $P_0, P_1$ is a Type A pair, $P_1, P_2$ is a Type A pair, and $P_2, P_3$ is either a Type B or a Type C pair and $d_{\cY_\Gamma}(P_2,P_3) \leq 20$.
        Since $P_0$ is a fixed polygon, we let $\cP_0 = \{P_0\}$. The set $\cP_1$ consists of the polygons that form a Type A pair with $P_0$. The set $\cP_2$ consists of the polygons that form a Type A pair with some polygon in $\cP_1$. Lastly, the collection $\cP_3$ consists of the polygons $P$ that form a Type B or Type C pair with some polygon $P'$ in $\cP_2$ and $d_{\cY_\Gamma}(P,P') \leq 20$.
    \end{example}

    As indicated in the above example, the desired bound on the cardinality of $\cP_j$ will come from the following computations carried out in \Cref{subsec:GeoArg}. Given a polygon $P$, we count the number of polygons $P'$ so that $P$ and $P'$ form a Type A pair, a Type B pair with $d_{\cY_\Gamma}(P,P') \leq \ell_i$, and a Type C pair with $d_{\cY_\Gamma}(P,P') \leq \ell_i$. These bounds are collected in the theorem below.

    \begin{thm} \label{thm:fixed_itinerary}
        Let $W_\Gamma \in \cW$ be generated by $m$ elements, and let $M$ be the maximum edge label in $\Gamma$.
        Let $P$ be an $\cX_\Gamma$-polygon. Let $\ell_i \in \N$.
        \begin{enumerate}
            \item There are at most $A$ $\cX_\Gamma$-polygons that form a Type A pair with $P$, where
                \[A = 2M(m-2)^2\left( \frac{m-1}{2} + 2M-5 + (2M-4)(m-3) + (2M-3)\frac{(m-3)^2}{4}\right).\]

            \item  There are at most $Be^{\ell_i}$ $\cX_\Gamma$-polygons $P'$ that form a Type B pair with $P$ and so that $d(P,P') \leq \ell_i$, where
                \[ B = \frac{32\pi}{3\sqrt{3}}\bigl(M(m^3-5m^2+8m-4)+m-2\bigr)\]

            \item There are at most $C_1e^{C_2\ell_i}$ $\cX_\Gamma$-polygons $P'$ that form a Type C pair with $P$ and so that $d(P,P') \leq \ell_i$, where
                \[ C_1 = (4M^2-6M)\bigl(M(m^3-5m^2+8m-4)+m-2\bigr),\]
            and
                \[ C_2 = \frac{\log(2M-3)}{D(y_0)},  \quad \textrm{ where } \quad D(y_0) = 4\sinh^{-1}\left(\frac{1}{4\sqrt{3} \sqrt{y_0^2+1}}\right).\]
            (Recall that $y_0$ is the height of the truncated blocks, as specified in \Cref{nota:blocks}.)

            \item  Let $\ell = (\ell_1, \ldots, \ell_j)$, be an itinerary type of a point $gx_0 \in B_{\cY_\Gamma}(r,x_0)$. Suppose $k$ entries in $\ell$ are equal to $\frac{1}{2}$. Let $F = \max\{B,C_1\}$. Then
            \[|\cP_j| \leq A^kF^{j-k} e^{3C_2r}.\]
        \end{enumerate}

    \end{thm}
    \begin{proof}
        The proofs of conclusions (1), (2), and (3) appear in \Cref{lem:TypeACount}, \Cref{cor:TypeBCount}, and \Cref{cor:TypeCCount}, respectively. We prove conclusion (4) holds.
        The recursive definition of $\cP_i$ together with conclusions (1)--(3) implies that if $\ell_i = \frac{1}{2}$, then $|\cP_{i}| \leq A|\cP_{i-1}|$, and if $\ell_i \geq 1$, then $|\cP_{i}| \leq |\cP_{i-1}|Fe^{C_2\ell_i}$.

        There are $k$ entries of $\ell$ equal to $\frac{1}{2}$, so $|\cP_{j}| \leq A^kF^{j-k}e^{C_2L}$, where $L = \sum_{\ell_i \geq 1} \ell_i$. It remains to bound~$L$:  \[L = \sum_{\ell_i \geq 1} \ell_i \leq (r - \frac{k}{2}) + (j-k) \leq r + j \leq 3r,\]
        where in the first inequality the term $r-\frac{k}{2}$ follows from the fact that each of the $\frac{1}{2}$ entries in $\ell$ comes from a path of length greater than or equal to $\frac{1}{2}$ by \Cref{lemma:hexside} and the $j-k$ term accounts for the rounding of the $\ell_i\geq1$. The last inequality follows from Corollary~\ref{cor:hexLength}.
    \end{proof}

    \subsection{Summing over possible itinerary types in a ball} \label{subsec:summing_over_types}

   The bound on the number of orbit points in a ball is obtained by summing over the possible itinerary types.

    \begin{prop}\label{upperboundprop}
        Let $r \in \N$. Using the constants from \Cref{thm:fixed_itinerary}, \[ |W_\Gamma x_0 \cap B_{\cY_\Gamma}(r,x_0)| \leq
            2M e^{3C_2r}(1+A+F)^{3r}.    \]
    \end{prop}
    \begin{proof}
        Suppose $gx_0 \in W_\Gamma x_0 \,\cap\, B_{\cY_\Gamma}(r,x_0)$. Then $g x_0$ is a vertex of a polygon in $\cP_{j}$ constructed with respect to an itinerary type $\ell = (\ell_1, \ldots, \ell_j)$ as in Definition~\ref{def:curlyP} and with $j \leq 2r$ by Lemma~\ref{cor:hexLength}. Thus, an upper bound on the cardinality of the set $W_{\Gamma}x_0 \cap B_{\cY_\Gamma}(r,x_0)$ is obtained via the sum of the cardinality of the sets $|\cP_{j}|$ as the vector $\ell$, of length at most $2r$, varies.

        First, fix $j \leq 2r$. Let $\cP^j$ be the union of all sets $\cP_{j}$ constructed with respect to an itinerary vector of length $j$.
        Then $\cP^j$ includes the ultimate polygons in all polygon sequences of length equal to $j$. As the number, $k$, of entries in $\ell$ equal to $\frac{1}{2}$ varies and the position of these entries varies, \Cref{thm:fixed_itinerary} implies

        \begin{eqnarray}
           |\cP^j| &\leq& \sum_{k=0}^j \left( A^kF^{j-k} e^{3C_2r} \right) \binom{j}{k} \binom{r+j - \lfloor\frac{3k}{2}\rfloor -1}{j-k-1}  \label{eq:one} \\
           & \leq & \binom{r + j}{j}e^{3C_2r}\sum_{k=0}^j A^kF^{j-k} \binom{j}{k} \label{eq:two} \\
           & \leq & \binom{r+j}{j}e^{3C_2r}\left(A+F \right)^j \label{eq:three}
        \end{eqnarray}
       Indeed, to see that Equation~(\ref{eq:one}) holds, fix the number $k$ of entries in $\ell$ equal to $\frac{1}{2}$. Then, $ |\cP_{j}| \leq A^kF^{j-k} e^{3C_2r}$ by \Cref{thm:fixed_itinerary}. The term $\binom{j}{k}$ accounts for the places in $\ell$ that the $k$ entries equal to $\frac{1}{2}$ can appear. Finally, as in the proof of \Cref{thm:fixed_itinerary}, if there are $k$ entries equal to $\frac{1}{2}$, then the sum of the remaining terms $\ell_i$ that are not equal to $\frac{1}{2}$ satisfies $\sum_{\ell_i \geq 1} \ell_i \leq (r - \frac{k}{2}) + (j-k)$. The final term $\binom{r+j - \lfloor \frac{3k}{2}\rfloor -1}{j-k-1}$ accounts for how these numbers are distributed in the remaining $j-k$ spots. This count is given by the classic ``stars and bars'' computation:
       the ``stars'' in this computation are the possible ways to distribute the values of $\ell_i$. Since each of the $\ell_i$ terms, of which there are $j-k$, must be an integer greater than or equal to one, there are at most $r - \lfloor \frac{k}{2} \rfloor$ stars remaining to distribute in the $j - k$ ``buckets''. This yields $j - k - 1$ ``bars''.
       Equation~(\ref{eq:two}) follows since $\binom{r+ j-\lfloor\frac{3k}{2}\rfloor -1}{j-k-1} \leq \binom{r+j-\lfloor \frac{k}{2}\rfloor}{j} \leq \binom{r+j}{j}$. Finally, Equation~(\ref{eq:three}) follows from applying the binomial theorem.
        The desired count is now given by varying $j$, using that there are at most 2M vertices per polygon:
    \begin{eqnarray*}
        |W_\Gamma x_0 \cap B_{\cY_\Gamma}(r,x_0)| &\leq&
            2M\sum_{j=1}^{2r} |\cP^j| \\
            &\leq& 2M \sum_{j=1}^{2r} \binom{r+j}{j}e^{3C_2r}\left(A+F \right)^j  \\
            & \leq & 2M e^{3C_2r} \sum_{j=0}^{3r} \binom{3r}{j} \left(A+F \right)^j  \\
            & \leq & 2M  e^{3C_2r}(1+A+F)^{3r},
    \end{eqnarray*}
    where the last lines follow from applying the binomial theorem and that $j \leq 2r$.
    \end{proof}

\subsection{Hausdorff dimension bound} \label{subsec:Hausdim_bound}

    \begin{thm}\label{thm:HausDimBound}
        If $m \ge 4$, the Hausdorff dimension of $\partial \cY_\Gamma$, and hence the conformal dimension of
        the Bowditch boundary $\partial(W_\Gamma, \mathcal{P})$,
        is bounded from above by
        \[
        3 C_2+ 3\log(1+A+F),
        \]
        where the constants are defined in \Cref{thm:fixed_itinerary},
    \end{thm}
    \begin{proof}
        \Cref{upperboundprop}
        proves $|W_\Gamma x_0 \cap B_{\cY_\Gamma}(r,x_0)| \le 2M e^{3C_2r}(1+A+F)^{3r}$.
        Substituting into the Poincar\'e series with constant $s$ yields
        \[
            \sum_{g \in W_\Gamma} e^{-s d(x_0, gx_0)}
            \leq \sum_{r = 1}^\infty |W_{\Gamma}x_0 \cap B_{\cY_{\Gamma}}(r,x_0)| e^{-sr}
            \le \sum_{r=1}^\infty 2M e^{3C_2r}(1+A+F)^{3r}e^{-sr}.
        \]

        By the root test,
        this series converges provided
        \[
        e^{-s +3C_2} \cdot (1+ A + F)^3 < 1.
        \]
        This holds when $s$ satisfies the inequality
        \[
        3C_2 + 3\log(1+ A+ F) < s.
        \]

       The stated bound on Hausdorff dimension follows by \Cref{thm:Paulin}, since
       $$\text{Confdim}(\partial(W_\Gamma, \cP))  = \text{Confdim}(\partial \cY_\Gamma) \leq \text{Hdim}(\partial \cY_\Gamma) = \text{Hdim}(\Lambda_c(W_\Gamma,\cY_\Gamma)) = \delta(W_\Gamma \curvearrowright \cY_\Gamma).$$ The second to last equality follows from the fact that $\partial \cY_\Gamma$ is a union of conical limit points of $W_\Gamma$ and countably many bounded parabolic points (since $W_\Gamma$ is hyperbolic relative to a finite collection of conjugacy classes of abelian subgroups). See discussion in \Cref{subsection:hausdim}.
    \end{proof}

    We note that in the statement of \Cref{thm:HausDimBound} the constant $C_2$ depends on $M = \max_{i \neq j}\{m_{ij}\}$ and the choice of $y_0$ in the construction of the model space $\cY_\Gamma$. The constants $A$ and $F$ depend on both $m$ and $M$.  In particular, as $y_0$ increases, the bound on conformal dimension in the above theorem increases. We now provide a specific upper bound in terms of the values $m$ and $M$ only.

    \begin{corollary} \label{cor:HausUpperBound}
        Let $W_\Gamma \in \cW$ be generated by $m \geq 4$ elements. Let $M = \max_{i \neq j}\{m_{ij}\}$.
        \begin{enumerate}
        \item If $M=3$, then
        \[ \Confdim(\p (W_\Gamma, \cP)) \leq 23 + 12\log m.  \]
        \item If $M \geq 4$, then
        \[ \Confdim(\p (W_\Gamma, \cP)) \leq  13 + 12 \log m + 19 \log M. \]
        \end{enumerate}
        \end{corollary}
    \begin{proof}
        Let $\cY_\Gamma$ be the $\CAT(-1)$ space constructed with $y_0 = 1.5$ as in \Cref{nota:blocks}. Then, $D(1.5)>0.3$ by \Cref{lemma:disjoint_prisms}. Hence, $3C_2 \leq 10 \log(2M-3)$. We now upper bound both $A$ and $F$. The highest degree term in $A$ is $m^4M^2$ after expanding. Since $m \geq 4$ and $M \geq 3$, one sees that $1+A \leq 3m^4M^2$. We recall that $F = \max\{B, C_1\}$. If $M = 3$, then $F = B$, and if $M \geq 4$, then $F = C_1$. Expanding the expressions for $B$ and $C_1$ and using that $m \geq 4$ and $M \geq 3$, one sees $B \leq 20m^3M$ and $C_1 \leq 4m^3M^3$. Therefore, if $M=3$, then
            \begin{eqnarray*}
                 \Confdim(\p (W_\Gamma, \cP)) &\leq& 10 \log 3 + 3\log(27m^4+60m^3)  \\
                 &\leq&  10 \log 3 + 3\log(42m^4) \\
                 &\leq& 10\log 3 + 3\log 42 + 12\log m
            \end{eqnarray*}
        If $M \geq 4$, then
                \begin{eqnarray*}
                    \Confdim(\p (W_\Gamma, \cP)) &\leq & 10 \log (2M-3) + 3\log(3m^4M^2+4m^3M^3) \\
                    &\leq& 10\log (2M-3) + 3\log(7m^4M^3) \\
                    & \leq & 10\log M + 10\log 2 + 3\log 7 + 12\log m + 9 \log M \\
                \end{eqnarray*}
        The statement of the corollary follows.
    \end{proof}

    \begin{remark} \label{rem:upperbounds}
        The upper bounds on the conformal dimension given in this section tend to infinity as the maximal edge label $M$ tends to infinity. On the other hand, we expect that, fixing $m$, the conformal dimension of the Bowditch boundary should {\it decrease} as $M$ tends to infinity. Intuitively, as $M$ increases, polygons in the Davis complex grow, and the group resembles a free group more and more. In this case, it takes longer to see the branching behavior that contributes to Hausdorff dimension. Indeed, this structure is exhibited in the upper bounds of Bourdon--Kleiner~\cite{bourdonkleiner15}, given in the Introduction. This discrepancy can be explained as follows.

        The upper bounds on the conformal dimension here are given by producing an upper bound on the Hausdorff dimension of the boundary of the $\CAT(-1)$ space $\cY_\Gamma$. We believe the Hausdorff dimension of $\cY_\Gamma$ does tend to infinity as $M \rightarrow \infty$ and $m$ is fixed. Indeed, the space $\cY_\Gamma$ is highly singular and constructed by attaching $3$-cells to uniformly bounded kites.
        Thus, we expect that trees with uniformly bounded edge lengths and increasing valence can be uniformly quasi-isometrically embedded into the model space $\cY_\Gamma$ as $M \rightarrow \infty$, which would imply the Hausdorff dimension of the boundary tends to infinity.

        This discussion motivates an understanding of other model geometries for groups in this family and stronger tools to produce upper bounds on conformal dimension in the relatively hyperbolic setting. The strategy employed by Bourdon--Kleiner in the hyperbolic case does not clearly translate to this setting.
        \end{remark}

\subsection{Geometric arguments and constant calculations} \label{subsec:GeoArg}

This section contains the geometric arguments used in \Cref{thm:fixed_itinerary}.

\subsubsection{Distance between disjoint polygons and prisms}

      First we bound the distance between disjoint $\cX_\Gamma$-polygons. We then define {\it prisms}, which deformation retract to $\cX_\Gamma$-polygons. We compute the distance between disjoint prisms in \Cref{lemma:disjoint_prisms}. This computation will be used to translate distances in $\cY_\Gamma$ to combinatorial distances in the Davis complex. To find these bounds, we consider the following subspaces of $\cY_{\Gamma}$.

        A {\it wall} in the $\CAT(-1)$ space $\cY_{\Gamma}$ is the fixed point set of a conjugate of a standard generator of~$W_{\Gamma}$.
        If $W$ is a wall in $\cY_{\Gamma}$, then $W$ is convex and separates $\cY_{\Gamma}$ into more than one component by \cite[Corollary II.2.8]{bridsonhaefliger}.

    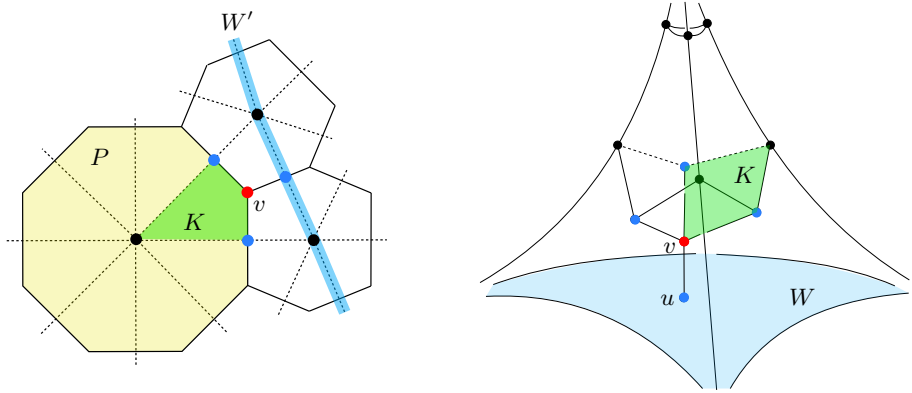
\begin{figure}
    \begin{centering}
	\begin{overpic}[width=.8\textwidth, tics=5]{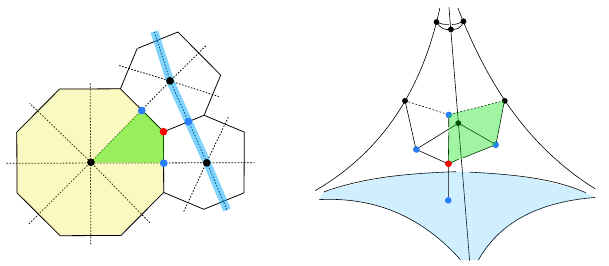}
        \put(27.5,20){\small{$v$}}
        \put(20,18){\small{$K$}}
        \put(10,25){\small{$P$}}
        \put(24,39.5){\small{$W'$}}
        \put(71.5,15.5){\small{$v$}}
        \put(79,23){\small{$K$}}
        \put(85,10){\small{$W$}}
        \put(71.2,10){\small{$u$}}
    \end{overpic}
	\caption{\small{Computing the distance between a $\cX_{\Gamma}$-polygon and a wall.}}
	\label{figure:polygon_distance}
    \end{centering}
    \end{figure}

    \begin{lemma} \label{lemma:disjoint_polygons}
        If $P$ and $Q$ are disjoint $\cX_\Gamma$-polygons, then $d_{\cY_\Gamma}(P,Q)\geq \frac{1}{2}$.
    \end{lemma}
    \begin{proof}
        Let $P$ and $Q$ be disjoint $\cX_{\Gamma}$-polygons. By construction, there is a wall of $\cY_{\Gamma}$ separating $P$ from $Q$. Let $W$ be a closest wall in $\cY_{\Gamma}$ to $P$. We will show $d_{\cY_{\Gamma}}(P, W) \geq \frac{1}{4}$. The result will follow.

There is a deformation retraction $\cY_{\Gamma} \rightarrow \cX_{\Gamma}$ by Proposition~\ref{prop:QIYAndCuspX}. The deformation retract $W'$ of $W$ to $\cX_\Gamma$ is combinatorially isomorphic to a subspace of $S X_\Gamma$, which corresponds to a wall in the Davis complex. The $\cX_\Gamma$-polygon $P$ is subdivided into kites. Let $v$ be the closest vertex of $P$ to $W$ and let $K \subset P$ be the kite containing $v$. Up to an isometry (as in \Cref{const:Ymhat}), we may assume  that $K$ is contained in an ideal tetrahedron $\cT$ contained in $\H^3$ and that $W \cap \cT$ is a face $T$ of $\cT$.  Viewing $\Hy^3$ in the unit ball model, we may assume the point $v$ lies on the disk in $\Hy^3$ with boundary the equator, $K$ lies in the northern hemisphere and $T$ lies in the southern hemisphere. See \Cref{figure:polygon_distance}.

       Let $u$ be the point in $T$ equidistant to the three sides of $T$. Since the geodesic in $\Hy^3$ from $v$ to $u$ meets both the disk through the equator and the triangle $T$ at right angles, $d_{\cY_\Gamma}(v,T) = d_{\Hy^3}(v,T) = d_{\Hy^3}(v,u)$. By construction, $d_{\Hy^3}(v,u) = d(x_0,x_i) = \frac{\log 2}{2}\geq \frac{1}{4}$ by Lemma~\ref{lemma:hexside}.

       We now show $d_{\cY_{\Gamma}}(v,T) = d_{\cY_{\Gamma}}(v,W)$. Let $\pi_W \colon \cY_{\Gamma} \to W$ be the nearest point projection map to the convex wall $W$. We claim that $\pi_W(v) = u$. Suppose not. Then there exists a geodesic triangle with vertex set $\{v,u,\pi_W(v)\}$. The Alexandrov angle $\angle_u(v, \pi_W(v)) = \pi/2$ by construction. Further, $\angle_{\pi_W(v)}(v,u) \geq \pi/2$ by~\cite[Proposition II.2.4]{bridsonhaefliger}. This contradicts that $\cY_\Gamma$ is $\CAT(-1)$.

    It remains to show that $d_{\cY_\Gamma}(P,W) = d_{\cY_\Gamma}(v,W)$. Suppose towards a contradiction that there exists $y\in P$ with $y \neq v$ so that $d_{\cY_\Gamma}(P,W) = d_{\cY_\Gamma}(y,W)$. We again argue using angles. Let $\pi_P:\cY_{\Gamma} \rightarrow P$ be the nearest point projection to the set $P$, which is convex by construction. Suppose first that $\pi_W(y) \neq \pi_W(p)$. Consider the quadrilateral on the ordered vertex set $\{v,y,\pi_W(y), u=\pi_W(v)\}$. Note the geodesic connecting $v$ and $y$ is contained in $P$, and the geodesic connecting $\pi_W(y)$ and $\pi_W(v)$ is contained in $W$. Further, $\pi_P(\pi_W(y)) = y$ since nearest point projections are unique in a $\CAT(0)$ space.  Consider the Alexandrov angles
        \[\alpha = \angle_v(u,y), \quad \beta=\angle_y(v, \pi_W(y)),
        \quad \gamma  =\angle_{\pi_W(y)}(u,y), \quad \delta=\angle_u(v,\pi_W(y)).\]
    Then, $\beta, \gamma, \delta \geq \frac{\pi}{2}$ by \cite[Proposition II.2.4]{bridsonhaefliger}. To see $\alpha \geq \frac{\pi}{2}$ note the geodesic from $v$ to $y$ begins in the kite $K$ and hence makes an obtuse angle with the geodesic from $v$ to $u$ in $\Hy^3$. These angles contradict that $\cY_\Gamma$ is $\CAT(-1)$. Otherwise, $\pi_W(y) = \pi_W(v)$, which also yields a contradiction since $\alpha \geq \frac{\pi}{2}$ and $\angle_y(v,u) \geq \frac{\pi}{2}$.
\end{proof}

        We will employ the following distance calculation from hyperbolic geometry.

    \begin{lemma}[\cite{katok}, Theorem 1.2.6(iii)] \label{rem:Katok}
        Let $z, w$ be two points in the upper half plane model $\Hy^2$, given as complex numbers. Let $\Ima(z)$ denote the imaginary part of $z$ and $|z|$ the norm of $z$. Then
            \[ \sinh\left( \frac{d_{\Hy^2}(z,w)}{2} \right) = \frac{|z-w|}{2\bigl(\Ima(z)\Ima(w)\bigr)^{1/2}}.  \]
    \end{lemma}

    \begin{defn}[Prisms]
    Let $P$ be an $\cX_\Gamma$-polygon.
    A \textbf{prism} $\P$ over $P$ will be a subset of $\cY_\Gamma$
    whose image under the deformation retraction to $\cX_\Gamma$ is $P$.
    In the full preimage of $P$ under this deformation retraction,
    after removing $P$ there is one component corresponding to each $3$-generator special subgroup
    of $W_\Gamma$.
    Such a component is unbounded with one ideal boundary point
    exactly when this $3$-generator special subgroup
    is the Euclidean $(3,3,3)$-triangle Coxeter group,
    and is bounded otherwise.
    We reserve the word ``prism'' for the closure of a bounded component.
    \end{defn}

    \begin{lemma}\label{lemma:disjoint_prisms}  If $\mathbb{P}$ and $\mathbb{P}'$ are disjoint prisms in $\cY_{\Gamma}$, then 
    \[d_{\cY_{\Gamma}}(\mathbb{P}, \mathbb{P}') \geq 4\sinh^{-1}\left(\frac{1}{4\sqrt{3} \sqrt{y_0^2+1}}\right) =: D(y_0),\] where $y_0$ is the height of the truncated blocks as specified in \Cref{nota:blocks}.
    \end{lemma}
    \begin{proof}
        Let $\P$ and $\P'$ be disjoint prisms corresponding to $\cX_\Gamma$-polygons $P$ and $P'$. Let $W$ be a closest wall to $\P$ in $\cY_\Gamma$.

        We continue the notation from the proof of Lemma~\ref{lemma:disjoint_polygons}; see Figure~\ref{figure:polygon_distance}. Specifically, let $v \in K \subset P$ and $u \in W$ be the points so that $d_{\cY_\Gamma}(P,W) = d_{\cY_\Gamma}(v,u)$. Then, up to isometry, $v$ and $u$ can be viewed as the points $x_0$ and $x_k$ in a kite as shown in Figure~\ref{figure:blockAndDavisHex}.
        Let $v_c$ and $u_c$ be the farthest points from $v$ and $u$ under the deformation retraction from $\cY_\Gamma$ to $\cX_\Gamma$.
        We will show that $d_{\cY_{\Gamma}}(\P,W) = d_{\cY_{\Gamma}}(v_c,u_c)$ and bound the latter.

        We first show that $d_{\cY_\Gamma}(v_c, W) = d_{\cY_\Gamma}(v_c,u_c)$. Suppose towards a contradiction $\pi_W(v_c) \neq u_c$. Consider the geodesic triangle with vertex set $\{v_c,u_c, \pi_W(v_c)\}$. Then, $\angle_{\pi_W(v_c)}(v_c,u_c) \geq \frac{\pi}{2}$ by \cite[Proposition II.2.4]{bridsonhaefliger}. Moreover, $\angle_{u_c}(v_c, \pi_W(v_c)) = \frac{\pi}{2}$ because the geodesic joining $v_c$ and $u_c$ projects to the geodesic joining $v$ and $u$, which makes an angle $\frac{\pi}{2}$ with the wall $W$. This is a contradiction, so  $d_{\cY_\Gamma}(v_c, W) = d_{\cY_\Gamma}(v_c,u_c)$.

        We next show $d_{\cY_\Gamma}(\P, W) = d_{\cY_\Gamma}(v_c,W)$, again using an angle argument. Suppose towards a contradiction there exists $p \in \P$ so that $p \neq v_c$ and $d_{\cY_\Gamma}(\P, W) = d_{\cY_{\Gamma}}(p,W)$. Suppose first that $\pi_W(p) \neq \pi_W(v_c)$ and consider the geodesic quadrilateral formed cyclically by $\{p,v_c, \pi_W(v_c)=u_c, \pi_W(p)\}$. Then $p = \pi_\P(\pi_W(p))$. Thus, the Alexandrov angles in this quadrilateral at $p, u_c$, and $\pi_W(p)$ are at least $\frac{\pi}{2}$. The Alexandrov angle $\angle_{v_c}(p,u_c) \geq \frac{\pi}{2}$ by construction, a contradiction. The case $\pi_W(v_c) = \pi_W(p)$ is similar. Thus, $d_{\cY_\Gamma}(\P, W) = d_{\cY_\Gamma}(v_c,W)$.

        It remains to bound $d_{\cY_\Gamma}(v_c,u_c)$. These points can be viewed in $\Hy^2 \subset \Hy^3$ via Notation~\ref{nota:blocks}.
        Lemma~\ref{rem:Katok} implies
            \[\sinh\left(\frac{d_{\H^2}(v_c, u_c)}{2}\right) = \frac{|v_c-u_c|}{2 \sqrt{\Ima(v_c)\Ima(u_c)}}. \]
                    In order to find a lower bound on $d_{\cY_{\Gamma}}(v', w') = d_{\H^2}(v',w')$, we have
        \begin{itemize}
            \item $|v_c-u_c| \geq \Rea(v_c-u_c) = \Rea(x_0-x_k)= \frac{1}{2\sqrt{3}}$, as seen using the isometry to the upper half space model as in Notation~\ref{nota:blocks}.
            \item $\Ima(v_c)$ and $\Ima(u_c)$ are bounded from above by the height of the corresponding cap in the upper half space model.
        \end{itemize}
        Therefore, using Lemma~\ref{lem:height_of_cap}, we get that
        \[\sinh\left(\frac{d_{\H^2}(v_c, u_c)}{2}\right) \geq \frac{1}{4\sqrt{3} \sqrt{y_0^2+1}}. \]
        The concludes the proof of the lemma as $d_{\cY_\Gamma}(\P, \P') = 2 d_{\H^2}(v_c, u_c)$.
        \end{proof}

    \subsubsection{Type A pairs}

      The next lemma counts the size of the combinatorial sphere of radius one around a polygon and is used in counting all types of pairs.

    \begin{lemma} \label{lem:countS1}
        If $P$ is a $\cX_\Gamma$-polygon, then if $M = \max_{i\neq j} m_{ij}$,
        \[ |\cS_1(P)| \leq M(m^2-3m+2) =: S_1.\]
    \end{lemma}
    \begin{proof}
        Fix a polygon $P$ of $\cX_\Gamma$, which corresponds to a polygon in the Davis complex $X_\Gamma$.  Any polygon $P'\in \cS_1(P)$ intersects $P$ in either an edge or a single vertex, but not both. For each edge $e$ of $P$, there are $m-2$ polygons that share that edge (one for each choice of generator that does not include the two generators that determine the polygon $P$). Therefore, there are at most $2M(m-2)$ polygons which intersect $P$ in an edge. For each vertex of $P$, there are $m-2$ edges which are not edges of $P$. As each pair of these edges determines a unique polygon, there are $2M\binom{m-2}{2}$ polygons which meet $P$ at a single vertex. Thus, $\cS_1(P) = 2M(m-2)+ 2M\binom{m-2}{2} = M(m^2-3m+2)$.
    \end{proof}

    Let $P$ be a $\cX_\Gamma$-polygon. If $P,P'$ form a Type A pair, then $P' \in \cS_2(P)$, whose size is computed in the next combinatorial lemma.

\begin{lemma}[Type A count]\label{lem:TypeACount}
        If $P$ is a $\cX_\Gamma$-polygon, then
        \[A:= |\cS_2(P)| \leq 2M(m-2)^2\left( \frac{m-1}{2} + (2M-5) + (2M-4)(m-3) + (2M-3)\frac{(m-3)^2}{4}\right),\]
        where $M = \max_{i \neq j}\{m_{ij}\}$. In particular, there are at most $A$ polygons that form a Type A pair with~$P$.
    \end{lemma}
    \begin{proof}
        Fix a polygon $P$ of $\cX_\Gamma$, which corresponds to a polygon in the Davis complex $X_\Gamma$. The polygons in $\cS_2(P)$ are disjoint from $P$ and intersect $\cS_1(P)$ in
        one of the following ways: \begin{enumerate}
            \item a path of combinatorial length 2,
            \item a single edge, or
            \item a single vertex.
        \end{enumerate}
        We will count polygons of each type, using the following notion. A vertex $v \in \cS_2(P) \cap \cS_1(P)$ is said to be an {\it inner vertex} if it is adjacent to $P$ and is an {\it outer vertex} otherwise.

        We first count the polygons in Case~(1). If a polygon $P'$ meets $\cS_1(P)$ in a path of combinatorial length~2, then the center vertex of this path is an inner vertex. There are $2m_{ij}(m-2)$ inner vertices. If $v$ is an inner vertex, then it is adjacent to $P$ via an edge. Each pair of the remaining $m-1$ edges incident to $v$ contribute a polygon to Case~(1). Thus, there are $2m_{ij}(m-2)\binom{m-1}{2}$ polygons that meet $\cS_1(P)$ in a path of combinatorial length 2.

        We next count the polygons in Case~(2). Since every pair of edges disjoint from $P$ that are adjacent to an inner vertex are contained in a polygon from Case~(1), if a polygon meets $\cS_1(P)$ in a single edge, it does so at an edge of $\cS_1(P)$ which is adjacent to two outer vertices.  We call such an edge an \textit{outer edge}; we count them and their contribution. Each of the $2m_{ij}(m-2)$ polygons intersecting $P$ in an edge contributes at most $2M-5$ outer edges, and each of the $2m_{ij}\binom{m-2}{2}$ polygons intersecting $P$ in a vertex contributes at most $2M-4$ outer edges. Each outer edge of a polygon in $\cS_1(P)$ is adjacent to $m-2$ polygons not in $\cS_1(P)$. So, there are at most $(m-2)\left(2m_{ij}(m-2)(2M-5) + 2m_{ij}\binom{m-2}{2}(2M-4)\right)$ polygons that intersect $\cS_1(P)$ in a single edge.

        Finally, we count the polygons in Case~(3). If a polygon meets $\cS_1(P)$ only in a single vertex, it does so at an outer vertex. Each  of the $2m_{ij}(m-2)$ polygons that intersect $P$ in an edge has at most $2M-4$ outer vertices, and each of the $2m_{ij}\binom{m-2}{2}$ polygons that intersect $P$ in a vertex has at most $2M-3$ outer vertices. Every outer vertex is incident to $\binom{m-2}{2}$ polygons which intersect $\cS_1(P)$ in exactly a vertex. Thus, the number of polygons that meet $\cS_1(P)$ in exactly one vertex is at most $\binom{m-2}{2}\left(2m_{ij}(m-2) (2M-4)  + 2m_{ij} \binom{m-2}{2} (2M-3)   \right)$.

        Adding up the counts in the three cases gives the desired count of $\cS_2(P)$.
    \end{proof}

    \subsubsection{Type B pairs}

    \begin{lemma} \label{lemma:flatcount}
        Fix an $\cX_\Gamma$-hexagon $H$ and an $\cX_\Gamma$-flat $\cF$ containing $H$. Let $\ell_i \in \N$. There are at most $b_0e^{\ell_i}$ $\cX_\Gamma$-hexagons $H'$ in $\cF$ satisfying $d_{\cY_\Gamma}(H,H') \leq \ell_i$ with $b_0 = \frac{32 \pi}{3 \sqrt{3}}$.
    \end{lemma}
    \begin{proof}
        We will translate the desired count to a count of hexagons in a tiling of the Euclidean plane as follows. Fix an $\cX_\Gamma$-hexagon $H$ and an $\cX_\Gamma$-flat $\cF$ containing $H$. The $\cX_\Gamma$-flat $\cF$ corresponds to three generators in a Euclidean triangle subgroup of $W_\Gamma$. Recall that in constructing the model space $\cY_\Gamma$ there is a collection $\cB$ of truncated blocks each of whose union of kites is glued to $\cF$. Further, there is a horoball $\cH$ glued to the boundary of the union of these blocks. The space $\cF \cup \cB \cup \cH$ can be identified with a subset $\Hy$ of $\Hy^3$. The vertices of $\cF$ are contained in a unique horosphere $\cS_{\cF}$  in $\Hy$. For each $\cX_\Gamma$-hexagon~$H$ in $\cF$ there is a corresponding regular Euclidean hexagon $H_{\mathbb{E}}$ in the horosphere $\cS_{\cF}$ with the same vertex set.
        Suppose the horosphere is at height $y$ in the upper halfspace model of $\H^3$ (under the identification described in Notation~\ref{nota:blocks}).
        If $z$ and $w$ are two adjacent vertices of $H$, then by \Cref{lemma:hexside}, $d_{\cY_\Gamma}(z,w) = \log 2$, and hence the diameter of $H$ is at most $3 \log 2$. By Lemma~\ref{rem:Katok}, $|z-w| = \frac{y}{\sqrt{2}}$, which is the Euclidean side length of $H_{\mathbb{E}}$. Thus, the Euclidean area of $H_{\mathbb{E}}$ equals $\frac{3y^2 \sqrt{3}}{4}$.

        Suppose $H' \in \cF$ with $d_{\cY_\Gamma}(H,H') \leq \ell_i$ as in the statement of the lemma. Suppose the geodesic in $\cY_\Gamma$ realizing this distance emanates from $\hat{p} \in H$ and terminates at $\hat{p}' \in H'$, points that may lie in the interior of these hexagons. Therefore, by the triangle inequality, there exist vertices of $H$ and $H'$ with distance in $\cY_\Gamma$ at most $\ell_i+\diam_{\cY_\Gamma}(H) \leq \ell_i+3\log2$. By \Cref{rem:Katok}, the Euclidean  distance in $\cS_{\cF}$ between these vertices, and thus the Euclidean distance between the hexagons $H$ and $H'$, is at most $2y \sinh(\frac{\ell_i+3\log2}{2})$, which is less than or equal to $y e^{\frac{\ell_i}{2}} 2 \sqrt{2}$. Thus, the number of Euclidean regular hexagons that can fit inside a ball of radius $y e^{\ell_i/2} 2\sqrt{2}$ is at most
        \[\frac{\pi(y e^{\frac{\ell_i}{2}} 2\sqrt{2})^2}{\frac{3y^2 \sqrt{3}}{4}} = \frac{32 \pi e^{\ell_i}}{3 \sqrt{3}},\]
        giving a bound on the number of $\cX_\Gamma$-hexagons in $\cF$ that are at distance (with respect to $\cY_\Gamma$) at most $\ell_i$ from $H$.
    \end{proof}

    \begin{corollary}[Type B count] \label{cor:TypeBCount}
        There are at most $Be^{\ell_i}$ polygons $P'$ that form a Type B pair with $P$ and so that $d(P,P') \leq \ell_i$, where
               \[B = \frac{32\pi}{3\sqrt{3}}M(m^3 - 5 m^2 + 8 m - 4).\]
    \end{corollary}
    \begin{proof}
        Let $P$ be an $\cX_\Gamma$-polygon. Suppose $P'$ forms a Type B pair with $P$. It follows from the definition of Type B pair that either
            \begin{enumerate}
                \item $P$ and $P'$ are in the same $\cX_\Gamma$-flat, or
                \item there exists an $\cX_\Gamma$-polygon $\bar{P}$ that intersects $P$, and $\bar{P}$ and $P'$ are in the same $\cX_\Gamma$-flat.
            \end{enumerate}
            Thus, there are at most $(S_1+1)(m-2)b_0e^{\ell_i}$ polygons $P'$ that form a Type B pair with $P$ and so that $d(P,P') \leq \ell_i$. Indeed, the $(S_1+1)$ term accounts for choosing either $P$ or a $\cX_\Gamma$-polygon $\bar{P}$ adjacent to $P$, where $S_1$ is computed in \Cref{lem:countS1}. The $m-2$ factor is to choose an $\cX_\Gamma$-flat or $\cX_\Gamma$-hyp plane containing either $P$ or $\bar{P}$. The $b_0e^{\ell_i}$ factor accounts for the number of $\cX_\Gamma$-polygons $P'$ in an $\cX_\Gamma$-flat with $d_{\cY_\Gamma}(P,P') \leq \ell_i$ or $d_{\cY_\Gamma}(\bar{P},P') \leq \ell_i$ by \Cref{lemma:flatcount}. The result follows from combining terms.
    \end{proof}

    \subsubsection{Type C pairs}

    We will use the next lemma on the size of combinatorial balls in the Davis-Moussong complex of a hyperbolic triangle group.

\begin{lemma}\label{lem:combinatorial ball size}
    Let $\Delta(p,q,r)$, $p \leq q \leq r$, be a hyperbolic triangle group with Davis-Moussong complex $X_{\Delta}$. Let $P$ be a polygon in $X_{\Delta}$. Then the size of the combinatorial ball of radius $k$ about $P$ in $X_{\Delta}$ is at most $2r(2r-3)^{k}$.
    \end{lemma}
    \begin{proof}
    It follows from induction that $|S_k(P)|\leq 2r(2r-3)^{k-1}$. Indeed, $|S_1(P)| = 2r$. For $k>1$, each polygon in $S_{k-1}(P)$ intersects $S_{k-2}(P)$ in either one or two edges (where $S_0(P) = P$), and shares at least three edges with polygons in $B_{k-1}(P)$. Thus, each polygon in $S_{k-1}(P)$ has at most $2r-3$ free edges to which a single polygon in $S_k(P)$ could be glued. The result follows. Therefore, since $r \geq 3$,
    \[|B_k(P)| \leq 1 + 2r \sum_{i=0}^{k-1} (2r-3)^i \leq 2r(2r-3)^k. \qedhere\]
    \end{proof}

\begin{lemma} \label{lemma:hyp plane count}
        Let $P$ be an $\cX_\Gamma$-polygon in an $\cX_\Gamma$-hyp plane $\cF$. Let $\ell_i \in \N$. There are at most $c_0e^{C_2\ell_i}$ $\cX_\Gamma$-polygons $P'$ in $\cF$ satisfying $d_{\cY_\Gamma}(P,P') \leq \ell_i$ with $c_0 = 2M(2M-3)$  and $C_2 = \frac{\log(2M-3)}{D(y_0)}$, where $M = \max_{i \neq j} m_{ij}$ and $D(y_0)$ is the constant from \Cref{lemma:disjoint_prisms}.
    \end{lemma}
    \begin{proof}
        Let $D = D(y_0)$ be the lower bound for distance between two disjoint prisms obtained above. Let $\mathbb{P}$ and $\mathbb{P}'$ be the prisms corresponding to the polygons $P$ and $P'$ as in the statement of the lemma. Then $d_{\cY_{\Gamma}}(\mathbb{P}, \mathbb{P}') \leq \ell_i$. Thus a geodesic in $\cY_{\Gamma}$ joining $P$ and $P'$ crosses at most $\ell_i/D$ prisms, each of which corresponds to a polygon in the $\cX_{\Gamma}$-hyp plane $\cF$. The plane $\cF$ corresponds to a hyperbolic triangle group $\Delta(p,q,r)$ for some $p \leq q \leq r$. By \Cref{lem:combinatorial ball size}, there are at most \[2r(2r-3)^{\lceil \ell_i/D \rceil} \leq
        2r(2r-3)^{(\ell_i/D) +1} =
        2r(2r-3)e^{(\ell_i/D)\log(2r-3)} \leq 2M(2M-3)e^{(\ell_i/D)\log(2M-3)} \] such polygons.
    \end{proof}

    \begin{corollary}[Type C count] \label{cor:TypeCCount}
        Let $P$ be an $\cX_\Gamma$-polygon. There are at most $C_1e^{C_2\ell_i}$ polygons $P'$ that form a Type C pair with $P$ and so that $d(P,P') \leq \ell_i$ with $C_1 = (4M^3-6M^2)(m^3 - 5 m^2 + 8 m - 4)$.
    \end{corollary}
    \begin{proof}
        The proof follows from \Cref{lemma:hyp plane count} and is analogous to that of \Cref{cor:TypeBCount}.
    \end{proof}

\section{Appendix} \label{sec:appendix}

In this section, we record the calculation to compute the constant $y_0$ in \Cref{prop:y_0}.

By \Cref{lemma_LaCAT1},  we want $2r \sigma_i > 2\pi$ for $i=1,3$ and $r \sigma_2 > 2\pi$. Since $r \geq 4$, we get
\[\sigma_{1,3} \geq \frac{\pi}{4} \quad \sigma_2 \geq \frac{\pi}{2}\]
By the Spherical Law of Cosines, we have
\[\cos(\sigma_{1,3}) = \cos^2(\theta_{1,3}) + \sin^2(\theta_{1,3}) \cos\left(\frac{\pi}{3}\right)\]
\[\cos(\sigma_2) = \cos^2(\theta_2) + \sin^2(\theta_2) \cos\left(\frac{2\pi}{3}\right)\]
and
\[\theta_1 = \pi - \arctan\left(\frac{2y_0}{\sqrt{3}}\right)\]
\[\theta_{2,3} = \pi- \arctan(2y_0)\]

Thus, for a fixed $i$, we can solve the above inequalities for $y_0$.

Solving for $i=1,3$:
Let $x = \frac{2y_0}{\sqrt{3}}$ for $i=1$ and $x=2y_0$ for $i=3$. Let $\arctan(x) = y$.
\begin{flalign*}
\sigma_{1,3} = \arccos\left(\cos^2\theta_{1,3} + \frac{1}{2}\sin^2\theta_{1,3}\right) & \geq \frac{\pi}{4} \\
\arccos\left(\cos^2 y + \frac{1}{2}\sin^2 y \right) & \geq \frac{\pi}{4} \quad \quad (\sin(\pi-\eta) = \sin(\eta))\\
\arccos\left(1-\frac{1}{2}\sin^2 y \right) & \geq \frac{\pi}{4} \quad \quad (\cos^2 \eta = 1- \sin^2 \eta)\\
\arccos\left(1-\frac{1-\cos 2y}{4}\right) & \geq \frac{\pi}{4} \quad \quad (\cos(2 \eta) = 1-2\sin^2 \eta)\\
\arccos\left(\frac{3}{4} + \frac{1}{4}\cos 2y \right) & \geq \frac{\pi}{4}\\
\arccos\left(\frac{3}{4} + \frac{1-\tan^2y}{4(1+\tan^2 y)}\right) & \geq \frac{\pi}{4} \quad \quad (\cos 2 \eta = \frac{1-\tan^2\eta}{1+\tan^2\eta})\\
\arccos\left(\frac{3}{4} + \frac{1-x^2}{4(1+x^2)}\right) & \geq \frac{\pi}{4} \quad \quad (\tan y = x)\\
\arccos\left( \frac{2+x^2}{2+2x^2} \right) \geq \frac{\pi}{4} \\
\frac{2+x^2}{2+2x^2} \leq \frac{1}{\sqrt{2}}\\
2^{1/4} \leq x
\end{flalign*}
Therefore for $i=1$, we have $2^{1/4} \leq \frac{2y_0}{\sqrt{3}}$ which gives $ 1.0298 \leq y_0$. For $i=3$, we get $2^{1/4} \leq 2y_0$, giving $0.5946 \leq y_0$. A similar calculation for $i=2$, yields $0.86 \leq y_0$. Therefore, $y_0=1.5$ satisfies \Cref{prop:y_0}.

\bibliographystyle{alpha}
\bibliography{refs.bib}

\bigskip

{\footnotesize
\noindent{}Elizabeth Field, \textsc{Division of Engineering \& Mathematics, University of Washington Bothell}\par\nopagebreak
\noindent{}\textit{E-mail address:} \texttt{ecfield@uw.edu}
}

\medskip

{\footnotesize
\noindent{}Radhika Gupta, \textsc{School of Mathematics, Tata Institute of Fundamental Research }\par\nopagebreak
\noindent{}\textit{E-mail address:} \texttt{radhikagupta.maths@gmail.com}
}

\medskip

{\footnotesize
\noindent{}Robert Alonzo Lyman, \textsc{Department of Mathematics and Computer Science, Rutgers University Newark}\par\nopagebreak
\noindent{}\textit{E-mail address:} \texttt{robbie.lyman@rutgers.edu}
}

\medskip

{\footnotesize
\noindent{}Emily Stark, \textsc{Department of Mathematics and Computer Science, Wesleyan University}\par\nopagebreak
\noindent{}\textit{E-mail address:} \texttt{estark@wesleyan.edu}
}

\end{document}